\documentclass[a4paper,12pt,reqno]{amsart}

\usepackage{amsmath,amssymb,amsthm}
\usepackage{latexsym}
\usepackage{color}
\usepackage{graphicx}
\usepackage{mathrsfs}
\usepackage{enumerate}
\usepackage[abbrev]{amsrefs}
\usepackage[T1]{fontenc}
\usepackage{mathtools}
\mathtoolsset{showonlyrefs=true}

\allowdisplaybreaks[4]

\theoremstyle{plain}
\newtheorem{thm}{Theorem}[section]
\newtheorem{lemm}[thm]{Lemma}
\newtheorem{prop}[thm]{Proposition}
\newtheorem{cor}[thm]{Corollary}

\theoremstyle{definition}
\newtheorem{df}[thm]{Definition}
\newtheorem{rem}[thm]{Remark}

\makeatletter

\@addtoreset{equation}{section}
\makeatother

\renewcommand{\div}{\operatorname{div}}
\newcommand{\dB}{\dot{B}}
\newcommand{\rank}{\operatorname{rank}}
\newcommand{\supp}{\operatorname{supp}}

\newcommand{\tb}{\widetilde{b}}
\newcommand{\tv}{\widetilde{v}}

\newcommand{\tf}{\widetilde{f}}
\newcommand{\tg}{\widetilde{g}}

\newcommand{\xih}{\xi_{\rm h}}
\renewcommand{\leq}{\leqslant}

\begin{document}
\title[The compressible Navier--Stokes--Coriolis system]
{Compressible Navier--Stokes--Coriolis system \\ in critical Besov spaces}
\author[M.~Fujii]{Mikihiro Fujii}
\address[M.~Fujii]{Graduate School of Science, Nagoya City University, Nagoya, 467-8501, Japan}
\email[Corresponding author]{fujii.mikihiro@nsc.nagoya-cu.ac.jp}
\author[K.~Watanabe]{Keiichi Watanabe}
\address[K.~Watanabe]{{School of General and Management Studies, Suwa University of Science, 5000-1, Toyohira, Chino, Nagano 391-0292, Japan}}
\email{{watanabe\_keiichi@rs.sus.ac.jp}}
\begin{abstract}
We consider the three{-}dimensional compressible Navier{--}Stokes system with the Coriolis force and prove the long-time existence of a unique strong solution.
More precisely, we show that
for any $0<T<\infty$ and
arbitrary large initial data in the scaling critical Besov spaces,
the solution uniquely exists on $[0,T]$ provided that the speed of rotation is high and the Mach numbers are low enough.
To the best of our knowledge, this paper is the first contribution to the well-posedness of the \textit{compressible} Navier{--}Stokes system with the Coriolis force in the whole space $\mathbb R^3$.
The key ingredient of our analysis is to establish the dispersive linear estimates despite a quite complicated structure of the linearized equation due to the anisotropy of the Coriolis force.
\end{abstract}
\maketitle

\section{Introduction}\label{sec:intro}
In this paper, we consider the Cauchy problem of the compressible Navier{--}Stokes system with the
Coriolis force on {the whole space} $\mathbb R^3$:
\begin{align}\label{eq:NSC-1}
	\begin{dcases}
		\partial_t \rho + \div ( \rho u ) = 0, 
		& 
		t > 0, x \in \mathbb{R}^3,\\
		\begin{aligned}
			\rho \left( \partial_t u + ( u \cdot \nabla ) u + \Omega ( e_3 \times u ) \right)
			&
			+
			\frac{1}{\varepsilon^2} \nabla P ( \rho ) \\
			={}
			&
			\mu \Delta u + ( \mu + \mu' ) \nabla \div u,
		\end{aligned}
		&
		t > 0, x \in \mathbb{R}^3,\\
		\rho(0,x) = \rho_0(x), \quad u(0,x) = u_0(x),
		&
		x \in \mathbb{R}^3.
	\end{dcases}
\end{align}
Here, {$\rho=\rho(t,x):(0,\infty)\times \mathbb{R}^3 \to (0,\infty)$} and {$u=u(t,x):(0,\infty)\times \mathbb{R}^3 \to \mathbb{R}^3$} stand for the unknown density and velocity of the fluid, respectively, while {$\rho_0=\rho_0(x):\mathbb{R}^3 \to (0,\infty)$} and {$u_0=u_0(x):\mathbb{R}^3 \to \mathbb{R}^3$} are given initial data.
The shear and bulk viscosity coefficients are constants and denoted by $\mu$ and $\mu'$, respectively{,} and we assume that $\mu$ and $\mu'$ {satisfy} the ellipticity conditions:
$\mu > 0$ and $\nu : = 2 \mu + \mu' > 0$.
The constants $\Omega \in \mathbb{R}$ and $\varepsilon > 0$ are the speed of rotation and Mach numbers, respectively,
{and the term $\Omega (e_3 \times (\rho u))$}, where $e_3 = (0,0,1)^\top$, represents the Coriolis force on the flow due to rotation of the fluid.
In this paper, we seek a solution $(\rho, u)$ around the constant equilibrium {state} $(\rho_\infty, 0)$, where $\rho_\infty$ is a positive constant.
The pressure $P$ is a suitably smooth function of $\rho$ satisfying $c_\infty := P' (\rho_\infty) > 0$.
The compressible Navier{--}Stokes{--}Coriolis system~\eqref{eq:NSC-1} is known as one of the models of geophysical flows and is physically well{-}justified at mid-latitude regions, where the centrifugal force is assumed to balance with the geostrophic force.
Concerning further physical backgrounds, we refer to~\cite{Cu-Ro-Be-11} and references therein.


The aim of this paper is to show the long-time existence of a unique solution in the scaling critical framework.
More precisely, for any {\it finite} time $0< T < \infty$ and arbitrarily large initial perturbations around the constant equilibrium state in the critical Besov spaces,
we prove that there exists a unique strong solution of~\eqref{eq:NSC-1} on $[0,T]$,
provided that the speed of rotation is high and the Mach number is low enough.
To the best of our knowledge, this paper is the first contribution to the analysis of the dispersive effect on \textit{strong} solutions of the compressible Navier{--}Stokes{--}Coriolis system on the whole space.


Before describing our results precisely, let us recall previous works related to our study briefly.
If the density $\rho$ is given as a positive constant (say $\rho \equiv 1$), System~\eqref{eq:NSC-1} reduces to the incompressible Navier{--}Stokes{--}Coriolis system:
\begin{align}\label{eq:incomp-coriolis}
	\begin{dcases}
		\partial_t u - \mu \Delta u + \Omega( e_3 \times u ) + ( u \cdot \nabla) u + \nabla p = 0, & \qquad t > 0, x \in \mathbb{R}^3,\\
		\div u = 0, & \qquad t \geqslant 0, x \in \mathbb{R}^3,\\
		u(0,x) = u_0(x), & \qquad x \in \mathbb{R}^3.
	\end{dcases}
\end{align}
Mathematical analysis for the incompressible rotating fluid was begun by Babin, Mahalov, and Nicolaenko \cites{Ba-Ma-Ni-97,Ba-Ma-Ni-00,Ba-Ma-Ni-01}.
According to their studies, the Coriolis term presents the dispersive phenomenon.
This is due to the fact that the linear solution of \eqref{eq:incomp-coriolis} is given by 
\begin{align}\label{lin:incomp}
	T_{\Omega}(t)u_0
	=
	e^{\mu t \Delta}
	e^{i\Omega t \frac{D_3}{|D|}}
	P_+u_0
	+
	e^{\mu t \Delta}
	e^{-i\Omega t \frac{D_3}{|D|}}
	P_-u_0,
\end{align}
where $P_{\pm}:=(1/2)(1\pm (iD/|D|)\times)$ (see \cite{Hi-Sh-10} for the derivation of \eqref{lin:incomp}).
Iwabuchi and Takada \cites{Iw-Ta-13,Iw-Ta-15} and Koh, Lee, and Takada \cite{Ko-Le-Ta-14-1} focused on the dispersive effect of the evolution group $\{e^{\pm i\Omega tD_3/|D|}\}_{t \in \mathbb{R}}$, generated by the Coriolis force, and obtained the Strichartz estimate for the linear solution:
\begin{align}\label{incomp:str}
	\| T_{\Omega}(t)u_0 \|_{L^{\theta}(0,\infty;L^p)}
	\leqslant
	C
	|\Omega|^{-\{\frac{1}{\theta}-(\frac{3}{4}-\frac{3}{2p})\}}
	\| u_0 \|_{L^2}
\end{align}
for $2<p<6$ and 
$3/4 - 3/(2p) \leqslant 1/\theta < \min \{ 1/2, 5/4 - 5/(2p) \}$.
We note that the power of $|\Omega|$ in \eqref{incomp:str} is negative, which implies the linear solution of \eqref{eq:incomp-coriolis} may be small for sufficiently large speed $|\Omega| \gg 1$ of the rotation.
Iwabuchi and Takada \cite{Iw-Ta-15} made use of this fact to the nonlinear problem and proved the long time solvability of~\eqref{eq:incomp-coriolis}.
More precisely, it was shown in \cite{Iw-Ta-15} that for any $u_0 \in \dot{H}^s(\mathbb R^3)^3$ ($1/2 < s < 5/4$) {with $\div u_0=0$} and $0 < T < \infty$, there exists a positive constant $\Omega_T=\Omega_T(s,u_0)$ such that \eqref{eq:incomp-coriolis} possesses a unique solution to \eqref{eq:incomp-coriolis} for all $\Omega \in \mathbb{R} \setminus \{0 \}$ with $|\Omega|\geqslant \Omega_T$.
Iwabuchi and Takada~\cite{Iw-Ta-13} and Koh, Lee, and Takada {\cite{Ko-Le-Ta-14-1}} improved the result of \cite{Iw-Ta-15} and proved the global well-posedness for large initial data. 
More specifically, it {was} shown in {\cites{Iw-Ta-13,Ko-Le-Ta-14-1}} that for \textit{arbitrarily} $u_0 \in \dot{H}^s(\mathbb R^3)^3$ ($1/2 \leqslant s < 9/10$) {with $\div u_0=0$}, there exists a positive constant $\Omega_0=\Omega_0(s,u_0)$ such that a unique global-in-time solution to \eqref{eq:incomp-coriolis} exists for all $\Omega \in \mathbb{R} \setminus \{0 \}$ with $|\Omega| \geqslant \Omega_0$.
In the aforementioned studies, to obtain the long time or global solvability results for \eqref{eq:incomp-coriolis}, the smallness condition on the initial data is replaced by the largeness condition on $\lvert \Omega \rvert$.
Concerning other related topics, see {Dutrifoy \cite{Du-05} and Koh, Lee, and Takada \cite{Ko-Le-Ta-14-2} for the long time existence of solutions to \eqref{eq:incomp-coriolis} with the inviscid case $\mu=0$, } \cites{Ch-De-Ga-Gr-02,OT} for the fast rotation limit results as $|\Omega| \to \infty$, \cites{Gi-In-Ma-Sa-08,Hi-Sh-10,Iw-Ta-14} for {the} uniform global well-posedness of \eqref{eq:incomp-coriolis}, and {\cites{Iw-Ta-12,Ko-Le-Ta-14-1,Iw-Ma-Ta-16} } for the time periodic problem.
We also refer to a series of papers by Charve \cites{Ch-04,Ch-05,Ch-06,Ch-08}, who investigated some dispersive properties arising in the primitive system with the Coriolis force. \par
In contrast to the incompressible case,
there are few results on the global well-posedness for the compressible Navier{--}Stokes{--}Coriolis system \eqref{eq:NSC-1}.
In the context of weak solutions, the existence and the fast rotation limit for \eqref{eq:NSC-1} {have} been studied by several authors.
Supposing that the Rossby number $\mathrm{Ro}=1/\Omega$ and Mach number $\mathrm{Ma}$ satisfy $\mathrm{Ro}=\mathrm{Ma}=\epsilon$ and that the complete slip boundary conditions on the boundary $\partial \mathcal D$, it was first shown in \cite{Fe-Ga-No-12} that the asymptotic dynamics of the compressible Navier{--}Stokes{--}Coriolis system in an infinite slab $\mathcal D = \mathbb R^2 \times (0, 1)$ as $\epsilon\to+0$ is described by a quasi-goestrophic type equation:
\begin{equation}
	\partial_t (\Delta_{{\mathrm{h}}} r - r) + \nabla_{{\mathrm{h}}}^\perp r \cdot \nabla_{{\mathrm{h}}} (\Delta_{{\mathrm{h}}} r) = \mu \Delta_{{\mathrm{h}}}^2 r
\end{equation}
in terms of global-in-time weak solutions, where we have set
\begin{equation}
	\nabla_{{\mathrm{h}}} r = (\partial_{x_1} r, \partial_{x_2} r, 0), \qquad \nabla_{{\mathrm{h}}}^\perp r = (\partial_{x_2} r, - \partial_{x_1} r, 0), \qquad \Delta_{{\mathrm{h}}} r = \partial_{x_1}^2 r + \partial_{x_2}^2 r.
\end{equation}
Notice that, the perfect slip boundary conditions on $\partial \mathcal D$ are crucial in their studies, since these avoid the effect of a boundary layer (Ekman layer);
{we also refer to \cite{Bo-Fa-Pr-22} for the case with the Ekman boundary layers.}
In a similar setting, further general multi-scale limits were considered in \cites{Fa-21,Fe-Ga-Ge-No-12,Fe-Ga-No-12,Fe-No-14}.
Here, general multi-scale limits mean that the Rossby and Mach numbers are assumed to be given by
\begin{equation}
	\mathrm{Ro} = \epsilon \qquad \text{and} \qquad \mathrm{Ma} = \epsilon^\theta, \qquad (\theta \geqslant 0).
\end{equation}
More precisely, the cases $\theta \gg 1$ and $\theta = 1$ were investigated by \cites{Fe-Ga-Ge-No-12,Fe-No-14} and \cite{Fe-Ga-No-12}, respectively. Recently, the remaining case $0 \leqslant \theta < 1$ was considered by Fanelli~\cite{Fa-21}.
According to their studies, if $\theta \ne 1$, then either the weak compressibility or the high rotation dominates the other, but this is not the case if $\theta = 1$.
Finally, we refer to the {work due to Ngo and Scrobogna \cite{Ng-Sc-18}}, who proved the long time existence of solutions to the compressible Euler-Coriolis system in $\mathbb{R}^3$, that is,~\eqref{eq:NSC-1} with $\mu = \lambda = 0$. \par
{
	We next focus on {the results for} the strong solutions of {the (non-rotational)} compressible Navier{--}Stokes system:
\begin{align}\label{eq:comp-NS}
	\begin{dcases}
		\partial_t \rho + \div ( \rho u ) = 0, 
		& 
		t > 0, x \in \mathbb{R}^3,\\
		\begin{aligned}
			\rho \left( \partial_t u + ( u \cdot \nabla ) u \right)
			&
			+
		      \nabla P ( \rho ) \\
			={}
			&
			\mu \Delta u + ( \mu + \mu' ) \nabla \div u,
		\end{aligned}
		&
		t > 0, x \in \mathbb{R}^3,\\
		\rho(0,x) = \rho_0(x), \quad u(0,x) = u_0(x),
		&
		x \in \mathbb{R}^3.
	\end{dcases}
\end{align}
    If $(\rho,u)$ is a solution to \eqref{eq:comp-NS}, then the pair of the scaled functions 
\begin{align}\label{scaling}
	\rho^{\lambda}(t,x):=\rho(\lambda^2 t, \lambda x),\qquad
	u^{\lambda}(t,x):=\lambda u(\lambda^2 t, \lambda x)
\end{align}
also solves \eqref{eq:comp-NS} for all $\lambda > 0$, up to a change of the pressure law $P$ into $P_{\ell}:=\ell^2 P$.
	Since it holds that 
	\begin{align}
		\| \rho^{\lambda}(0,\cdot)-\rho_{\infty} \|_{\dot{B}_{p,q}^{\frac{3}{p}}} = \| \rho(0,\cdot)-\rho_{\infty} \|_{\dot{B}_{p,q}^{\frac{3}{p}}}
		\quad \text{and} \quad 
		\| u^{\lambda}(0,\cdot) \|_{\dot{B}_{p,q}^{\frac{3}{p}-1}} = \| u(0,\cdot) \|_{\dot{B}_{p,q}^{\frac{3}{p}-1}}
	\end{align}
    for all dyadic numbers $\lambda > 0$, we see that
	the product of the homogeneous Besov spaces $\dot{B}_{p,q}^{\frac{3}{p}}(\mathbb{R}^3) \times \dot{B}_{p,q}^{\frac{3}{p}-1}(\mathbb{R}^3)^3$ is {a critical} space for the scaling \eqref{scaling}.
	It was Danchin who first proved the well-posedness in the critical space.
	In \cite{Da-00}, he proved the local well-posedness in $\dot{B}_{2,1}^{\frac{3}{2}}(\mathbb{R}^3) \times \dot{B}_{2,1}^{\frac{1}{2}}(\mathbb{R}^3)^3$ and the global well-posedness in $(\dot{B}_{2,1}^{\frac{1}{2}}(\mathbb{R}^3) \cap \dot{B}_{2,1}^{\frac{3}{2}}(\mathbb{R}^3)) \times \dot{B}_{2,1}^{\frac{1}{2}}(\mathbb{R}^3)^3$.
	The local well-posedness {result} was improved by {Chen, Miao, and Zhang} \cite{Ch-Mi-Zh-10} and Danchin \cite{Da-14}.  {Indeed, they} obtained the local well-posedness of \eqref{eq:comp-NS} in $\dot{B}_{p,1}^{\frac{3}{p}}(\mathbb{R}^3) \times \dot{B}_{p,1}^{\frac{3}{p}-1}(\mathbb{R}^3)^3$ for $1\leqslant p <6$. {Here,} the condition $1 \leqslant p <6$ is optimal since the system \eqref{eq:comp-NS} is ill-posed for the case $p\geqslant 6$, which was shown by {Chen, Miao, and Zhang}~\cite{Ch-Mi-Zh-15} and {Iwabuchi and Ogawa} \cite{Iw-Og-22}.
	By~\cites{Ch-Da-10,Ha-11}, the {global well-posedness} result due to {\cite{Da-00}} {was} extended to the $L^p${-}based Besov spaces in the high frequency part.
	In contrast to the well-posedness of~\eqref{eq:comp-NS}, there does not seem to exist any result for the well-posedness of the Navier{--}Stokes{--}Coriolis system~\eqref{eq:NSC-1}.
	The aim of this paper is to construct a unique long time solution in the scaling critical setting $(\dot{B}_{2,1}^{\frac{1}{2}}(\mathbb{R}^3) \cap \dot{B}_{2,1}^{\frac{3}{2}}(\mathbb{R}^3)) \times \dot{B}_{2,1}^{\frac{1}{2}}(\mathbb{R}^3)^3$.} 
    
    
    This paper aims to focus on the dispersive effect of the rotation and the acoustic waves arising in the compressible Navier{--}Stokes{--}Coriolis system \eqref{eq:NSC-1}.
    In particular, we study three{-}dimensional flows in the whole space and prove that some space-time norm of the linear solution may be sufficiently small for \textit{arbitrarily} large size of the initial data provided that the speed of rotation is high and the Mach number is small enough, which is developed by the Strichartz type estimate of the linear solution. As a by-product, we construct a unique \textit{long-time} solution in the critical $L^2$ framework with $1 \ll |\Omega| \leqslant 1/\varepsilon$.
    More precisely, we prove that for a given initial perturbation in $( \dB_{2,1}^{\frac{1}{2}}(\mathbb R^3) \cap \dB_{2,1}^{\frac{3}{2}}(\mathbb R^3) ) \times \dB_{2,1}^{\frac{1}{2}}(\mathbb R^3)^3$ and any {\it finite} time $0 < T < \infty$, there exists a large constant $\Omega_T \gg 1$, depending on the initial data and $T$, such that if $\Omega_T \leqslant |\Omega| \leqslant 1/\varepsilon$, then \eqref{eq:NSC-1} admits a unique solution on $[0,T]$.
    {Concerning relevant studies, the global existence of regular solutions to the two-dimensional compressible Navier{--}Stokes system for a class of large initial data
    follows provided that the Mach number is sufficiently small \cites{Fu-22,Ha-Lo-98} or that the volume viscosity coefficient $\mu'$ is sufficiently large \cites{Ch-Zh-19,Da-Mu-17}.
    In contrast to the aforementioned studies, to the best of the authors' knowledge, this paper is the first contribution to the existence of long time solutions to the compressible Navier{--}Stokes system \textit{with} the Coriolis force in the whole space $\mathbb R^3$ provided that the speed of rotation is high and the Mach number is small enough, where there is no restriction on the size of initial data.}    
  
\subsection*{Main results}
    {To} state our main result precisely, we reformulate the problem.
    In order to simplify the representation, we perform a suitable rescaling in \eqref{eq:NSC-1} so as to reduce to the case $\nu = \rho_\infty = c_\infty = 1$. More precisely, if $(\rho,u)$ is the solution to~\eqref{eq:NSC-1}, then the pair of the scaled functions
    \begin{align}
	    {\widetilde \rho (t, x) = \frac1{\rho_\infty} \rho \bigg(\frac{\nu t}{\rho_\infty c_{\infty} } , \frac{\nu x}{\rho_\infty \sqrt{c_{\infty}}}\bigg), \quad
	    \widetilde u (t, x) = \frac{1}{\sqrt{c_{\infty}}} u \bigg(\frac{\nu t}{\rho_\infty c_{\infty} } , \frac{\nu x}{\rho_\infty \sqrt{c_{\infty}}}\bigg)}
    \end{align}
    solves 
    \begin{align}\label{eq:NSC-1-scaled}
        \begin{dcases}
            \partial_t \widetilde \rho + \div ( \widetilde \rho \widetilde  u ) = 0, 
            & 
            t > 0, x \in \mathbb{R}^3,\\
            \begin{aligned}
                \widetilde \rho \left( \partial_t \widetilde u + ( \widetilde u \cdot \nabla ) \widetilde u + \widetilde  \Omega ( e_3 \times \widetilde  u ) \right)
                &
                +
                \frac{1}{\varepsilon^2} \nabla \widetilde  P ( \widetilde  \rho ) \\
                ={}
                &
                \widetilde \mu \Delta \widetilde  u + ( \widetilde  \mu + \widetilde{\mu'} ) \nabla \div \widetilde  u,
            \end{aligned}
            &
            t > 0, x \in \mathbb{R}^3,
        \end{dcases}
    \end{align}
    where 
    $\widetilde{\Omega}$,
    $\widetilde{\mu}$,
    $\widetilde{\mu'}$, and
    $\widetilde{P}(\widetilde{\rho})$
    are defined by
    \begin{align}
        \widetilde{\Omega}:=\frac{\nu}{\rho_{\infty}c_{\infty}} \Omega,\quad
        \widetilde{\mu}:=\frac{\mu}{\nu},\quad
        \widetilde{\mu'}:=\frac{\mu'}{\nu},\quad 
        \widetilde{P}(\widetilde{\rho}):= \frac{1}{\rho_{\infty}c_{\infty}}P(\rho_{\infty}\widetilde{\rho}).
    \end{align}
    Note that $\widetilde{\nu}:=2\widetilde{\mu}+\widetilde{\mu'}=1$ and $\widetilde{P}'(1)=1$.
    Hence, omitting the script $\sim$, we shall consider the problem \eqref{eq:NSC-1} with 
    {\[\nu=\rho_{\infty}=c_{\infty}=1.\]}  
    Let $a$ be the density perturbation defined by
    \begin{align}
        a(t,x) : = \frac{\rho(t,x) - 1}{\varepsilon}
    \end{align}
	with
	\begin{align}
		 a(0,x) = a_0(x) : = \frac{\rho_0(x) - 1}{\varepsilon}.
	\end{align}
    Then, there holds
    \begin{align}\label{eq:NSC-2}
        \begin{dcases}
            \partial_t a + \frac{1}{\varepsilon} \div u = - \div ( a u ), 
            & 
            t > 0, x \in \mathbb{R}^3,\\
            \partial_t u 
            -
            \mathcal{L} u
            + 
            \Omega ( e_3 \times u ) 
            +
            \frac{1}{\varepsilon} \nabla a 
            =
            - N_{\varepsilon}[a,u],
            &
            t > 0, x \in \mathbb{R}^3,\\
            a(0,x) = a_0(x), \quad u(0,x) = u_0(x),
            &
            x \in \mathbb{R}^3,
        \end{dcases}
    \end{align}
    where we have put $\mathcal{L}u : = \mu \Delta u + ( \mu + \mu') \nabla \div u$ with $\mu>0$ and $2\mu+\mu'=1$, and
    \begin{gather}
        N_{\varepsilon}[a,u]
        :=
        ( u \cdot \nabla ) u
        +
        J( \varepsilon a ) \mathcal{L}u
        +
        \frac{1}{\varepsilon}
        K( \varepsilon a )
        \nabla a,\\
        J(a)
        :=
        \frac{a}{1+a},\qquad
        K(a)
        :=
        \frac{P'(1+a)}{1+a}-1.
    \end{gather}
    Our goal is to prove the unique existence of the long time solution to \eqref{eq:NSC-2} for large initial data in the scaling critical Besov space{s} provided that $1 \ll |\Omega| \leqslant 1/\varepsilon$. 
    The following theorem is our main result.
    \begin{thm}\label{thm}
        Let $0 < T < \infty$ and 
        \begin{align}
            a_0 \in \dB_{2,1}^{\frac{1}{2}}(\mathbb{R}^3) \cap \dB_{2,1}^{\frac{3}{2}}(\mathbb{R}^3),\qquad 
            u_0 \in \dB_{2,1}^{\frac{1}{2}}(\mathbb{R}^3)^3.
        \end{align}
        Then, there exists a positive constant $\Omega_{T}$ depending on $T$, $\mu$, $P$, $a_0$, and $u_0$ such that 
        if $\Omega \in \mathbb R \setminus \{0 \}$ and $0 < \varepsilon \leqslant 1$ satisfy
        \begin{align}
        \label{cond-omega-epsilon}
            \Omega_T \leqslant |\Omega| \leqslant  \frac{1}{\varepsilon} ,
        \end{align}
        then
        the equation \eqref{eq:NSC-2} possesses a unique solution $(a,u)$ on {$[0,T]$} in the class 
        \begin{align}
            &
            a \in {C} ( [0,T] ; \dB_{2,1}^{\frac{1}{2}}(\mathbb{R}^3) \cap \dB_{2,1}^{\frac{3}{2}}(\mathbb{R}^3) ),\\
            &
            u \in {C} ( [0,T] ; \dB_{2,1}^{\frac{1}{2}} (\mathbb{R}^3) )^3 \cap L^1 ( 0,T ; \dB_{2,1}^{\frac{5}{2}}(\mathbb{R}^3) )^3
        \end{align}
        with $\rho = 1 +\varepsilon a >0$ on $[0,T] \times \mathbb{R}^3$.
    \end{thm}
    \begin{rem}
    We provide some remarks on Theorem \ref{thm}.
    \begin{enumerate}
        \item 
        The relation $|\Omega| \leq 1/\varepsilon$ assumed in \eqref{cond-omega-epsilon} is a technical one due to the low-frequency estimates for the linear solutions; see  \eqref{lin-str-visc-3} for instance.
        \item 
        In contrast to the local well-posedness theory by \cites{Da-01-L,Da-14},
        it is additionally required to assume $a_0 \in \dB_{2,1}^{\frac{1}{2}}(\mathbb{R}^3)$ since we need to obtain the global a priori estimate for the medium frequency part of the linear solutions; see Lemma \ref{lemm:middle-ene}.
    \end{enumerate}
    \end{rem}
    Let us make some comments on the proof of Theorem \ref{thm}. To construct the long time solution,
    the dispersive analysis due to the rotation and acoustic wave plays a crucial role.
    In the incompressible case, the explicit formula \eqref{lin:incomp} for the linear solution yields the Strichartz type estimate \eqref{incomp:str}, where the proof is based on the stationary phase method for the dispersive evolution group $\{e^{i\Omega t D_3/|D|} \}_{t \in \mathbb{R}}$ together with $TT^*$ argument.
    However, the situation becomes more complicated in the compressible case.
    Indeed, consider the linearized system
    \begin{align}\label{eq:lin-0}
    	\begin{dcases}
    		\partial_t a + \frac{1}{\varepsilon} \div u = f, 
    		& 
    		t > 0, x \in \mathbb{R}^3,\\
    		\partial_t u 
    		-
    		\mathcal{L} u
    		+ 
    		\Omega ( e_3 \times u ) 
    		+
    		\frac{1}{\varepsilon} \nabla a 
    		=
    		g,
    		&
    		t > 0, x \in \mathbb{R}^3,\\
    		a(0,x) = a_0(x), \quad u(0,x) = u_0(x),
    		&
    		x \in \mathbb{R}^3,
    	\end{dcases}
    \end{align}
    where {$a_0=a_0(x):\mathbb{R}^3 \to \mathbb{R}$} and {$u_0=u_0(x):\mathbb{R}^3 \to \mathbb{R}^3$} are given initial data and
    {$f=f(t,x):(0,\infty)\times \mathbb{R}^3 \to \mathbb{R}$} and {$g=g(t,x):(0,\infty)\times \mathbb{R}^3 \to \mathbb{R}^3$} are given external forces.
    Then the eigen polynomial for \eqref{eq:lin-0}
    is given by
    \begin{align}\label{egn-ply}
        \begin{split}
        \lambda^4
        &
        - (4\mu + \mu') |\xi|^2 \lambda^3
        + \left\{\frac{|\xi|^2}{\varepsilon^2} + \mu (5\mu + 2\mu') |\xi|^4 + \Omega^2 \right\} \lambda^2\\
        &
        - \left\{\frac{2\mu}{\varepsilon^2}|\xi|^4 + \mu^2 |\xi|^6 + \Omega^2\mu|\xi|^2 + \Omega^2(\mu + \mu') \xi_3^2 \right\} \lambda
        + \frac{\mu^2}{\varepsilon^2}|\xi|^6 + \frac{\Omega^2}{\varepsilon^2}\xi_3^2 = 0
        \end{split}
    \end{align}
    and thus
    it is difficult to give an explicit formula for the eigen frequencies {$\lambda$}. 
    In addition, it is also difficult to derive the asymptotic expansions of the eigen frequencies as $\lvert \xi \rvert \to 0$ and $\lvert \xi \rvert \to \infty$ due to the anisotropy for $\xi$.
    Hence, we may not obtain the explicit {dispersive} relation of the linear equation \eqref{eq:lin-0}, which yields a difficulty to prove the Strichartz type estimate for the solution to \eqref{eq:lin-0} just following the argument in the incompressible case.
    To overcome this difficulty, we shall follow the idea used in the studies of the low Mach number limit \cites{Ba-Ch-Da-11,Da-02-R,Da-He-16,Fu-22} and focus on the inviscid case $\mu=\mu'=0$.
    In fact, the eigen polynomial for the inviscid case $\mu=\mu'=0$ is so simple that the eigen frequencies may be represented explicitly.
    Namely, it is possible to obtain the dispersive estimate and the Strichartz {type} estimate for the \textit{inviscid} linear solution.
    Then, together with the energy estimate, we may obtain the Stricahrtz type estimate for the viscous linear solution and shrink the suitable space-time norm of the linear solution to \eqref{eq:lin-0} regardless of the size of initial data, provided that $\Omega$ and $\varepsilon$ satisfy the condition \eqref{cond-omega-epsilon} with a (possibly large) constant $\Omega_{T}$ depending on $T$.
    See Section \ref{sec:lin-str} for the precise argument. \par
    Finally, {let us} indicate why it is difficult to obtain the global well-posedness for System \eqref{eq:NSC-2}.
    It is well-known that the maximal regularity $u \in L^1(0,T;\dot{B}_{2,1}^{\frac{5}{2}}(\mathbb{R}^3))^{3}$ of the velocity field is a key ingredient for the global well-posedness of the compressible Navier{--}Stokes system.
    However, {it is not straightforward to extend this philosophy to the case of the compressible Navier{--}Stokes{--}Coriolis system~\eqref{eq:NSC-2}, since} it {seems to be} quite difficult to obtain the maximal $L^1(0,T;\dot{B}_{2,1}^{\frac{5}{2}})$-regularity for the low frequency part of the {linear velocity $u$} to \eqref{eq:lin-0}.
    {In general,} {it is known that} there are two methods to obtain the maximal {$L^1(0,T;\dot{B}_{2,1}^{\frac{5}{2}})$-}regularity estimates {in} the low frequency part of the solution {to the linearized system} {\eqref{eq:lin-0}}.
    The first one is to use the energy estimates.
    As it has been mentioned in \cites{Da-00,Ba-Ch-Da-11},
    the estimates for $\varepsilon \langle \nabla \Delta_j a, \Delta_j u\rangle_{L^2}$ play a crucial role in the energy method.
    Here, $\langle \, \cdot \, , \, \cdot \, \rangle_{L^2}$ stands for the $L^2$-inner product and $\Delta_j$ denotes the $j$th dyadic block
    of a Littlewood-Paley decomposition.
    To compute the time derivative of $\varepsilon \langle \nabla \Delta_j a, \Delta_j u\rangle_{L^2}$, we need to apply $\Delta_j$ to the second equation of~\eqref{eq:lin-0} and {taking inner product of it with $\Delta_j a$}.
    Then, from the Coriolis force, we meet the term $\Omega\varepsilon \langle e_3 \times \Delta_j u,\Delta_j a \rangle_{L^2}$, which provides a difficulty {in estimating the solution $u$} in the low frequency part $j \ll -1$ since the Coriolis force is anisotropic and {includes} no spacial derivative.
    The second {method to obtain the maximal {$L^1(0,T;\dot{B}_{2,1}^{\frac{5}{2}})$-}regularity estimate for the low frequency part} relies on the explicit formula for the linear semigroup corresponding to~\eqref{eq:lin-0}.
    If the real parts of all eigenvalues for~\eqref{eq:lin-0} are bounded by $-c|\xi|^2$ from below for some positive constant $c$ {\textit{independent} of $\Omega$ as well as $\varepsilon$}, then we may achieve our objective.
    However, as {we have mentioned} before, the eigenvalues {of the linear operator associated with \eqref{eq:lin-0}} are so complicated that it is quite hard to understand their properties.
    {Summarizing, at this moment}, it {seems to be} difficult to obtain the maximal $L^1(0,T;\dot{B}_{2,1}^{\frac{5}{2}})$-regularity for the {linear velocity} {$u$ to \eqref{eq:lin-0}},
    and thus the present paper only focuses on the long-time existence issue. 
\subsection*{Notation}
    We state some notations which are to be used throughout this paper.
    For $\alpha \in \mathbb R$, we set $\langle \alpha \rangle:= \sqrt{1+\alpha^2}$.
    We denote by $C$ the constant, which may differ in each line.
    In particular, $C=C(a_1,...,a_n)$ means that $C$ depends only on quantities $a_1,...,a_n$.
    {{We use} the notation $X \sim Y$ as an equivalent to $C^{- 1} X \le Y \le C X$.}
    We define a commutator for two operators $A$ and $B$ as $[A,B]=AB-BA$.
    
    Next, we mention the notations for function spaces.
    Let $\mathscr{S}(\mathbb{R}^3)$ be the set of all Schwartz functions on $\mathbb{R}^3$ and
    $\mathscr{S}'(\mathbb{R}^3)$ be the set of all tempered distributions on $\mathbb{R}^3$.
    For $f \in \mathscr{S}'(\mathbb{R}^3)$, we denote by $\mathscr{F}[f]=\widehat{f}$ and $\mathscr{F}^{-1}[f]$ the Fourier transform and the inverse transform of $f$, respectively.
    We recall the definition of the Besov spaces and the basic Littlewood-Paley theory.
    Let $\phi_0 \in \mathscr{S}(\mathbb{R}^3)$ satisfy
    \begin{align}
        0 \leq \phi_0 \leq 1, \qquad 
        \supp \phi_0 \subset \{\xi \in \mathbb{R}^3 \ ;\ 2^{-1} \leq |\xi| \leq 2 \},
    \end{align}
    and 
    \begin{align}
        \sum_{j \in\mathbb{Z}} \phi_j(\xi) = 1 \qquad {\rm for\ all\ }\xi \in \mathbb{R}^3\setminus \{0 \},
    \end{align}
    where $\phi_j(\xi):=\phi(2^{-j}\xi)$.
    Let $\{\Delta_j \}_{j \in \mathbb{Z}}$ be the Littlewood-Paley projection operators defined as 
    \begin{align}
        \Delta_j f := \mathscr{F}^{-1}\phi_j\mathscr{F}f.
    \end{align}
    Then, for $1 \leq p,\sigma \leq \infty$ and $s \in \mathbb{R}$, the Besov space $\dB_{p,\sigma}^s(\mathbb{R}^3)$ is defined by
    \begin{align}
        \dB_{p,\sigma}^s(\mathbb{R}^3)
        :=
        \left\{
        f \in \mathscr{S}'(\mathbb{R}^3)/\mathscr{P}(\mathbb{R}^3)
        \ ;\ 
        \| f \|_{\dB_{p,\sigma}^s} < \infty
        \right\},
    \end{align}
    where $\mathscr{P}(\mathbb{R}^3)$ denotes the set of all polynomials on $\mathbb{R}^3$ and
    \begin{align}
        \| f \|_{\dB_{p,\sigma}^s}
        :=
        \left\| 
        \{2^{sj} \| \Delta_j f \|_{L^p} \}_{j \in \mathbb{Z}} 
        \right\|_{\ell^{\sigma}(\mathbb{Z})}.
    \end{align}
    It is well-known that if $(s,\sigma) \in (-\infty, 3 \slash p) \times [1,\infty]$ or $(s,\sigma)=(3/p,1)$, 
    then $\dB_{p,\sigma}^s(\mathbb{R}^3)$ is identified by
    \begin{align}
        \left\{
        f \in \mathscr{S}'(\mathbb{R}^3)
        \ ;\ 
        f = \sum_{j \in \mathbb{Z}} \Delta_j f
        {\rm \ in\ }
        \mathscr{S}'(\mathbb{R}^3)
        {\rm \ and\ }
        \| f \|_{\dB_{p,\sigma}^s} < \infty
        \right\}.
    \end{align}
    We define the truncated Besov semi-norms by
    \begin{align}
        \| f \|_{\dB_{p,\sigma}^s}^{h;\beta}
        :={}&
        \left\| 
        \{2^{sj} \| \Delta_j f \|_{L^p} \}_{j} 
        \right\|_{\ell^{\sigma}(\{j \in \mathbb{Z}\ ; \ \beta < 2^j\})},\\
        \| f \|_{\dB_{p,\sigma}^s}^{m;\alpha,\beta}
        :={}&
        \left\| 
        \{2^{sj} \| \Delta_j f \|_{L^p} \}_{j} 
        \right\|_{\ell^{\sigma}( \{j \in \mathbb{Z}\ ;\ \alpha < 2^j \leq \beta \} )},\\
        \| f \|_{\dB_{p,\sigma}^s}^{\ell;\alpha}
        :={}&
        \left\| 
        \{2^{sj} \| \Delta_j f \|_{L^p} \}_{j} 
        \right\|_{\ell^{\sigma}(\{j \in \mathbb{Z} \ ;\ 2^j \leq \alpha \})}
    \end{align}
    for $0 \leq \alpha < \beta \leq \infty$.
    To control functions with the space-time variable, we use the Chemin-Lerner spaces.
    For $1 \leqslant p,\sigma,r \leqslant \infty$, $s \in \mathbb{R}$ and an interval $I \subset \mathbb R$, we define
    \begin{align}
        \widetilde{L^r}(I ; \dB_{p,\sigma}^s(\mathbb{R}^3))
        :={}&
        \left\{
        F : I \to \mathscr{S}'(\mathbb R^3)\ ;\ 
        \| F \|_{\widetilde{L^r}(I ; \dB_{p,\sigma}^s)} < \infty 
        \right\},\\
        \| F \|_{\widetilde{L^r}(I ; \dB_{p,\sigma}^s)}
        :={}&
        \left\| 
        \{2^{sj} \| \Delta_j F \|_{L^r( I ; L^p)} \}_{j \in \mathbb{Z}} 
        \right\|_{\ell^{\sigma}(\mathbb{Z})}.
    \end{align}
    We also define the truncated type norm as 
    {
    \begin{align}
        \| F \|_{{L^r}(I ; \dB_{p,\sigma}^s)}^{h;\beta}
        :={}&
        \left\| \| F \|_{\dB_{p,\sigma}^s}^{h;\beta} \right\|_{L^r(I)},\\
        \| F \|_{{L^r}(I ; \dB_{p,\sigma}^s)}^{m;\alpha,\beta}
        :={}&
        \left\| \| F \|_{\dB_{p,\sigma}^s}^{m;\alpha,\beta} \right\|_{L^r(I)},\\
        \| F \|_{{L^r}(I ; \dB_{p,\sigma}^s)}^{\ell;\alpha}
        :={}&
        \left\| \| F \|_{\dB_{p,\sigma}^s}^{\ell;\alpha} \right\|_{L^r(I)}.
    \end{align}
    }
    and 
    \begin{align}
        \| F \|_{\widetilde{L^r}(I ; \dB_{p,\sigma}^s)}^{h;\beta}
        :={}&
        \left\| 
        \{2^{sj} \| \Delta_j F \|_{L^r( I ; L^p)} \}_{j} 
        \right\|_{\ell^{\sigma}(\{j \in \mathbb{Z}\ ; \ \beta < 2^j\})},\\
        \| F \|_{\widetilde{L^r}(I ; \dB_{p,\sigma}^s)}^{m;\alpha,\beta}
        :={}&
        \left\| 
        \{2^{sj} \| \Delta_j F \|_{L^r( I ; L^p)} \}_{j} 
        \right\|_{\ell^{\sigma}( \{j \in \mathbb{Z}\ ;\ \alpha < 2^j \leq \beta \} )},\\
        \| F \|_{\widetilde{L^r}(I ; \dB_{p,\sigma}^s)}^{\ell;\alpha}
        :={}&
        \left\| 
        \{2^{sj} \| \Delta_j F \|_{L^r( I ; L^p)} \}_{j} 
        \right\|_{\ell^{\sigma}(\{j \in \mathbb{Z} \ ;\ 2^j \leq \alpha \})}.
    \end{align}
\subsection*{Organization of this paper}
    This paper is organized as follows.
    We focus on the linear analysis in Sections \ref{sec:lin-ene} and \ref{sec:lin-str}.
    In Section \ref{sec:lin-ene},
    we establish several linear energy estimates.
    The dispersive estimate and the Strichartz {type} estimate for the linear solution is proved in Section \ref{sec:lin-str}.
    In Section \ref{sec:LWP}, we prove the local well-posedness for~\eqref{eq:NSC-2}.
    Then, we prove our main result in Section \ref{sec:pf}.
\section{Linear energy estimates}\label{sec:lin-ene}
    In this section, we establish several energy estimates for the linearized system
    {
                \begin{align}\label{eq:lin}
                	\begin{dcases}
                		\partial_t a + \frac{1}{\varepsilon} \div u = f, 
                		& 
                		t > 0, x \in \mathbb{R}^3,\\
                		\partial_t u 
                		-
                		\mathcal{L} u
                		+ 
                		\Omega ( e_3 \times u ) 
                		+
                		\frac{1}{\varepsilon} \nabla a 
                		=
                		g,
                		&
                		t > 0, x \in \mathbb{R}^3,\\
                		a(0,x) = a_0(x), \quad u(0,x) = u_0(x),
                		&
                		x \in \mathbb{R}^3,
                	\end{dcases}
                \end{align}
                where {$a_0=a_0(x):\mathbb{R}^3 \to \mathbb{R}$} and {$u_0=u_0(x):\mathbb{R}^3 \to \mathbb{R}^3$} are given initial data and
    {$f=f(t,x):(0,\infty)\times \mathbb{R}^3 \to \mathbb{R}$} and {$g=g(t,x):(0,\infty)\times \mathbb{R}^3 \to \mathbb{R}^3$} are given external forces.
    }
    We begin with the following {basic} energy estimate{.}
    \begin{lemm}\label{lemm:simple-ene}
        There exists a positive constant $C=C(\mu)$ such that
        for each $\Omega \in \mathbb{R}$ and $\varepsilon>0$, the solution to {\eqref{eq:lin}} satisfies
        \begin{align}
            \| \Delta_j (a,u) \|_{L^{\infty}(0,t:L^2)}^2
            +
            \| \nabla \Delta_j u \|_{L^2( 0,t ; L^2 )}^2
            \leq 
            C \left( 
            \| \Delta_j(a_0,u_0) \|_{L^2}^2
            +
            \| \Delta_j(f,g) \|_{L^1( 0,t ; L^2)}^2
           \right)
        \end{align}
        for all $j \in \mathbb R$ and $t >0$, provided that the right{-}hand side is finite.
    \end{lemm}
    \begin{proof}
        Applying $\Delta_j$ to {\eqref{eq:lin}}, we have
        \begin{align}\label{eq:lin-j}
            \begin{dcases}
            \partial_t\Delta_j a + \frac{1}{\varepsilon} \div \Delta_j u = \Delta_j f , \\
            \partial_t\Delta_j u - \mathcal{L} \Delta_j u + \Omega ( e_3 \times \Delta_j u ) + \frac{1}{\varepsilon} \nabla \Delta_j a = \Delta_j g.
            \end{dcases}
        \end{align}
        Taking $L^2(\mathbb R^3)$-inner product of the first equation of \eqref{eq:lin-j} with $\Delta_j a$, 
        taking $L^2(\mathbb R^3)^3$-inner product of the second equation of \eqref{eq:lin-j} with $\Delta_j u$, and 
        summing up of them, we have 
        \begin{align}\label{simple-d}
            \frac{1}{2}
            \frac{d}{dt}
            \| \Delta_j (a,u) \|_{L^2}^2
            +
            \underline{\mu}
            \| \nabla \Delta_j u \|_{L^2}^2
            \leqslant
            \| \Delta_j(f,g) \|_{L^2}
            \| \Delta_j(a,u) \|_{L^2},
        \end{align}
        where $\underline{\mu}:=\min \{\mu,1\}$.
        Let us fix $t_0>0$ and let $0<t\leqslant t_0$. 
        Integrating {\eqref{simple-d}} with respect to time interval on {$[0,t]$}, we have
        \begin{align}
            &\frac{1}{2} \| \Delta_j (a,u) (t) \|_{L^2}^2
            +
            {\underline{\mu}}
            {\| \nabla \Delta_j u \|_{L^2(0,t;L^2)}^2}\\
            &\quad\leqslant
            \frac{1}{2} \| \Delta_j (a_0,u_0) \|_{L^2}^2
            +
            {\bigg\|}
            \| \Delta_j(f,g) \|_{L^2}
            \| \Delta_j(a,u) \|_{L^2}
            {\bigg\|_{L^1(0,t)}}\\
            &\quad\leqslant
            \frac{1}{2} \| \Delta_j (a_0,u_0) \|_{L^2}^2
            +
            \| \Delta_j(f,g) \|_{L^1(0,t_0;L^2)}
            \| \Delta_j(a,u) \|_{L^{\infty}( 0,{t};L^2)}\\
            &\quad\leqslant
            \frac{1}{2} \| \Delta_j (a_0,u_0) \|_{L^2}^2
            +
            {2}
            \| \Delta_j(f,g) \|_{L^1(0,t_0;L^2)}^2
            +
            \frac{1}{8}
            \| \Delta_j(a,u) \|_{L^{\infty}( 0,{t};L^2)}^2.
        \end{align}
        Hence, taking $L^{\infty}(0,t_0)$ norm with respect to $t$, we complete the proof.
    \end{proof}
    To obtain the maximal regularity estimates of the linear solution, we {follow} the argument in \cites{Ba-Ch-Da-11,Da-01-G} {to} obtain the following lemma{.}
    \begin{lemm}\label{lemm:middle-ene}
        For {$\beta,\mu>0$}, there exists a positive constant $C=C(\beta,\mu)$ such that 
        the solution $(a,u)$ of {\eqref{eq:lin}} satisfies 
        \begin{align}
            \| (a,u) \|_{\widetilde{L^{\infty}}(0,t ; \dB_{2,\sigma}^s ) \cap \widetilde{L^1}(0,t ; \dB_{2,\sigma}^s )}^{m;\alpha,\frac{\beta}{\varepsilon}}
            \leqslant
            C
            \left( 
            \| (a_0,u_0) \|_{\dB_{2,\sigma}^s}^{m;\alpha,\frac{\beta}{\varepsilon}}
            +
            \| (f,g) \|_{{\widetilde{L^1}}( 0,t ; \dB_{2,\sigma}^s)}^{m;\alpha,\frac{\beta}{\varepsilon}}
            \right)
        \end{align}
        for all $\Omega, \alpha \in \mathbb{R}$, $\varepsilon>0$ with $|\Omega|\varepsilon \leqslant \alpha < \beta/\varepsilon$, $s \in \mathbb R$, $1 \leqslant \sigma \leqslant \infty$, and $t>0$.
    \end{lemm}
    \begin{proof}
        Applying ${\varepsilon \nabla \Delta_j}$ to the first equation of \eqref{eq:lin-j} and taking $L^2(\mathbb R^3)$-inner product of it with $\Delta_j u$, we have
        \begin{align}\label{au-j-1}
            \langle \partial_t \varepsilon \nabla  \Delta_j a , \Delta_j u \rangle_{L^2}
            -
            \| \div \Delta_j u \|_{L^2}^2
            =
            \langle \varepsilon \nabla \Delta_j f , \Delta_j u \rangle_{L^2}.
        \end{align}
        Applying $\Delta_j$ to \eqref{eq:lin-j} and
        taking $L^2(\mathbb R^3)^3$-inner product of the second equation of~\eqref{eq:lin-j} with $\varepsilon \nabla \Delta_j a$, we see that
        \begin{align}\label{au-j-2}
            \begin{split}
            &\langle \varepsilon \nabla  \Delta_j a , \partial_t \Delta_ju \rangle_{L^2}
            +
            \langle  \varepsilon \nabla  \Delta_j a ,  \Delta \Delta_ju \rangle_{L^2}\\
            &\quad
            +
            \Omega\varepsilon
            \langle \nabla \Delta_j a, e_3 \times \Delta_j u \rangle_{L^2}
            +
            \| \nabla \Delta_j a \|_{L^2}^2\\
            &\quad=
            \langle \varepsilon \nabla \Delta_j a , \Delta_j g \rangle_{L^2}
            .
            \end{split}
        \end{align}
        Gathering \eqref{au-j-1} and \eqref{au-j-2}, we obtain 
        \begin{align}\label{au-j-3}
            \begin{split}
            &\partial_t 
            {\rm Re}
            \langle  \varepsilon \nabla  \Delta_j a , \Delta_ju \rangle_{L^2}
            +
            \| \nabla \Delta_j a \|_{L^2}^2\\
            &\quad
            \leqslant{}
            C\varepsilon2^{j}
            \| \Delta_j(f,g) \|_{L^2}
            \| \Delta_j(a,u) \|_{L^2}
            +
            C2^{2j}\| \Delta_ju \|_{L^2}^2\\
            &\qquad
            +
            C\varepsilon2^{3j}\| \Delta_j a \|_{L^2}\| \Delta_j u \|_{L^2}
            +
            C|\Omega|\varepsilon2^{j}\| \Delta_j a \|_{L^2}\| \Delta_j u \|_{L^2}.
            \end{split}
        \end{align}
        Let $0< \delta < 1$ be a constant to be determined later and set 
        \begin{align}
            {
                    V_j(t) : = 
                    \left(
                    \| \Delta_j(a,u) (t)\|_{L^2}^2 
                    + 
                    2\delta 
                    {\rm Re}
                    \langle  \varepsilon \nabla  \Delta_j a(t) , \Delta_ju(t) \rangle_{L^2}
                    \right)^{\frac{1}{2}}
            }
            .
        \end{align}
        We note that 
        \begin{align}
            \left|
            V_j^2 - \| \Delta_j(a,u) \|_{L^2}^2
            \right|
            \leqslant{}
            C\delta\varepsilon2^j\| \Delta_ja \|_{L^2}\| \Delta_ju \|_{L^2}
            \leqslant
            C_*\delta\| \Delta_j(a,u) \|_{L^2}^2
        \end{align}
        for some positive constant $C_*{=C_*(\beta)}$.
        {By \eqref{simple-d} and \eqref{au-j-3},} 
        for every $j \in \mathbb{Z}$ with $|\Omega|\varepsilon \leqslant 2^j \leqslant \beta/\varepsilon$ we find that
        \begin{align}
            &
            \frac{1}{2}\frac{d}{dt}V_j^2 
            +
            \underline{\mu}
            \| \nabla \Delta_j u \|_{L^2}^2
            +
            \delta 
            \| \nabla \Delta_j a \|_{L^2}^2\\
            &\quad 
            \leqslant{}
            C
            \| \Delta_j(f,g) \|_{L^2}
            \| \Delta_j(a,u) \|_{L^2}
            +
            C\delta 
            2^{2j}\| \Delta_ju \|_{L^2}^2
            +
            C\delta2^{2j}\| \Delta_j a \|_{L^2}\| \Delta_j u \|_{L^2}\\
            &\quad 
            \leqslant{}
            C
            \| \Delta_j(f,g) \|_{L^2}
            \| \Delta_j(a,u) \|_{L^2}
            +
            C^*\delta 
            \| \nabla \Delta_ju \|_{L^2}^2
            +
            \frac{1}{2}\delta\| \nabla \Delta_j a \|_{L^2}^2
        \end{align}
        for some constant $C^*\geqslant 1$.
        Choosing $\delta$ such that $\delta \leqslant {\min\{} \underline{\mu}/(2C^*), 1/(2C_*)\}$, 
        we obtain 
        \begin{gather}
            \frac{1}{2}\frac{d}{dt}V_j^2 
            +
            {\frac{\delta}{2}} 
            \| \nabla \Delta_j (a,u) \|_{L^2}^2
            \leqslant{}
            C
            \| \Delta_j(f,g) \|_{L^2}
            \| \Delta_j(a,u) \|_{L^2},\\
            \frac{1}{2}
            \| \Delta_j(a,u) \|_{L^2}^2 
            \leqslant
            V_j^2
            \leqslant 
            \frac{3}{2}
            \| \Delta_j(a,u) \|_{L^2}^2.\label{Vj}
        \end{gather}
        Thus, we see that 
        \begin{align}
            \frac{1}{2}\frac{d}{dt}V_j^2 
            +
            c
            {2^{2j}}V_j^2
            \leqslant{}
            C
            \| \Delta_j(f,g) \|_{L^2}
            V_j,
        \end{align}
        which implies 
        \begin{align}
            \sup_{0 \leqslant \tau \leqslant t }V_j(\tau)
            +
            2^{2j}
            \int_0^t
            V_j(\tau)
            d\tau 
            \leqslant 
            CV_j(0)
            +
            C\int_0^t
            \| \Delta_j(f,g)(\tau) \|_{L^2}
            d\tau.
        \end{align}
        Using \eqref{Vj} and taking $2^{sj}$-weighted $\ell^{\sigma}$-norm with respect to $j$ with $|\Omega| \varepsilon \leqslant 2^j \leqslant \beta/\varepsilon$ yield the desired result.
    \end{proof}
    Combining Lemmas \ref{lemm:simple-ene} and \ref{lemm:middle-ene}, we obtain the maximal regularity estimate {for the low frequency part}.
    \begin{cor}\label{cor:low-ene}
        {Let} {$\beta,\mu>0$}{.} There exists a positive constant $C=C(\beta,\mu)$ such that 
        the solution $(a,u)$ of {\eqref{eq:lin}} satisfies 
        \begin{align}\label{low-max-reg}
            \| (a,u) \|_{\widetilde{L^{\infty}}(0,t ; \dB_{2,\sigma}^s ) \cap \widetilde{L^1}(0,t ; \dB_{2,\sigma}^{s+2} )}^{\ell;\frac{\beta}{\varepsilon}}
            \leqslant
            C
            \langle {\Omega^2\varepsilon^2}T\rangle
            \left( 
            \| (a_0,u_0) \|_{\dB_{2,\sigma}^s}^{\ell;\frac{\beta}{\varepsilon}}
            +
            \| (f,g) \|_{{\widetilde{L^1}}( 0,t ; \dB_{2,\sigma}^s)}^{\ell;\frac{\beta}{\varepsilon}}
            \right)
        \end{align}
        for all $\Omega\in \mathbb{R}$, $\varepsilon>0$ with {$|\Omega|\varepsilon  < \beta/\varepsilon$}, $s \in \mathbb R$, $1 \leqslant \sigma \leqslant \infty$, and $0<t\leqslant T<\infty$.
        \end{cor}
        \begin{rem}
            Lemma \ref{lemm:middle-ene} does not allow us to take $\alpha$ as $\alpha=0$ unless $\Omega = 0$,
            {and thus the solution $(a, u)$ in the low frequency part $|\xi| \leqslant |\Omega|\varepsilon$ may be controlled solely by Lemma \ref{lemm:simple-ene}, according to our analysis.} 
            {Notice that} Lemma \ref{lemm:simple-ene} does not provide the global{-}$L^1${-}in{-}time estimate of the linear solution, which {forces} us {to} assume $T<\infty$.
        \end{rem}
        \begin{proof}[{Proof of Corollary \ref{cor:low-ene}}]
        Let us decompose the left{-}hand side as 
        \begin{align}
            \| (a,u) \|_{\widetilde{L^{\infty}}(0,t ; \dB_{2,\sigma}^s ) \cap \widetilde{L^1}(0,t ; \dB_{2,\sigma}^{s+2} )}^{\ell;\frac{\beta}{\varepsilon}}
            \leqslant{}&
            \| (a,u) \|_{\widetilde{L^{\infty}}(0,t ; \dB_{2,\sigma}^s ) \cap \widetilde{L^1}(0,t ; \dB_{2,\sigma}^s )}^{{\ell;|\Omega|\varepsilon}}\\
            &
            +
            \| (a,u) \|_{\widetilde{L^{\infty}}(0,t ; \dB_{2,\sigma}^s ) \cap {\widetilde{L^1}(0,t;}\dB_{2,\sigma}^{s+2} )}^{{m;|\Omega|\varepsilon,\frac{\beta}{\varepsilon}}}
            {.}
        \end{align}
        It holds by $T<\infty$ and Lemma \ref{lemm:simple-ene} that
        \begin{align}
            \| (a,u) \|_{\widetilde{L^{\infty}}(0,t ; \dB_{2,\sigma}^s ) \cap \widetilde{L^1}(0,t ; \dB_{2,\sigma}^{s+2} )}^{{\ell;|\Omega|\varepsilon}}
            \leqslant{}&
            C\langle {\Omega^2\varepsilon^2}T \rangle 
            \| (a,u) \|_{\widetilde{L^{\infty}}(0,t ; \dB_{2,\sigma}^s ) }^{{\ell;|\Omega|\varepsilon}}\\
            \leqslant{}&
            C
            \langle {\Omega^2\varepsilon^2}T \rangle
            \left( 
            \| (a_0,u_0) \|_{\dB_{2,\sigma}^s}^{{\ell;|\Omega|\varepsilon}}
            +
            \| (f,g) \|_{{\widetilde{L^1}}( 0,t ; \dB_{2,\sigma}^s)}^{{\ell;|\Omega|\varepsilon}}
            \right).
        \end{align}
        It follows from Lemma \ref{lemm:middle-ene} that
        \begin{align}
            \| (a,u) \|_{\widetilde{L^{\infty}}(0,t ; \dB_{2,\sigma}^s ) \cap \widetilde{L^1}(0,t ; \dB_{2,\sigma}^s )}^{{m;|\Omega|\varepsilon,\frac{\beta}{\varepsilon}}}
            \leqslant
            C
            \left( 
            \| (a_0,u_0) \|_{\dB_{2,\sigma}^s}^{{m;|\Omega|\varepsilon,\frac{\beta}{\varepsilon}}}
            +
            \| (f,g) \|_{{\widetilde{L^1}}( 0,t ; \dB_{2,\sigma}^s)}^{{m;|\Omega|\varepsilon,\frac{\beta}{\varepsilon}}}
            \right).
        \end{align}
        Gathering the above three estimates, we complete the proof.
    \end{proof}
    We may easily obtain the high frequency estimate for the linear solution by the standard effective velocity method (see \cite{Ha-11} for the detail).
    \begin{lemm}\label{lemm:high-ene}
        There exist positive constants $\beta_0=\beta_0(\mu)$ and $C=C(\mu)$ such that
        for each $\Omega \in \mathbb{R}$ and $\varepsilon > 0$, the solution to {\eqref{eq:lin}} satisfies
        \begin{align}
            &
            \varepsilon \| a \|_{\widetilde{L^{\infty}}( 0,t ; \dB_{p,1}^{s+1} )}^{h;\frac{\beta}{\varepsilon}}
            +
            \frac{1}{\varepsilon}\| a \|_{L^1( 0,t ; \dB_{p,1}^{s+1} )}^{h;\frac{\beta}{\varepsilon}}
            +
            \| u \|_{\widetilde{L^{\infty}}( 0,t ; \dB_{p,1}^{s} ) \cap L^1( 0,t ; \dB_{p,1}^{s+2} ) }^{h;\frac{\beta}{\varepsilon}}\\
            &\quad 
            \leqslant
            C
            \| (\varepsilon a_0, u_0) \|_{\dB_{p,1}^{s+1} \times \dB_{p,1}^{s}}^{h;\frac{\beta}{\varepsilon}}
            +
            \frac{C\varepsilon}{p}
            \left\|
            \| \div u \|_{L^{\infty}}
            \| a \|_{\dB_{p,1}^{s+1}}
            \right\|_{L^1(0,t)}\\
            &\qquad 
            +
            C
            \sum_{j \in \mathbb{Z}}
            2^{(s+1)j}
            \| \Delta_j f + u\cdot \nabla \Delta_j a \|_{{L^1(0,t;L^p)}}
            +
            C
            \| (f,g) \|_{L^1( 0,t ; \dB_{p,1}^{s})}^{h;\frac{\beta}{\varepsilon}}
        \end{align}
        for all $s \in \mathbb{R}$, $\beta \geqslant \beta_0 \sqrt{\langle \Omega \varepsilon^2 \rangle / 2 }$, $1 \leqslant p < \infty$, and $t>0$, provided that the right{-}hand side is finite.
    \end{lemm}
    \begin{proof}
        {Following} the argument in  \cite{Ha-11}{,}
        {we} introduce the effective velocity $w:=u+\varepsilon^{-1}\nabla(-\Delta)^{-1}a$.
        Then, there holds 
        \begin{align}\label{eq:eff-ve}
            \begin{dcases}
            \varepsilon \partial_t a + \frac{1}{\varepsilon} a 
            = 
            \varepsilon f - \div w, \\
            \partial_t w - \mathcal{L}w = \frac{1}{\varepsilon}\nabla (-\Delta)^{-1}f + g - \Omega(e_3 \times u) - \frac{1}{\varepsilon^2}\nabla\div(-\Delta)^{-1}u,\\
            a(0,x) = a_0(x), \quad w(0,x)=w_0(x),
            \end{dcases}
        \end{align}
        where $w_0:=u_0+\varepsilon^{-1}\nabla(-\Delta)^{-1}a_0$.
        Applying $\Delta_j$ to {the first equation in} \eqref{eq:eff-ve} we see that 
        \begin{align}\label{eq:eff-ve-j}
            \varepsilon \partial_t \Delta_j a + \frac{1}{\varepsilon} \Delta_j a
            +
            \varepsilon 
            u \cdot \nabla \Delta_j a
            = 
            \varepsilon (\Delta_j f + u \cdot \nabla \Delta_j a) 
            - 
            \div \Delta_j w
        \end{align}
        Taking the $L^2(\mathbb R^3)^3$-inner product the first equation of \eqref{eq:eff-ve-j} with $|\Delta_j a|^{p-2}\Delta_ja$, we obtain 
        \begin{align}
            \frac{\varepsilon}{p}
            \frac{d}{dt}
            \| \Delta_j a \|_{L^p}^p
            +
            \frac{1}{\varepsilon}\| \Delta_j a \|_{L^p}^p
            \leqslant{}&
            \left(
            \varepsilon\| \Delta_j f + u \cdot \nabla \Delta_j a \|_{L^p}
            +
            \| \div \Delta_j w \|_{L^p}
            \right)
            \| \Delta_j a \|_{L^p}^{p-1}\\
            &
            +
            \frac{1}{p}
            \int_{\mathbb R^3} | \div u | |\Delta_j a |^p dx,
        \end{align}
        where we have used 
        \begin{align}
            \int_{\mathbb R^3} ( u \cdot \nabla \Delta_j a )\cdot|\Delta_ja|^{p-2}\Delta_ja dx
            =
            -
            \frac{1}{p}
            \int_{\mathbb R^3} (\div u)
            |\Delta_ja|^p dx.
        \end{align}
        Thus, we obtain 
        \begin{align}\label{eff-a}
            \begin{split}
            &\varepsilon
            \| \Delta_j a \|_{L^{\infty}(0,t;L^p)}
            +
            \frac{1}{\varepsilon}
            \| \Delta_j a \|_{L^1(0,t;L^p)}
            \leqslant
            \varepsilon 
            \| \Delta_j a_0 \|_{L^p}
            +
            C_1^*2^j\| \Delta_j w \|_{L^1(0,t;L^p)}\\
            &\quad
            +
            \varepsilon
            \| \Delta_j f + u \cdot \nabla \Delta_j a\|_{L^1(0,t;L^p)}
            +
            \frac{\varepsilon}{p}
            \left\|
            \| \div u \|_{L^{\infty}}
            \| a \|_{\dB_{p,1}^{\frac{3}{p}}}
            \right\|_{L^1(0,t)}
            \end{split}
        \end{align}
        for some positive constant $C_1^*$.
        It follows from the maximal regularity estimates for the semigroup generated by the Lam\'e operator that
        \begin{align}\label{eff-w}
            \begin{split}
            &
            \| \Delta_j w \|_{L^{\infty}(0,t;L^p)}
            +
            2^{2j}
            \| \Delta_j w \|_{L^1(0,t;L^p)}
            \leqslant{}
            C_2^*\| \Delta_j w_0 \|_{L^p}\\
            &\quad
            +
            \frac{C_2^*}{\varepsilon}2^{-j}\| \Delta_j f \|_{L^1(0,t;L^p)}
            +
            C_2^*\| \Delta_j g \|_{L^1(0,t;L^p)}
            +
            C_2^*
            \frac{\langle \Omega \varepsilon^2 \rangle}{\varepsilon^2}
            \|\Delta_ju\|_{L^1(0,t;L^p)}
            \end{split}
        \end{align}
        for some positive constant $C_2^*=C_2^*(\mu)$.
        Let $\delta$ be a positive constant to be determined later.
        Then, combining \eqref{eff-a} and \eqref{eff-w}, we have
        \begin{align}\label{est-eff-1}
            \begin{split}
            &\delta 2^j
            \varepsilon
            \| \Delta_j a \|_{L^{\infty}(0,t;L^p)}
            +
            \frac{\delta2^j}{\varepsilon}
            \| \Delta_j a \|_{L^1(0,t;L^p)}
            +
            \| \Delta_j w \|_{L^{\infty}(0,t;L^p)}
            +
            2^{2j}
            \| \Delta_j w \|_{L^1(0,t;L^p)}\\
            &\quad
            \leqslant
            \delta 2^j
            \varepsilon 
            \| \Delta_j a_0 \|_{L^p}
            +
            C_2^*
            \| \Delta_j w_0 \|_{L^p}
            +
            \delta 
            C_1^*2^{2j}\| \Delta_j w \|_{L^1(0,t;L^p)}\\
            &\qquad
            +
            \delta 2^j 
            \varepsilon
            \| \Delta_j f + u \cdot \nabla \Delta_j a\|_{L^1(0,t;L^p)}
            +
            \frac{\delta 2^j \varepsilon}{p}
            \bigg\|
            \| \div u \|_{L^{\infty}}
            \|\Delta_j a \|_{L^p}
            \bigg\|_{L^1(0,t)}\\
            &\qquad 
            +
            \frac{C_2^*}{2^j\varepsilon}\| \Delta_j f \|_{L^1(0,t;L^p)}
            +
            C_2^*\| \Delta_j g \|_{L^1(0,t;L^p)}
            +
            C_2^*
            \frac{\langle \Omega \varepsilon^2 \rangle}{\varepsilon^2}
            \|\Delta_ju\|_{L^1(0,t;L^p)}.
            \end{split}
        \end{align}
        For the last term of the right{-}hand side of {\eqref{eff-a}}, there holds by the definition of $w$ that
        \begin{align}\label{est-eff-2}
            C_2^*
            \frac{\langle \Omega \varepsilon^2 \rangle}{\varepsilon^2}
            \|\Delta_ju\|_{L^1(0,t;L^p)}
            \leqslant
            C_3^*
            \frac{\langle \Omega \varepsilon^2 \rangle}{(2^j\varepsilon)^2}
            \left(
            2^{2j}\|\Delta_jw\|_{L^1(0,t;L^p)}
            +
            \frac{2^j}{\varepsilon}
            \|\Delta_ja\|_{L^1(0,t;L^p)}
            \right)\qquad
        \end{align}
        for some positive constant $C_3^*=C_3^*(\mu)$.
        Let us choose $\delta$ so that $1-C_1^*\delta  - {\delta} \geqslant {\delta}$ and assume that $j$ satisfies ${C_3^*}\langle \Omega \varepsilon^2 \rangle/(2^j\varepsilon)^2 \leqslant {\delta}$.
        Then, collecting \eqref{est-eff-1} and \eqref{est-eff-2}, 
        we see that 
        \begin{align}
            &
            2^j
            \varepsilon
            \| \Delta_j a \|_{L^{\infty}(0,t;L^p)}
            +
            \frac{2^j}{\varepsilon}
            \| \Delta_j a \|_{L^1(0,t;L^p)}
            +
            \| \Delta_j w \|_{L^{\infty}(0,t;L^p)}
            +
            2^{2j}
            \| \Delta_j w \|_{L^1(0,t;L^p)}\\
            &\quad
            \leqslant
            {C}
            2^j
            \varepsilon 
            \| \Delta_j a_0 \|_{L^p}
            +
            {C}
            \| \Delta_j w_0 \|_{L^p}
            +
            C\| \Delta_j (f,g) \|_{L^1(0,t;L^p)}\\
            &\qquad
            +
            C
            2^j 
            \varepsilon
            \| \Delta_j f + u \cdot \nabla \Delta_j a\|_{L^1(0,t;L^p)}
            +
            \frac{C2^j \varepsilon}{p}
            \bigg\|
            \| \div u \|_{L^{\infty}}
            \|\Delta_j a \|_{L^p}
            \bigg\|_{L^1(0,t)}.
        \end{align}
        Here, since
        there hold
        \begin{align}
            2^j\varepsilon\|\Delta_ja_0 \|_{L^p}
            +
            \| \Delta_j w_0 \|_{L^p}
            \sim 
            2^j\varepsilon\|\Delta_ja_0 \|_{L^p}
            +
            \| \Delta_j u_0 \|_{L^p}
        \end{align}
        and
        \begin{gather}
            2^j
            \varepsilon
            \| \Delta_j a \|_{L^{\infty}(0,t;L^p)}
            +
            \| \Delta_j w \|_{L^{\infty}(0,t;L^p)}
            \sim 
            2^j
            \varepsilon
            \| \Delta_j a \|_{L^{\infty}(0,t;L^p)}
            +
            \| \Delta_j u \|_{L^{\infty}(0,t;L^p)},\\
            \frac{2^j}{\varepsilon}
            \| \Delta_j a \|_{L^1(0,t;L^p)}
            +
            2^{2j}
            \| \Delta_j w \|_{L^1(0,t;L^p)}
            \sim
            \frac{2^j}{\varepsilon}
            \| \Delta_j a \|_{L^1(0,t;L^p)}
            +
            2^{2j}
            \| \Delta_j u \|_{L^1(0,t;L^p)},
        \end{gather}
        provided that $2^j\varepsilon \geqslant R$ for some sufficiently large absolute constant $R>1$, we complete the proof.
    \end{proof}

\section{Linear dispersive estimates}\label{sec:lin-str}
    In this section, we revisit the linearized equation 
    \begin{align}\label{eq:lin-re}
        \begin{dcases}
            \partial_t a + \frac{1}{\varepsilon} \div u = f, 
            & 
            t > 0, x \in \mathbb{R}^3,\\
            \partial_t u 
            -
            \mathcal{L} u
            + 
            \Omega ( e_3 \times u ) 
            +
            \frac{1}{\varepsilon} \nabla a 
            =
            g,
            &
            t > 0, x \in \mathbb{R}^3,\\
            a(0,x) = a_0(x), \quad u(0,x) = u_0(x),
            &
            x \in \mathbb{R}^3
        \end{dcases}
    \end{align}
    and consider the dispersive nature of the rotation and the acoustic waves in \eqref{eq:lin-re}.
    Our aim of this section is to prove the following proposition{.}
    \begin{prop}\label{prop:lin-str-visc}
        Let $2 \leq q,r \leq \infty$ satisfy
        \begin{align}
            \frac{1}{q} + \frac{1}{r} \leq \frac{1}{2}, \qquad
            ( q,r ) \neq ( \infty,2 ).
        \end{align}
        Then, 
        there exists a positive constant $C=C(q,r)$ such that
        for any $\varepsilon>0$ and $\Omega \in \mathbb{R} \setminus \{0\}$,
        the solution $(a,u)$ of \eqref{eq:lin-re} satisfies for any $0 < t \leqslant T < \infty$,
        \begin{align}\label{lin-str-visc-1}
            \begin{split}
            \| \Delta_j ( a,u ) \|_{L^r( 0,t ; L^q )}
            \leq{}
            &
            C
            2^{3(\frac{1}{2} - \frac{1}{q} - \frac{1}{r})j}
            |\Omega|^{\frac{2}{r}}
            \varepsilon^{\frac{3}{r}}
            \left\langle \Omega^2\varepsilon^{2}T \right\rangle\\
            &\times
            \left(
            \| \Delta_j(a_0, u_0) \|_{L^2}
            +
            \| \Delta_j(f, g) \|_{L^1( 0,t ; L^2)}
            \right)
            \end{split}
        \end{align}
        for all $j \in \mathbb{Z}$ with $2^j \leq |\Omega| \varepsilon$ and
        \begin{align}\label{lin-str-visc-2}
            \begin{split}
            \| \Delta_j ( a,u ) \|_{L^r( 0,t ; L^q )}
            \leq{}
            &
            C
            2^{3(\frac{1}{2} - \frac{1}{q})j}
            |\Omega|^{-\frac{1}{r}}\\
            &\times
            \left(
            \| \Delta_j(a_0, u_0) \|_{L^2}
            +
            \| \Delta_j(f, g) \|_{L^1( 0,t ; L^2)}
            \right)
            \end{split}
        \end{align}
        for all $j \in \mathbb{Z}$ with $|\Omega| \varepsilon < 2^j \leqslant \beta_0/\varepsilon$.
        Here, $\beta_0$ is the constant {appearing} in Lemma~\ref{lemm:high-ene}.
        In particular, it holds
        \begin{align}\label{lin-str-visc-3}
            \begin{split}
            \| (a,u) \|_{\widetilde{L^r}( 0,t ; \dB_{q,1}^{\frac{3}{q}-1+\frac{2}{r}} )}^{\ell;\alpha}
            \leqslant{}&
            C
            \langle \alpha^2T \rangle^{2} 
            \left( 
            \varepsilon^{\frac{2}{3r}} 
            + 
            |\Omega|^{-\frac{1}{r}}
            \alpha^{\frac{2}{r}} 
            \right)\\
            &
            \times
            \left(
            \| (a_0,u_0) \|_{\dB_{q,1}^{\frac{1}{2}}}^{\ell;\alpha}
            +
            \| (f,g) \|_{L^1(0,t ; \dB_{2,1}^{\frac{1}{2}})}^{\ell;\alpha}
            \right)
            \end{split}
        \end{align}
        if $|\Omega|\varepsilon \leqslant 1$ and $1 \leqslant \alpha \leqslant \beta_0/\varepsilon$.
    \end{prop}
    In order to analyze the dispersive nature of the linear equation {\eqref{eq:lin}},
    we focus on the eigen frequencies of the linear operator.
    However, the corresponding eigen polynomial is given by
    {\eqref{egn-ply}, and hence the corresponding eigen frequencies are so complicated}. Moreover, it is difficult to derive the {asymptotic} expansion of the eigen frequencies due to the anisotropy for $\xi$.
    However, the situation  becomes much simpler in the inviscid case:
    \begin{align}\label{eq:lin-inv-1}
        \begin{dcases}
            \partial_t b + \frac{1}{\varepsilon} \div v = f, 
            & 
            t \in \mathbb{R}, x \in \mathbb{R}^3,\\
            \partial_t v 
            + 
            \Omega ( e_3 \times v ) 
            +
            \frac{1}{\varepsilon} \nabla b 
            =
            g,
            &
            t \in \mathbb{R}, x \in \mathbb{R}^3,\\
            b(0,x) = b_0(x), \quad v(0,x) = v_0(x),
            &
            x \in \mathbb{R}^3.
        \end{dcases}
    \end{align}
    {As it will be mentioned below, we may obtain the explicit formulas for the eigen frequencies of \eqref{eq:lin-inv-1} and thus obtain the dispersive estimates and Strichartz estimates for the solutions to \eqref{eq:lin-inv-1}.}
    \begin{lemm}\label{lemm:lin-str-inv}
        Let $2 \leq q,r \leq \infty$ satisfy
        \begin{align}
            \frac{1}{q} + \frac{1}{r} \leq \frac{1}{2}, \qquad
            ( q,r ) \neq ( \infty,2 ).
        \end{align}
        Then, there exists a positive constant $C=C(q,r)$ such that
        for any $\varepsilon >0$ and $\Omega \in \mathbb{R} \setminus \{0\}$,
        the solution $(b,v)$ of \eqref{eq:lin-inv-1} satisfies
        \begin{align}
            \| \Delta_j ( b,v ) \|_{L^r( \mathbb{R} ; L^q )}
            \leq
            C
            2^{3(\frac{1}{2} - \frac{1}{q} - \frac{1}{r})j}
            |\Omega|^{\frac{2}{r}}
            \varepsilon^{\frac{3}{r}}
            \left(
            \| \Delta_j(b_0, v_0) \|_{L^2}
            +
            \| \Delta_j(f, g) \|_{L^1( \mathbb{R} ; L^2)}
            \right)
        \end{align}
        for all $j \in \mathbb{Z}$ with $2^j \leq |\Omega| \varepsilon$ and
        \begin{align}
            \| \Delta_j ( b,v ) \|_{L^r( \mathbb{R} ; L^q )}
            \leq
            C
            2^{3(\frac{1}{2} - \frac{1}{q})j}
            |\Omega|^{-\frac{1}{r}}
            \left(
            \| \Delta_j(b_0, v_0) \|_{L^2}
            +
            \| \Delta_j(f, g) \|_{L^1( \mathbb{R} ; L^2)}
            \right)
        \end{align}
        for all $j \in \mathbb{Z}$ with $2^j > |\Omega| \varepsilon$.
    \end{lemm}
    Before starting the proof of Lemma \ref{lemm:lin-str-inv}, we make some preparations.
    First, we introduce the following scaling transform:
    \begin{align}\label{inv-scaling}
        \begin{aligned}
        \widetilde{b}(t,x)
        & :=
        b
        \left( \frac{t}{\Omega}, \frac{x}{\Omega \varepsilon} \right), & \qquad
        \widetilde{v}(t,x)
        & :=
        v
        \left( \frac{t}{\Omega}, \frac{x}{\Omega \varepsilon} \right),\\
        \widetilde{f}(t,x)
        & :=
        \frac{1}{\Omega}
        f
        \left( \frac{t}{\Omega}, \frac{x}{\Omega \varepsilon} \right), & \qquad
        \widetilde{g}(t,x)
        & :=
        \frac{1}{\Omega}
        g
        \left( \frac{t}{\Omega}, \frac{x}{\Omega \varepsilon} \right)
        \end{aligned}
    \end{align}
    with $\tb_0(x):=\tb(0,x)$ and $\tv_0(x):=\tv(0,x)$.
    Then, we easily see that
    \begin{align}\label{eq:lin-inv-2}
            \begin{dcases}
            \partial_t \tb + \div \tv = \tf, 
            & 
            t\in \mathbb{R}, x \in \mathbb{R}^3,\\
            \partial_t \tv 
            + 
            e_3 \times \tv  
            +
            \nabla \tb 
            =
            \tg,
            &
            t \in \mathbb{R}, x \in \mathbb{R}^3,\\
            \tb(0,x) = \tb_0(x),\quad \tv(0,x) = \tv_0(x), & x \in \mathbb{R}^3.
            \end{dcases}
    \end{align}
    Applying the Fourier transform to \eqref{eq:lin-inv-2}, we have
    \begin{align}
            \partial_t 
            \begin{pmatrix}
                \widehat{\tb}(t,\xi) \\ \widehat{\tv}(t,\xi)
            \end{pmatrix}
            -
            {A} (\xi)
            \begin{pmatrix}
                \widehat{\tb}(t,\xi) \\ \widehat{\tv}(t,\xi)
            \end{pmatrix}
            =
            \begin{pmatrix}
                \widehat{\tf}(t,\xi) \\ \widehat{\tg}(t,\xi)
            \end{pmatrix},
    \end{align}
    where
    \begin{align}
        A(\xi)
            :=
            \begin{pmatrix}
                0 & -i\xi_1 & -i\xi_2 & -i\xi_3 \\
                -i\xi_1 & 0 & 1 & 0 \\
                -i\xi_2 & -1 & 0 & 0 \\
                -i\xi_3 & 0 & 0 & 0
            \end{pmatrix}.
    \end{align}
    Thus, the solution $(\tb,\tv)$ to \eqref{eq:lin-inv-2} is given by 
    \begin{align}
            \begin{pmatrix}
                \tb(t) \\ \tv(t)
            \end{pmatrix}
            =
            e^{tA(D)}
            \begin{pmatrix}
                \tb_0 \\ \tv_0
            \end{pmatrix}
            +
            \int_0^t
            e^{(t-\tau)A(D)}
            \begin{pmatrix}
                \tf(\tau) \\ \tg(\tau)
            \end{pmatrix}
            d\tau{,}
    \end{align}
    where the evolution group $\{e^{tA(D)} \}_{t \in \mathbb{R}}$ is defined by
    \begin{align}\label{A-group}
            e^{tA(D)}F
            :=
            \mathscr{F}^{-1}
            \left[
            e^{tA(\xi)}
            \widehat{F}(\xi)
            \right]
    \end{align}
    for given function $F : \mathbb{R}^3 \to \mathbb{R}^4$.
    By the direct calculation, the eigenvalues of the matrix $A(\xi)$ are given by
    {
    $\sigma_1 \lambda^{\sigma_2}(\xi)$, $\sigma_1,\sigma_2 \in \{\pm \}$,
    where $\lambda^{\pm}(\xi)$ are given by
    }
    \begin{align}
            \lambda^{\pm}(\xi)
            :={}
            \frac{1}{2}
            \left(
            \sqrt{|\xi|^2 + 2\xi_3 +1}
            \pm
            \sqrt{|\xi|^2 - 2\xi_3 +1}
            \right)
            ={}
            \frac{1}{2}
            \left(
            \eta^+(\xi) \pm \eta^{-}(\xi)
            \right)
    \end{align}
    and $\eta^{\pm}(\xi):=\sqrt{|\xi|^2 \pm 2\xi_3 +1}$.
    Let $a_{\sigma_1,\sigma_2}(\xi)$ be the eigenvector corresponding to the eigenvalue $\sigma_1 \lambda^{\sigma_2}(\xi)$ for $\sigma_1,\sigma_2 \in \{\pm \}$.
    Here, we note that the eigenvectors $\{a_{\sigma_1,\sigma_2}(\xi)\}_{\sigma_1,\sigma_2\in\{\pm\}}$ can be chosen as an orthonormal basis of $\mathbb{C}^4$ since $A(\xi)$ is skew symmetric.
    Then, we have
    \begin{align}\label{sol-form}
            e^{tA(D)}F
            &
            =
            \mathscr{F}^{-1}
            \left[
            e^{tA(\xi)}
            \widehat{F}(\xi)
            \right]\\
            & 
            =
            \sum_{\sigma_1,\sigma_2 \in \{\pm \}}
            \mathscr{F}^{-1}
            \left[
            e^{\sigma_1 i t\lambda^{\sigma_2}(\xi)}
            \langle
            a_{\sigma_1,\sigma_2}(\xi),
            \widehat{F}(\xi)
            \rangle_{\mathbb{C}^4}
            a_{\sigma_1,\sigma_2}(\xi)
            \right]\\
            &
            =
            \sum_{\sigma_1,\sigma_2 \in \{\pm \}}
            e^{\sigma_1 i t\lambda^{\sigma_2}(D)}
            P_{\sigma_1,\sigma_2}
            F,
    \end{align}
    where we have set
    \begin{align}
        e^{\sigma_1 i t\lambda^{\pm}(D)}f
        :={}&
        \mathscr{F}^{-1}
        \left[
        e^{{\sigma_1} it\lambda^{\pm}(\xi)}
        \widehat{f}(\xi)
        \right]
        \intertext{and}
        P_{\sigma_1,\sigma_2}F
        :={}&
        \mathscr{F}^{-1}
        \left[
            \langle
            a_{\sigma_1,\sigma_2}(\xi),
            \widehat{F}(\xi)
            \rangle_{\mathbb{C}^4}
            a_{\sigma_1,\sigma_2}(\xi)
        \right].\label{P}
    \end{align}
    See \cite{Ng-Sc-18} for the {detailed} calculation {for} the derivation of the solution formula for~\eqref{eq:lin-inv-2}.
    We now focus on the dispersive properties of the evolution group $\left\{e^{it \lambda^{\pm} (D) } \right\}_{t \in \mathbb{R}}${and show the following lemma}.
    \begin{lemm}\label{lemm:disp}
        There exists a positive constant $C$ such that
        \begin{align}
            \left \| \Delta_j e^{i t \lambda^{\pm}(D)} \psi \right \|_{L^{\infty}} 
            \leq{}&
            C 2^{3j} \left( 1 + \min \{2^{3j}, 1 \} | t | \right)^{-1} \| \psi \|_{L^1} 
        \end{align}
        for all $t \in \mathbb{R}$, $j \in \mathbb{Z}$, and $\psi \in L^1(\mathbb{R}^3)$.
    \end{lemm}
    \begin{rem}
        In \cite{Ng-Sc-18},
        it is proved that 
        for any {$(r, R)$ satisfying} $0<r\ll 1 \ll R<\infty$, 
        there exists a positive constant $C_{r,R}$ such that    
        \begin{align}\label{disp-known}
            \left\| \chi_{r,R}(D) e^{it\lambda^\pm(D)} \psi \right\|_{L^{\infty}}
            \leqslant
            C_{r,R}( 1 + |t| )^{-\frac{1}{2}}\| \psi \|_{L^1}
        \end{align}
        for all $\psi \in L^1(\mathbb{R}^3)$.
        Here, the symbol $\chi_{r,R}(\xi)$ of the Fourier multiplier $\chi_{r,R}(D)$ satisfies
        \begin{align}
            \chi_{r,R} \in C_c^{\infty}(\mathbb{R}^3),\qquad
            \supp \chi_{r,R} \subset \mathcal{D}_{\frac{r}{2},2R},\qquad 
            \chi_{r,R}(\xi)=1 {\rm \ on\ }\mathcal{D}_{r,R},
        \end{align}
        where $\mathcal{D}_{r,R}:=\{\xi \in \mathbb{R}^3\ ;\ |\xih|, |\xi_3|\geqslant r, |\xi|\leqslant R\}$.
        Therefore, Lemma \ref{lemm:disp} improves~\eqref{disp-known} 
        in the {viewpoint} of the decay rate and the choice of the frequency localization operator. 
    \end{rem}
    In order to prove Lemma \ref{lemm:disp}, 
    we prepare a lemma for the decay estimates of oscillatory integrals given by the stationary phase method.
    We first recall {the following result that is well-known as the Littman lemma}.
    \begin{lemm}[\cites{St-Sh-11}]\label{lemm:sta-pha}
	    Let $d\geqslant 2$ be an integer.
	    Let $\varphi \in C_c^{\infty}(\mathbb{R}^d)$ and $p \in C^{\infty}(\mathcal{D};\mathbb{R})$,
	    where $\mathcal{D}$ is a neighborhood of $\supp \varphi$.
	    Assume that $p$ satisfies
	    \begin{equation*}
		    \rank(\nabla^2p(\xi))\geqslant k,\qquad \xi\in \supp \varphi
	    \end{equation*}
	    for some $k\in \{1,...,d\}$.
	    Then, there exists a positive constant $C$ depending on $d$, $\varphi$, and $\{\nabla^{\alpha} p\}_{2\leqslant |\alpha| \leqslant d+3}$ such that
	    \begin{align}\label{osc.rate}
		    \left|
		    \int_{\mathbb{R}^d}
		    e^{ix\cdot \xi}
		    e^{i\tau p(\xi)}\varphi(\xi) d\xi
		    \right|
		    \leqslant
		    C(1+|\tau|)^{-\frac{k}{2}}
	    \end{align}
	    for all $x\in \mathbb{R}^d$ and $\tau \in \mathbb{R}$.
    \end{lemm}
    The estimate \eqref{osc.rate} is stable under small $C^{d+3}$-perturbation of the phase function $p$.
    More precisely, the following lemma holds{.}
    \begin{lemm}\label{lemm:stab}
	    Under the same assumption as in Lemma \ref{lemm:sta-pha} there exist positive constants $\epsilon=\epsilon(d,k,\varphi,\nabla^2p)\leqslant 1$ and $C=C(d,k,\varphi,\{\nabla^{\alpha} p\}_{2\leqslant |\alpha| \leqslant d+3})$
        such that
	    \begin{equation}\label{osc:est}
		    \left|
		    \int_{\mathbb{R}^d}
		    e^{ix\cdot \xi}
		    e^{i\tau q(\xi)}\varphi(\xi) d\xi
		    \right|
		    \leqslant
		    C(1+|\tau|)^{-\frac{k}{2}}
	    \end{equation}
	    for all $x\in \mathbb{R}^d$, $\tau \in \mathbb{R}$, and
	    $q\in C^{\infty}(\mathcal{D};\mathbb{R})$ satisfying
	    $\|q-p\|_{C^{d+3}(\supp \varphi)}\leqslant \epsilon$.
    \end{lemm}
    We remark that the non-degenerate case $k=d$ is mentioned in \cite{Le-Ta-17}*{Lemma 3.3}.
    Although the proof of Lemma \ref{lemm:stab} is {similar} to \cite{St-Sh-11}*{Proposition 2.5, p.329}, we shall give a sketch of its proof in Appendix \ref{ses:app} for {the} readers' convenience.

    Under these preparations, we may prove Lemma \ref{lemm:disp}.
    \begin{proof}[Proof of Lemma \ref{lemm:disp}]
        Since it holds
        \begin{align}
            \left[\Delta_j e^{i t \lambda^{\pm}(D)} \psi\right](x)
            =
            \frac{1}{(2\pi)^3}
            \left[ I_j^{\pm}(t,\cdot) * \psi \right](x),
        \end{align}
        where
        \begin{align}
            I_j^{\pm}(t,x)
            :=
            \int_{\mathbb{R}^3}
            e^{i x \cdot \xi}
            e^{it \lambda^{\pm}(\xi)}
            {\phi_j}(\xi) d\xi
            =
            2^{3j}
            \int_{\mathbb{R}^3}
            e^{i ( 2^j x ) \cdot \xi}
            e^{it \lambda^{\pm}( 2^j \xi )}
            {\phi_0}(\xi) d\xi,
        \end{align}
        we see by the Hausdorff-Young inequality that
        \begin{align}\label{pf:disp-1}
            \left\|
            \Delta_j e^{i t \lambda^{\pm}(D)} \psi
            \right\|_{L^{\infty}}
            \leq{}&
            \frac{1}{(2\pi)^3}
            \left\| I_j^{\pm}(t) \right\|_{L^{\infty}}
            \|\psi\|_{L^1}.
        \end{align}
        Thus, it suffices to focus on the estimates of the oscillatory integral $I_j^{\pm}(t,x)$.
        In this proof, we denote the smooth function $R_n(b)$ by the {$n$th} remainder term of the Maclaurin series of $\sqrt{1 + \eta}$, that is,
        \begin{align}\label{taylor}
            \sqrt{1 + \eta}
            =
            \sum_{k=0}^n
            \binom{1/2}{k}
            \eta^k
            +
            R_n(\eta),\qquad
            n=0,1,2,...,\quad \eta>-1.
        \end{align}
    
        \noindent
        {\it Step.1 High frequency analysis.}
        We first consider the estimate for $I_j^+(t)$.
        By \eqref{taylor}, it holds
        \begin{align}
            \lambda^{+}(\xi)
            &={}
            \frac{|\xi|}{2}
            \left(
            \sqrt{1 + \frac{2\xi_3 + 1}{|\xi|^2}}
            +
            \sqrt{1 + \frac{-2\xi_3 + 1}{|\xi|^2}}
            \right)\\
            &={}
            \frac{|\xi|}{2}
            \left(
            1 + R_0\left(\frac{2\xi_3 + 1}{|\xi|^2}\right)
            +
            1 + R_0\left(\frac{-2\xi_3 + 1}{|\xi|^2}\right)
            \right)\\
            &={}
            |\xi|
            +
            \frac{|\xi|}{2}
            \left(
            R_0\left(\frac{2\xi_3 + 1}{|\xi|^2}\right)
            +
            R_0\left(\frac{-2\xi_3 + 1}{|\xi|^2}\right)
            \right)\\
            &=:{}
            |\xi| + R_{\rm high}^+(\xi),
        \end{align}
        which implies 
        \begin{align}
            I_j^+(t,x)
            =
            2^{3j}
            \int_{\mathbb{R}^3}
            e^{i(2^j x) \cdot \xi}
            e^{i 2^j t \left\{|\xi| + 2^{-j}R_{\rm high}^+(2^j \xi ) \right\} }
            {\phi_0}(\xi)
            d\xi.
        \end{align}
        By the direct calculation, we see that
        \begin{align}\label{phase-1}
            \nabla^2( |\xi| )
            =
            \frac{1}{|\xi|^3}
            \begin{pmatrix}
                \xi_2^2 + \xi_3^2 & -\xi_1\xi_2 & -\xi_1\xi_3 \\
                -\xi_1\xi_2 & \xi_1^2 + \xi_3^2 & -\xi_2\xi_3 \\
                -\xi_1\xi_3 & -\xi_2\xi_3 & \xi_1^2 + \xi_3^2
            \end{pmatrix},
            \qquad
            \rank \nabla^2( |\xi| ) \geqslant 2
        \end{align}
        for $\xi \in \supp {\phi_0}$ 
        and
        \begin{align}
            \left\|
            2^{-j}R_{\rm high}^+(2^j\cdot) 
            \right\|_{C^{6}(\supp {\phi_0})}
            ={}&
            2^{-j}
            \sum_{m=0}^6
            2^{mj}
            \| (\nabla^m R^+_{\rm high})(2^{j}\cdot) \|_{L^{\infty}(\supp {\phi_0})}\\
            \leq{}&
            {C}
            2^{-j}
            \sum_{m=0}^6
            2^{mj}
            2^{-(m+2)j}
            =
            C 2^{-3j}.
        \end{align}
        Thus, it follows {from Lemma \ref{lemm:stab}} that 
        there exists a positive integer $j_{\rm high}^+$ such that
        \begin{align}\label{+:high}
            \left\| I_j^+(t) \right\|_{L^{\infty}}
            \leq
            C 2^{3j} ( 1 + 2^j |t| )^{-1}
        \end{align}
        for all $t \in \mathbb{R}$ and $j \geqslant j_{\rm high}^+$.

        Next, we consider the estimate for $I_j^-(t)$.
        We see that 
        \begin{align}
            \lambda^{-}(\xi)
            &={}
            \frac{|\xi|}{2}
            \left(
            \sqrt{1 + \frac{2\xi_3 + 1}{|\xi|^2}}
            {-}
            \sqrt{1 + \frac{-2\xi_3 + 1}{|\xi|^2}}
            \right)\\
            &={}
            \frac{|\xi|}{2}
            \left(
            1 + \frac{2\xi_3 + 1}{2|\xi|^2} + R_1\left(\frac{2\xi_3 + 1}{|\xi|^2}
            \right)
            -
            1 - \frac{-2\xi_3 + 1}{2|\xi|^2} - R_1\left(\frac{-2\xi_3 + 1}{|\xi|^2}
            \right)
            \right)\\
            &={}
            \frac{\xi_3}{|\xi|}
            +
            \frac{|\xi|}{2}
            \left(
            R_1\left(\frac{2\xi_3 + 1}{|\xi|^2}\right)
            -
            R_1\left(\frac{-2\xi_3 + 1}{|\xi|^2}\right)
            \right)\\
            &=:{}
            \frac{\xi_3}{|\xi|} + R_{\rm high}^-(\xi).
        \end{align}
        This yields 
        \begin{align}
            I_j^-(t,x)
            =
            2^{3j}
            \int_{\mathbb{R}^3}
            e^{i(2^j x) \cdot \xi}
            e^{i t \left\{\frac{\xi_3}{|\xi|} + R_{\rm high}^-(2^j \xi ) \right\} }
            {\phi_0}(\xi)
            d\xi.
        \end{align}
        In \cites{Ko-Le-Ta-14-1,Ko-Le-Ta-14-2}, it is calculated that
        \begin{align}
            &\nabla^2 \left( \frac{\xi_3}{|\xi|} \right)
            =
            \frac{1}{|\xi|^5}
            \begin{pmatrix}
                \xi_3(3\xi_1^2 - |\xi|^2) & 3\xi_1\xi_2\xi_3 & \xi_1(3\xi_3^2 - |\xi|^2) \\
                3\xi_1\xi_2\xi_3 & \xi_3(3\xi_2^2 - |\xi|^2) &\xi_2(3\xi_3^2 - |\xi|^2) \\
                \xi_1(3\xi_3^2 - |\xi|^2) & \xi_2(3\xi_3^2 - |\xi|^2) & -3\xi_3(\xi_1^2 + \xi_2^2)
            \end{pmatrix},\\
            &\rank \nabla^2 \left( \frac{\xi_3}{|\xi|} \right) \geqslant 2
        \end{align}
        for $\xi \in \supp {\phi_0}$.
        The direct computation yields
        \begin{align}
            \left\|
            R_{\rm high}^-(2^j\cdot ) 
            \right\|_{C^{6}(\supp {\phi_0})}
            ={}&
            \sum_{m=0}^6
            2^{mj}
            \| (\nabla^m R^-_{\rm high})(2^{j}\cdot) \|_{L^{\infty}(\supp {\phi_0})}\\
            \leq{}&
            C
            \sum_{m=0}^6
            2^{mj}
            2^{-(m+2)j}
            =
            C 2^{-2j}.
        \end{align}
        By Lemma \ref{lemm:stab}, there exists a positive integer $j_{\rm high}^-$ such that
        \begin{align}\label{-:high}
            {\left\| I_j^-(t) \right\|_{L^{\infty}}}
            \leq
            C 2^{3j} ( 1 + |t| )^{-1}
        \end{align}
        for all $t \in \mathbb{R}$ and $j \geqslant j_{\rm high}^-$.
        Hence, it follows {from \eqref{+:high} and \eqref{-:high}} that
        \begin{align}\label{I:high}
            {\left\|I_j^{\pm}(t)\right\|_{L^{\infty}} }
            \leq
            C2^{3j}
            (1+|t|)^{-1}
            =
            C2^{3j}\left(1 + \min\{2^{3j},1 \}|t|\right)^{-1}
        \end{align}
        for all $t \in \mathbb{R}$ and $j \geqslant j_{\rm high}:= \max\{j_{\rm high}^+,j_{\rm high}^-\}$.
    
        \noindent
        {\it Step.2 Low frequency analysis.}
        We consider the estimate for $I_j^+(t)$.
        Using \eqref{taylor}, we have
        \begin{align}
            \lambda^+(\xi)
            ={}&
            \frac{1}{2}
            \left(
            1 + \frac{1}{2}( |\xi|^2 + 2\xi_3 ) - \frac{1}{8} ( |\xi|^2 + 2\xi_3 )^2 + R_2( |\xi|^2 + 2\xi_3 )\right.\\
            &\left.
            \qquad
            +
            1 + \frac{1}{2}( |\xi|^2 - 2\xi_3 ) - \frac{1}{8} ( |\xi|^2 - 2\xi_3 )^2 + R_2( |\xi|^2 - 2\xi_3 )
            \right)\\
            ={}&
            1 + \frac{1}{2}|\xih|^2 - \frac{1}{4} |\xi|^4 + \frac{1}{2}\left(R_2( |\xi|^2 + 2\xi_3 ) + R_2( |\xi|^2 - 2\xi_3 )
            \right)\\
            ={}&
            1 + \frac{1}{2}|\xih|^2 + R_{\rm low}^+(\xi),
        \end{align}
        {
        where $\xih:=(\xi_1,\xi_2)$.
        Thus, we have
        }
        \begin{align}
            I_j^-(t,x)
            ={}&
            2^{3j}
            \int_{\mathbb{R}^3}
            e^{i(2^j x)\cdot \xi}
            e^{it\left\{1 + \frac{1}{2}(2^j|\xih|)^2 + R_{\rm low}^+(2^j\xi)) \right\}}
            {\phi_0}(\xi)
            d\xi\\
            ={}&
            2^{3j}e^{it}
            \int_{\mathbb{R}^3}
            e^{i(2^j x)\cdot \xi}
            e^{i2^{2j}t\left\{\frac{1}{2}|\xih|^2 + 2^{-2j}R_{\rm low}^+(2^j\xi) \right\}}
            {\phi_0}(\xi)
            d\xi.
        \end{align}
        It is easy to see that 
        \begin{align}
            \rank \nabla^2 \left( \frac{1}{2} |\xih|^2 \right)
            =
            \rank 
            \begin{pmatrix}
                1 & 0 & 0 \\
                0 & 1 & 0 \\
                0 & 0 & 0
            \end{pmatrix}
            =2.
        \end{align}
        By the direct computation, we have
        \begin{align}
            \left\| 2^{-2j}R_{\rm low}^+(2^j \cdot ) \right\|_{C^{6}(\supp {\phi_0})}
            ={}&
            2^{-2j}
            \sum_{m=0}^6
            2^{mj}
            \| (\nabla^m R_{\rm low}^+)(2^j\cdot) \|_{L^{\infty}(\supp {\phi_0})}\\
            \leq{}&
            C
            2^{-2j}
            \sum_{m=0}^6
            2^{mj}
            2^{(3-m)j}
            =
            C2^j.
        \end{align}
        Hence, Lemma \ref{lemm:stab} implies that there exists a negative integer $j_{\rm low}^+$ such that 
        \begin{align}\label{+:low}
            \|I_j(t)\|_{L^{\infty}}
            \leq
            C2^{3j}
            (1+2^{2j}|t|)^{-1}
        \end{align}
        for all $t \in \mathbb{R}$ and $j \leq j_{\rm low}^+$.

        We next focus on the estimate of $I_j^-(t)$.
        Since there holds
        \begin{align}
            \lambda^-(\xi)
            ={}&
            \frac{1}{2}
            \left(
            1 + \frac{1}{2}( |\xi|^2 + 2\xi_3 ) - \frac{1}{8} ( |\xi|^2 + 2\xi_3 )^2 + \frac{1}{16} ( |\xi|^2 + 2\xi_3 )^3 + R_3( |\xi|^2 + 2\xi_3 )\right.\\
            &\left.
            -
            1 - \frac{1}{2}( |\xi|^2 - 2\xi_3 ) + \frac{1}{8} ( |\xi|^2 - 2\xi_3 )^2 - \frac{1}{16} ( |\xi|^2 - 2\xi_3 )^3 - R_3( |\xi|^2 - 2\xi_3 )\right)\\
            ={}& 
            \xi_3 - \frac{1}{2}\xi_3|\xih|^2 + \frac{1}{2}
            \left(
            \frac{3}{4}\xi_3|\xi|^4
            + R_3( |\xi|^2 + 2\xi_3 ) - R_3( |\xi|^2 - 2\xi_3 )
            \right)\\
            ={}& 
            \xi_3 - \frac{1}{2}\xi_3|\xih|^2 + R_{\rm low}^-(\xi),
        \end{align}
        we see that
        \begin{align}
            I_j^-(t,x)
            ={}&
            2^{3j}
            \int_{\mathbb{R}^3}
            e^{i (2^jx)\cdot \xi}
            e^{-it\left\{2^j\xi_3 + \frac{1}{2}2^{3j}\xi_3|\xih|^2 - R_{\rm low}^-(2^j\xi) \right\}}
            {\phi_0}(\xi)d\xi\\
            ={}&
            2^{3j}
            \int_{\mathbb{R}^3}
            e^{i 2^j(x_{\rm h},x_3+t)\cdot \xi}
            e^{-i2^{3j}t\left\{\frac{1}{2}\xi_3|\xih|^2 - 2^{-3j}R_{\rm low}^-(2^j\xi) \right\}}
            {\phi_0}(\xi)d\xi.
        \end{align}
        Here, by the direct computation, we have
        \begin{align}
            \rank \nabla^2 \left( \frac{1}{2}\xi_3|\xi|^2 \right)
            =
            \rank
            \begin{pmatrix}
                \xi_3 & 0 & \xi_1 \\
                0 & \xi_3 & \xi_2 \\
                \xi_1 & \xi_2 & 3\xi_3 
            \end{pmatrix}
            \geqslant 2
        \end{align}
        for $\xi \in \supp {\phi_0}$ and
        \begin{align}
            \left\| 2^{-3j}R_{\rm low}^-(2^j \cdot ) \right\|_{C^{6}(\supp {\phi_0})}
            ={}&
            2^{-3j}
            \sum_{m=0}^6
            2^{mj}
            \| (\nabla^m R_{\rm low}^-)(2^j\cdot) \|_{L^{\infty}(\supp {\phi_0})}\\
            \leq{}&
            C
            2^{-3j}
            \sum_{m=0}^6
            2^{mj}
            2^{(4-m)j}
            =
            C2^j.
        \end{align}
        Thus, by vitue of Lemma \ref{lemm:stab}, 
        there exists a negative integer $j_{\rm low}^-$ such that
        \begin{align}\label{-:low}
            {\left\|I_j^-(t)\right\|_{L^{\infty}}}
            \leq
            C2^{3j}
            (1+2^{3j}|t|)^{-1}
        \end{align}
        for all $t \in \mathbb{R}$ and $j \leq j_{\rm low}^-$.
        Therefore, it follows {from \eqref{+:low} and \eqref{-:low}} that
        \begin{align}\label{I:low}
            {\left\|I_j^{\pm}(t)\right\|_{L^{\infty}}}
            \leq
            C2^{3j}
            (1+2^{3j}|t|)^{-1}
            =
            C2^{3j}\left(1 + \min\{2^{3j},1 \}|t|\right)^{-1}
        \end{align}
        for all $t \in \mathbb{R}$ and $j \leq j_{\rm low}:= \min\{j_{\rm low}^+,j_{\rm low}^-\}$.
    
        \noindent
        {\it Step.3 Middle frequency analysis.}
        We first consider the case of $j \neq 0$ in which $\eta^{\pm} \neq 0$ and $\lambda^{\pm}$ is smooth on $\supp \phi_j$.
        It follows from the direct computation that 
        \begin{align}
            \nabla^2 \eta^{\pm}(\xi) 
            = 
            \frac{1}{\eta^{\pm}(\xi)}
            \begin{pmatrix}
                1 & 0 & 0 \\
                0 & 1 & 0 \\
                0 & 0 & 1
            \end{pmatrix}
            - 
            \frac{1}{\eta^{\pm}(\xi)^3}
            \begin{pmatrix}
                \xi_1^2 & \xi_1 \xi_2 & \xi_1(\xi_3 \pm 1) \\
                \xi_1 \xi_2 & \xi_2^2 & \xi_2(\xi_3 \pm 1) \\
                \xi_1(\xi_3 \pm 1) & \xi_2(\xi_3 \pm 1) & (\xi_3 \pm 1)^2
            \end{pmatrix}
        \end{align}
        and
        \begin{align}
            \det \left( \nabla^2_{\xi} \lambda^+(\xi) \right)
            ={}&
            \frac{|\xih|^2 ( \eta^+(\xi) + \eta^-(\xi) )}{2\eta^+(\xi)^4\eta^-(\xi)^4},\\
            \det \left( \nabla^2_{\xi} \lambda^-(\xi) \right)
            ={}&
            -
            \frac{2|\xih|^2\xi_3}{\eta^+(\xi)^4\eta^-(\xi)^4( \eta^+(\xi) + \eta^-(\xi) )}.
        \end{align}
        This implies $\det\left( \nabla^2_{\xi} \lambda^+(\xi) \right) \neq 0$ if ${\xih} \neq 0$, and $\det\left( \nabla^2_{\xi} \lambda^-(\xi) \right) \neq 0$ if $\xih \neq 0$ and $\xi_3 \neq 0$.
        Thus, in the case of $\det\left( \nabla^2_{\xi} \lambda^{\pm}(\xi) \right) = 0$, we see that
        \begin{align}
            &
            \nabla^2 \lambda^{\pm} (0,0,\xi_3) 
            = 
            \frac{1}{2}
            \left(
            \frac{1}{|\xi_3 + 1|}
            \pm
            \frac{1}{|\xi_3 - 1|}
            \right)
            \begin{pmatrix}
                1 & 0 & 0 \\
                0 & 1 & 0 \\
                0 & 0 & 0
            \end{pmatrix},\\
            &
            \nabla^2 \lambda^- (\xi_1,\xi_2,0) 
            = 
            -\frac{1}{(|\xih|^2 + 1)^{\frac{3}{2}}}
            \begin{pmatrix}
                0 & 0 & \xi_1 \\
                0 & 0 & \xi_2 \\
                \xi_1 & \xi_2 & 0
            \end{pmatrix}.
        \end{align}
        Hence, we obtain $\rank \left( \nabla^2_{\xi} \lambda^+(\xi) \right) \geqslant 2$ for all $\xi \in \supp \phi_j$
        and Lemma \ref{lemm:sta-pha} implies that there exists a positive constant $M_j$ such that 
        \begin{align}\label{Mj}
            {\left\| I_j^{\pm}(t) \right\|_{L^{\infty}} \leqslant M_j ( 1 + | t | )^{-1}, \qquad j \neq 0}
        \end{align}
        for all $t \in \mathbb R$.
        
        Next, we consider the case of $j =0$.
        Since $\lambda^{\pm}(\xi)$ is not even $C^1$ on $\supp \phi_0$, 
        we cannot apply Lemma \ref{lemm:sta-pha} directly.
        To overcome this, we follow the argument in \cite{Le-Ta-17} and decompose $I_0^{\pm}(t,x)$ as follows:
        \begin{align}\label{I_0-1}
            \begin{split}
            I_0^{\pm}(t,x)
            &={}
            \sum_{k,\ell \in \mathbb{Z}}
            I_{k,\ell}^{\pm}(t,x)\\
            &={}
            \sum_{0 \leqslant k,\ell \leqslant 2}
            I_{k,\ell}^{\pm}(t,x)
            +
            \sum_{\substack{k \leqslant -1 \\ 0 \leqslant \ell \leqslant 2}}
            I_{k,\ell}^{\pm}(t,x)
            +
            \sum_{\substack{0 \leqslant k \leqslant 2 \\ \ell \leqslant -1}}
            I_{k,\ell}^{\pm}(t,x)\\
            &=:
            J_1^{\pm}(t,x)
            +
            J_2^{\pm}(t,x)
            +
            J_3^{\pm}(t,x),
            \end{split}
        \end{align}
        where we have put
        \begin{gather}
            I_{k,\ell}^{\pm}(t,x)
            :=
            \int_{\mathbb R^3} e^{ix \cdot \xi} e^{it\lambda^{\pm}(\xi)} \psi_{k,\ell}(\xi) d\xi,
        \end{gather}
        and
        \begin{gather}
            \psi_{k,\ell}(\xi) 
            :={}
            \phi_0(\xi)
            \phi_k^+(\xi)
            \phi_{\ell}^-(\xi),\\
            \phi_k^+(\xi)
            :={}
            \phi_0(2^{-k}\xih,2^{-k}(\xi_3 - 1)),\qquad 
            \phi_{\ell}^-(\xi)
            :={}
            \phi_0(2^{-\ell}\xih,2^{-\ell}(\xi_3 + 1)).
        \end{gather}
        Here, we have used the fact that $I_{k,\ell}(t) = 0$ if $k \geqslant 2$ or $\ell \geqslant 2$ in the decomposition~\eqref{I_0-1}.
        By the change of variables $(\xi_1,\xi_2,\xi_3) \mapsto (\xi_1,\xi_2,-\xi_3)$, we see that
        \begin{align}\label{J3}
            J_3^{\pm}(t,x_1,x_2,x_3)
            =
            J_2^{\pm}(\pm t,x_1,x_2,-x_3),
        \end{align}
        which implies that it suffices to consider the estimates of $J_1^{\pm}(t,x)$ and $J_2^{\pm}(t,x)$.
        Since $\lambda^{\pm}$ is smooth on $\supp \psi_{k,\ell}$, we see by the same procedure as above in the case $j \neq 0$ that
        \begin{align}\label{I_0-3}
            \left\| I_{k,\ell}^{\pm}(t,\cdot) \right\|_{L^{\infty}}
            \leqslant 
            M_{k,\ell}( 1 + | t | )^{-1}
        \end{align}
        for some positive constant $M_{k,\ell}$,
        which implies
        \begin{align}\label{J1}
            \left\| J_1^{\pm}(t,\cdot) \right\|_{L^{\infty}}
            \leqslant 
            C
            ( 1 + | t | )^{-1},\qquad
            C=\sum_{0\leqslant k, \ell \leqslant 2} M_{k,\ell}.
        \end{align}
        Concerning the estimates for $J_2(t,x)$, we apply the change of variables and see that
        \begin{align}\label{I_0-2}
            &
            \sum_{0 \leqslant \ell \leqslant 2}I_{k,\ell}^{{\pm}} (t,x)
            =
            e^{ix_3}
            \int_{\mathbb R^3}
            e^{ix \cdot \xi}
            e^{it \lambda^{\pm}(\xi)}
            \phi_0(\xi)
            \phi_{0}(2^{-k}\xih,2^{-k}(\xi_3-1))
            \sum_{0 \leqslant \ell \leqslant 2}
            \phi_{\ell}(\xi)
            d\xi\\
            &\quad=
            e^{ix_3}
            2^{3k}
            \int_{\mathbb R^3}
            e^{i2^kx \cdot \xi}
            e^{it \widetilde{\lambda^{\pm}}(2^k\xi)}
            \phi_{0}(2^k\xih,2^k\xi_3+1)
            \phi_0(\xi)
            \sum_{0 \leqslant \ell \leqslant 2}
            \phi_{\ell}(2^k\xih,2^k\xi_3+1)
            d\xi\\
            &\quad=
            e^{ix_3}
            2^{3k}
            \int_{\mathbb R^3}
            e^{i2^kx \cdot \xi}
            e^{it \widetilde{\lambda^{\pm}}(2^k\xi)}
            \phi_0(\xi)
            \chi(2^k\xi)
            d\xi,
        \end{align}
        where we have set
        \begin{align}
            \widetilde{\lambda^{\pm}}(\xi) & :=\lambda^{\pm}(\xi_h,\xi_3+1)=\frac{1}{2}\left( \sqrt{|\xi|^2+4\xi_3+4} \pm |\xi| \right),\\
            \chi(\xi) & :=\phi_{0}(\xih,\xi_3+1)
            \sum_{0 \leqslant \ell \leqslant 2}
            \phi_{\ell}(\xih,\xi_3+1).
        \end{align}
        Taking $L>0$ so that $\supp \chi \subset [-L,L]^3$, we expand $\chi$ as the Fourier series on $[-L,L]^3$:
        \begin{align}\label{I_0-chi-F}
            \chi(\xi)
            =
            \sum_{n \in \mathbb Z^3}
            c_n e^{i \frac{\pi}{L}n \cdot \xi}.
        \end{align}
        Note that, since $\chi$ is smooth, the Fourier coeffidient $\{c_n\}_{n \in \mathbb Z}$ {belongs to $\ell^1(\mathbb{Z}^3)$}.
        By the Taylor theorem, we see that
        \begin{align}\label{Taylor}
            \begin{split}
            \widetilde{\lambda^{\pm}}(\xi)
            &={}
            \sqrt{1 + \xi_3 + \frac{1}{4}|\xi|^2}
            \pm 
            \frac{1}{2}|\xi|\\
            &={}
            1 + \frac{1}{2}\xi_3 \pm \frac{1}{2}|\xi| + \frac{1}{8}|\xi|^2 + R_1\left( \xi_3 + \frac{1}{4}|\xi|^2 \right)\\
            &=:{}
            1 + \frac{1}{2}\xi_3 \pm \frac{1}{2}|\xi| + r(\xi).
            \end{split}
        \end{align}
        Substituting \eqref{I_0-chi-F} and \eqref{Taylor} into \eqref{I_0-2}, we have
        \begin{align}
            \sum_{0 \leqslant \ell \leqslant 2}{I_{k,\ell}^{\pm}}(t,x)
            &
            =
            2^{3k}
            e^{ix_3}
            \sum_{n \in \mathbb Z^3}
            c_n
            \int_{\mathbb R^3}
            e^{i2^k(x+\frac{\pi}{L}n)\cdot\xi} 
            e^{it \widetilde{\lambda^{\pm}}(2^k\xi)}
            \phi_0(\xi)
            d\xi\\
            &
            =
            2^{3k}
            e^{i(x_3+t)}
            \sum_{n \in \mathbb Z^3}
            c_n
            \int_{\mathbb R^3}
            e^{i2^k(x+\frac{\pi}{L}n + \frac{t}{2}e_3)\cdot\xi} 
            e^{i2^kt( \frac{1}{2}|\xi| + 2^{-k}r(2^k\xi))}
            \phi_0(\xi)
            d\xi.
        \end{align}
        {We remember \eqref{phase-1} and see that}
        \begin{align}
            \left\|
            2^{-k}r(2^k\cdot) 
            \right\|_{C^{6}(\supp {\phi_0})}
            ={}&
            2^{-k}
            \sum_{m=0}^6
            2^{mk}
            \| (\nabla^m r)(2^k\cdot) \|_{L^{\infty}(\supp {\phi_0})}\\
            \leq{}&
            C
            2^{-k}
            \sum_{m=0}^6
            2^{mk}
            2^{(2-m)k}
            =
            C 2^k.
        \end{align}
        Then, Lemma \ref{lemm:stab} implies that there exists a negative integer $k_0$ such that
        \begin{align}
            \left|
            \int_{\mathbb R^3}
            e^{i2^k(x+\frac{\pi}{L}n + \frac{t}{2}e_3)\cdot\xi} 
            e^{i2^kt( \frac{1}{2}|\xi| + 2^{-k}r(2^k))}
            \phi_0(\xi)
            d\xi
            \right|
            \leqslant
            C(1+2^k|t|)^{-1}
        \end{align}
        for all $k\leqslant k_0$,
        which yields
        \begin{align}\label{I_0-4}
            \left|
            \sum_{0 \leqslant \ell \leqslant 2}I_{k,\ell}^{\pm}(t,x)
            \right|
            \leqslant{}
            C\sum_{n \in \mathbb Z^3}|c_n|2^{3k}(1+2^k|t|)^{-1}
            \leqslant{}
            C2^{3k}(1+2^k|t|)^{-1}.
        \end{align}
        It follows {from \eqref{I_0-3} and \eqref{I_0-4}} that 
        \begin{equation}\label{J2}
            \begin{split}
                \| J_2(t,\cdot) \|_{L^{\infty}}
                \leqslant{}&
                \sum_{k \leqslant 0} 
                \left|
                \sum_{0 \leqslant \ell \leqslant 2}I_{k,\ell}^{\pm}(t,x)
                \right|\\
                \leqslant{}&
                C
                \sum_{k \leqslant 0}
                2^{3k}(1+2^k|t|)^{-1}\\
                \leqslant{}&
                C(1+|t|)^{-1}.
            \end{split}        
        \end{equation}
        Gathering {\eqref{I_0-1}, \eqref{J3}, }\eqref{J1}{,} and \eqref{J2}, we {have}
        \begin{align}\label{M0}
            {\left\|I_0^{\pm}(t)\right\|_{L^{\infty}} \leq M_0(1 + |t|)^{-1}}
        \end{align}
        {with some positive constant $M_0$, }for all $t \in \mathbb{R}$.
        Then, for $\min\{j_{\rm low}^+,j_{\rm low}^-\} < j <\max\{j_{\rm high}^+,j_{\rm high}^-\}$, we {obtain by \eqref{Mj} and \eqref{M0} that}
        \begin{align}\label{I:middle}
            \begin{split}
            {\left\|I_j^{\pm}(t)\right\|_{L^{\infty}} }
            \leq{}& M_j(1 + |t|)^{-1}\\
            \leq{}& 2^{-3j}M_j \cdot 2^{3j}\left(1 + \min\{2^{3j},1 \}|t|\right)^{-1}\\
            \leq{}& 
            C2^{3j}\left(1 + \min\{2^{3j},1 \}|t|\right)^{-1}
            \end{split}
        \end{align}
        for all $t \in \mathbb{R}$.

        From \eqref{pf:disp-1}, \eqref{I:high}, \eqref{I:low}, and \eqref{I:middle}, we complete the proof.
    \end{proof}
    By vitue of Lemma \ref{lemm:disp} {and the argument in \cite{Ke-Ta-98}}, we establish the Strichartz estimates for {$\left\{e^{it\lambda_{\pm}(D)} \right\}_{t \in \mathbb{R}}$}. 
    \begin{lemm}\label{lemm:str}
        Let $2 \leq q,r \leq \infty$ satisfy
        \begin{align}
            \frac{1}{q} + \frac{1}{r} \leq \frac{1}{2}, \qquad
            ( q,r ) \neq ( \infty,2 ).
        \end{align}
        Then, there exists a positive constant $C=C(q,r)$ such that 
        \begin{align}
            { \left\| \Delta_j e^{it\lambda^{\pm}(D)}f \right\|_{L^r(\mathbb{R};L^q)}}
            \leq
            \begin{cases}
                C 2^{3(\frac{1}{2}-\frac{1}{q}-\frac{1}{r})j}\| \Delta_j f \|_{L^2} & (j \leq 0),\\
                C 2^{3(\frac{1}{2}-\frac{1}{q})j}\| \Delta_j f \|_{L^2} & (j > 0)
            \end{cases}
        \end{align}
        for all {$j \in \mathbb{Z}$ and }$f \in \mathscr{S}'(\mathbb{R}^3)$ with $\Delta_jf \in L^2(\mathbb{R}^3)$
    \end{lemm}
    \begin{proof}
        The strategy of the proof is based on the standard $TT^*$ argument (see \cites{Ke-Ta-98}).
        For the readers' convenience, we shall give the proof. 
        It follows from the Plancherel theorem that
        \begin{align}
            \left \| \Delta_j e^{i t \lambda^{\pm}(D)} f \right \|_{L^2}
            =
            C\left \|  e^{i t \lambda^{\pm}(\xi)} {\phi_j}(\xi) \widehat{f}(\xi) \right \|_{L^2_{\xi}}
            =
            \| \Delta_j f \|_{L^2},
        \end{align}
        which implies the case $(q,r)=(2,\infty)$.
        Thus, it suffices to consider the other cases.
        Let $\psi \in \mathscr{S}(\mathbb{R}\times\mathbb{R}^3)$.
        By the Parseval formula for the space variables, it holds that
        \begin{align}\label{pf:str-1}
            \begin{split}
            &\left| \int_{\mathbb{R} \times \mathbb{R}^3} \Delta_j e^{it\lambda^{\pm}(D)}f(x) \overline{\psi(t,x)} dtdx \right|\\
            &\quad={}
            \left| \int_{\mathbb{R}^3} \Delta_j f(x) \overline{\int_{\mathbb{R}}\widetilde{\Delta_j}e^{-it\lambda^{\pm}(D)}\psi(t,x)dt} dx \right|\\
            &\quad\leq{}
            \| \Delta_j f \|_{L^2}
            \left\| \int_{\mathbb{R}} \widetilde{\Delta_j}e^{-it\lambda^{\pm}(D)}\psi(t) dt \right\|_{L^2},
            \end{split}
        \end{align}
        where we have set $\widetilde{\Delta_j}:=\Delta_{j-1} + \Delta_j + \Delta_{j+1}$.
        For the estimate of the second term of the right{-}hand side, we observe by the Parseval formula and the H\"older inequality that
        \begin{align}\label{pf:str-2}
            \begin{split}
            &\left\| \int_{\mathbb{R}} \widetilde{\Delta_j}e^{-it\lambda^{\pm}(D)}\psi(t) dt \right\|_{L^2}^2\\
            &\quad =
            \int_{\mathbb{R}}
            \int_{\mathbb{R}^3}
            \widetilde{\Delta_j}\psi(t,x)
            \overline{\int_{\mathbb{R}}\widetilde{\Delta_j}e^{i(t-s)\lambda^{\pm}(D)}\psi(s,x)ds}
            dxdt\\
            &\quad \leq
            C
            \| \psi \|_{L^{r'}(\mathbb{R};L^{q'})}
            \left\| \int_{\mathbb{R}}\widetilde{\Delta_j}e^{i(t-s)\lambda^{\pm}(D)}\psi(s)ds \right\|_{L^{r}(\mathbb{R};L^{q})}.
            \end{split}
        \end{align}
        Here, by Lemma \ref{lemm:disp}, we have
        \begin{align}\label{pf:str--2}
           \begin{split}
            &
            \left \| \widetilde{\Delta_j} e^{i (t-s) \lambda^{\pm}(D)} \psi(s) \right \|_{L^{\infty}} \\
            &\quad
            \leq{}
            C \sum_{-1\leq m \leq 1} 2^{3(j+m)} \left( 1 + \min \{2^{3(j+m)}, 1 \} | t-s | \right)^{-1} \| \psi(s) \|_{L^1}\\
            &\quad
            \leq{}
            C 2^{3j} \left( 1 + \min \{2^{3j}, 1 \} | t-s | \right)^{-1} \| \psi(s) \|_{L^1}.
           \end{split}
        \end{align}
        By the Plancherel theorem, we see that
        \begin{align}\label{pf:str--1}
            \begin{split}
            \left \| \widetilde{\Delta_j} e^{i t \lambda^{\pm}(D)} \psi(s) \right \|_{L^2}
            ={}&
            C
            \left \| \sum_{-1 \leq m \leq 1} \phi_{j+m}(\xi) e^{i (t-s) \lambda^{\pm}(\xi)} \widehat{\psi}(s,\xi) \right \|_{L^2_{\xi}}\\
            \leq{}&
            C\| \psi(s) \|_{L^2}.
            \end{split}
        \end{align}
        By using the complex interpolation between \eqref{pf:str--2} and \eqref{pf:str--1}, there holds
        \begin{align}\label{pf:str-0}
            &\left \| \widetilde{\Delta_j} e^{i (t-s) \lambda^{\pm}(D)} \psi(s) \right \|_{L^q} \\
            &\quad
            \leq{}
            C 2^{3(1-\frac{2}{q})j} \left( 1 + \min \{2^{3j}, 1 \} | t-s | \right)^{-(1-\frac{2}{q})} \| \psi(s) \|_{L^{q'}}.
        \end{align}
        Thus, we have
        \begin{align}\label{pf:str-3}
            \begin{split}
            &\left\| \int_{\mathbb{R}}\widetilde{\Delta_j}e^{i(t-s)\lambda^{\pm}(D)}\psi(s)ds \right\|_{L^{r}(\mathbb{R};L^{q})}\\
            &\quad\leq{}
            C2^{3(1-\frac{2}{q})j}
            \left\| \int_{\mathbb{R}}
            \left( 1 + \min \{2^{3j},1\}|t-s| \right)^{-(1-\frac{2}{q})}
            \| \psi(s) \|_{L^{q'}}ds \right\|_{L^{r}(\mathbb{R})}.
            \end{split}
        \end{align}
        In the case of $1/q +1/r < 1/2$, we see by the Hausdorff-Young inequality that
        \begin{align}\label{pf:str-4}
            \begin{split}
            &\left\| \int_{\mathbb{R}}\widetilde{\Delta_j}e^{i(t-s)\lambda^{\pm}(D)}\psi(s)ds \right\|_{L^{r}(\mathbb{R};L^{q})}\\
            &\quad \leq
            C2^{3(1-\frac{2}{q})j}\left\| \left( 1 + \min \{2^{3j},1\}|t| \right)^{-(1-\frac{2}{q})} \right\|_{L^{\frac{r}{2}}(\mathbb{R}_t)}
            \| \psi \|_{L^{r'}(\mathbb{R};L^{q'})}\\
            &\quad \leq
            C2^{3(1-\frac{2}{q})j}\left( \min \{2^{3j},1\} \right)^{-\frac{2}{r}}
            \| \psi \|_{L^{r'}(\mathbb{R};L^{q'})}.
            \end{split}
        \end{align}
        In the case of $1/q +1/r = 1/2$ with $(q,r) {\notin \{} (2,\infty),(\infty,2) {\}}$, it follows from the Hardy--Littlewood--Sobolev inequality that
        \begin{align}\label{pf:str-5}
            \begin{split}
            &\left\| \int_{\mathbb{R}}\widetilde{\Delta_j}e^{i(t-s)\lambda^{\pm}(D)}\psi(s)ds \right\|_{L^{r}(\mathbb{R};L^{q})}\\
            &\quad\leq{}
            C2^{3(1-\frac{2}{q})j}\left( \min \{2^{3j},1\} \right)^{-\frac{2}{r}}
            \left\| \int_{\mathbb{R}}
            |t-s|^{-\frac{2}{r}}
            \| \psi(s) \|_{L^{q'}}ds \right\|_{L^{r}(\mathbb{R})}\\
            &\quad \leq
            C
            2^{3(1-\frac{2}{q})j}
            \left( \min \{2^{3j},1\} \right)^{-\frac{2}{r}}
            \| \psi \|_{L^{r'}(\mathbb{R};L^{q'})}.
            \end{split}
        \end{align}
        Gathering  \eqref{pf:str-1}, \eqref{pf:str-2}, \eqref{pf:str-3}, \eqref{pf:str-4}, and \eqref{pf:str-5}, we obtain
        \begin{align}
            { \left\| \Delta_j e^{it\lambda^{\pm}(D)}f \right\|_{L^r(\mathbb{R};L^q)} }
            \leq
            C
            2^{3(\frac{1}{2}-\frac{1}{q})j}
            \left( \min \{2^{3j},1\} \right)^{-\frac{1}{r}}
            \| \Delta_j f \|_{L^2},
        \end{align}
        which completes the proof.
    \end{proof}
    Now, we are ready to prove Lemma \ref{lemm:lin-str-inv}.
    \begin{proof}[Proof of Lemma \ref{lemm:lin-str-inv}]    
        {Let $F:=(\widetilde{b}_0,\widetilde{v}_0)$ and $G:=(\widetilde{f},\widetilde{g})$.}
        It follows from Lemma \ref{lemm:str} that
        \begin{align}
            {\left\| \Delta_je^{itA(D)}F \right\|_{L^r(\mathbb{R};L^q)}}
            &
            \quad
            \leq{}
            \sum_{\sigma_1,\sigma_2 \in \{\pm\}}
            {\left\| \Delta_je^{\sigma_1it \lambda^{\sigma_2}(D)} P_{\sigma_1,\sigma_2}F \right\|_{L^r(\mathbb{R};L^q)}}\\
            &
            \quad
            \leq{}
            C
            2^{3(\frac{1}{2}-\frac{1}{q})j}
            \left( \min \{2^{3j},1\} \right)^{-\frac{1}{r}}
            \sum_{\sigma_1,\sigma_2 \in \{\pm\}}
            \|  \Delta_j P_{\sigma_1,\sigma_2} F \|_{L^2},
        \end{align}
        where $P_{\sigma_1,\sigma_2}$ denote the projection operators defined in \eqref{P}.
        Since
        the eigenvectors $\{a_{\sigma_1,\sigma_2}(\xi)\}_{\sigma_1,\sigma_2 \in \{\pm\}}$
        are the orthonormal basis of $\mathbb{C}^4$, it holds
        \begin{align}
            \sum_{\sigma_1,\sigma_2 \in \{\pm\}}
            \|  \Delta_j P_{\sigma_1,\sigma_2} F \|_{L^2}
            & \leq{}
            2
            \left(
            \sum_{\sigma_1,\sigma_2 \in \{\pm\}}
            \|  \Delta_j P_{\sigma_1,\sigma_2} F \|_{L^2}^2
            \right)^{\frac{1}{2}}\\
            & ={}
            C
            \left(
            \int_{\mathbb{R}^3}
            \sum_{\sigma_1,\sigma_2 \in \{\pm\}}
            \left|
            \left\langle
            a_{\sigma_1,\sigma_2}(\xi),
            {\phi_j}(\xi)\widehat{F}(\xi)
            \right\rangle_{\mathbb{C}^4}
            \right|^2
            d\xi
            \right)^{\frac{1}{2}}\\
            & ={}
            C
            \left(
            \int_{\mathbb{R}^3}
            \left|{\phi_j}(\xi)\widehat{F}(\xi)\right|^2
            d\xi
            \right)^{\frac{1}{2}} \\
            & =
            C\| \Delta_j F \|_{L^2}.
        \end{align}
        Thus, we have
        \begin{align}\label{hom}
            \left\| \Delta_je^{itA(D)}F \right\|_{L^r(\mathbb{R};L^q)}
            \leq{}
            C
            2^{3(\frac{1}{2}-\frac{1}{q})j}
            \left( \min \{2^{3j},1\} \right)^{-\frac{1}{r}}
            \|  \Delta_j F \|_{L^2}.
        \end{align}
        For the inhomogeneous estimate, since we have
        \begin{align}
            \left\| \Delta_j \int_0^t e^{i(t-\tau)A(D)}G(\tau)d\tau \right\|_{L^q}
            \leqslant{}&
            \int_{\min\{0,t\}}^{\max\{0,t\}} \left\| \Delta_je^{i(t-\tau)A(D)}G(\tau) \right\|_{L^q} d\tau\\
            \leqslant{}&
            \int_{\mathbb{R}} \left\| \Delta_je^{i(t-\tau)A(D)}G(\tau) \right\|_{L^q} d\tau,
        \end{align}
        it follows from \eqref{hom} with $F$ replaced by $G(\tau)$ that
        \begin{align}
            \left\| \Delta_j \int_0^t e^{i(t-\tau)A(D)}G(\tau)d\tau \right\|_{L^r(\mathbb{R};L^q)}
            \leqslant{}&
            \int_{\mathbb{R}} \left\| \Delta_je^{i(t-\tau)A(D)}G(\tau) \right\|_{L^r(\mathbb{R};L^q)} d\tau\\
            ={}&
            \int_{\mathbb{R}} \left\| \Delta_je^{itA(D)}G(\tau) \right\|_{L^r(\mathbb{R};L^q)} d\tau\\
            \leqslant{}&
            C2^{3(\frac{1}{2}-\frac{1}{q})j}
            \left( \min \{2^{3j},1\} \right)^{-\frac{1}{r}}
            \int_{\mathbb{R}} \left\| G(\tau) \right\|_{L^2} d\tau\\
            ={}&
            C2^{3(\frac{1}{2}-\frac{1}{q})j}
            \left( \min \{2^{3j},1\} \right)^{-\frac{1}{r}}
            \| G \|_{L^1(\mathbb{R};L^2)}.
        \end{align}
        Thus, the solution $(\tb,\tv)$ to \eqref{eq:lin-inv-2} satisifes
        \begin{align}
            {\left\| \Delta_j\left(\tb,\tv\right) \right\|_{L^r(\mathbb{R};L^q)}}
            \leq{}&
            C
            2^{3(\frac{1}{2}-\frac{1}{q})j}
            \left( \min \{2^{3j},1\} \right)^{-\frac{1}{r}}
            \|  \Delta_j(\tb_0,\tv_0) \|_{L^2}\\
            &
            +
            C
            2^{3(\frac{1}{2}-\frac{1}{q})j}
            \left( \min \{2^{3j},1\} \right)^{-\frac{1}{r}}
            {
            \|  \Delta_j(\tf,\tg) \|_{L^1(\mathbb{R};L^2)}.}
        \end{align}
        Using the scaling relation \eqref{inv-scaling}, we complete the proof.
    \end{proof}
    Finally, we prove Proposition \ref{prop:lin-str-visc}.
    \begin{proof}[Proof of Proposition \ref{prop:lin-str-visc}]
        Since $(a,u)$ {satisfies} \eqref{eq:lin-inv-1} with $g$ replaced by $g+\mathcal{L}u$, we see that
        \begin{align}\label{pf:prop-str-1}
            \begin{split}
            \| \Delta_j ( a,u ) \|_{L^r( 0,t ; L^q )}
            \leq{}&
            C
            2^{3(\frac{1}{2} - \frac{1}{q} - \frac{1}{r})j}
            |\Omega|^{\frac{2}{r}}
            \varepsilon^{\frac{3}{r}}\\
            &
            \times
            \left(
            \| \Delta_j(a_0, u_0) \|_{L^2}
            +
            \| \Delta_j(f, g + \mathcal{L}u) \|_{L^1( 0,t ; L^2)}
            \right)
            \end{split}
        \end{align}
        for all $j \in \mathbb{Z}$ with $2^j \leq |\Omega| \varepsilon$ and
        \begin{align}\label{pf:prop-str-2}
            \begin{split}
            \| \Delta_j ( a,u ) \|_{L^r( 0,t ; L^q )}
            \leq{}&
            C
            2^{3(\frac{1}{2} - \frac{1}{q})j}
            |\Omega|^{-\frac{1}{r}}\\
            &
            \times
            \left(
            \| \Delta_j(a_0, u_0) \|_{L^2}
            +
            \| \Delta_j(f, g + \mathcal{L}u) \|_{L^1( 0,t ; L^2)}
            \right)
            \end{split}
        \end{align}
        for all $j \in \mathbb{Z}$ with $2^j > |\Omega| \varepsilon$.
        It follows from Lemma \ref{lemm:simple-ene} that 
        \begin{align}\label{pf:prop-str-3}
            \begin{split}
            \| \Delta_j\mathcal{L}u \|_{L^1( 0,t ; L^2)}
            \leqslant{}&
            C2^{2j}T\| u \|_{L^{\infty}(0,t;L^2)}\\
            \leqslant{}&
            C2^{2j}T
            \left(
            \| \Delta_j(a_0, u_0) \|_{L^2}
            +
            \| \Delta_j(f, g ) \|_{L^1( 0,t ; L^2)}
            \right)\\
            \leqslant{}&
            C|\Omega|^2\varepsilon^2T
            \left(
            \| \Delta_j(a_0, u_0) \|_{L^2}
            +
            \| \Delta_j(f, g ) \|_{L^1( 0,t ; L^2)}
            \right)
            \end{split}
        \end{align}
        {for all $j \in \mathbb{Z}$ with $2^j \leqslant |\Omega|\varepsilon$.}
        Lemma \ref{lemm:middle-ene} yields 
        \begin{align}\label{pf:prop-str-3-1}
            \| \Delta_j\mathcal{L}u \|_{L^1( 0,t ; L^2)}
            \leqslant{}&
            C
            \left(
            \| \Delta_j(a_0, u_0) \|_{L^2}
            +
            \| \Delta_j(f, g ) \|_{L^1( 0,t ; L^2)}
            \right)
        \end{align}
        for all $j$ with $2^j \geqslant |\Omega|\varepsilon$.
        By \eqref{pf:prop-str-1}, \eqref{pf:prop-str-2}, \eqref{pf:prop-str-3}, and \eqref{pf:prop-str-3-1}, 
        we obtain \eqref{lin-str-visc-1} and~\eqref{lin-str-visc-2}.

        Next, we prove \eqref{lin-str-visc-3}.
        Let $\widetilde{r} : = (3/2) r$.
        Then, we see by the H\"older inequality for the time integral and \eqref{lin-str-visc-1} that
        \begin{align}\label{pf:prop-str-4}
            \begin{split}
            &\| (a,u) \|_{\widetilde{L^r}( 0,t ; \dB_{q,1}^{\frac{3}{q}-1+\frac{2}{r}} )}^{\ell;|\Omega|\varepsilon}
            ={}
            \| (a,u) \|_{\widetilde{L^{\frac{2}{3}\widetilde{r}}}( 0,t ; \dB_{q,1}^{\frac{3}{q}-1+\frac{3}{\widetilde{r}}} )}^{\ell;|\Omega|\varepsilon}\\
            &\quad 
            \leqslant{}
            T^{\frac{1}{2\widetilde{r}}}
            \| (a,u) \|_{\widetilde{L^{\widetilde{r}}}( 0,t ; \dB_{q,1}^{\frac{3}{q}-1+\frac{3}{\widetilde{r}}} )}^{\ell;|\Omega|\varepsilon}\\
            &\quad 
            \leqslant{}
            CT^{\frac{1}{2\widetilde{r}}}
            |\Omega|^{\frac{2}{\widetilde{r}}}
            \varepsilon^{\frac{3}{\widetilde{r}}}
            \langle  {\Omega^2\varepsilon^2}T \rangle 
            \left(
            \| (a_0,u_0) \|_{\dB_{q,1}^{\frac{1}{2}}}^{\ell;|\Omega|\varepsilon}
            +
            \| (f,g) \|_{L^1(0,t ; \dB_{2,1}^{\frac{1}{2}})}^{\ell;|\Omega|\varepsilon}
            \right)\\
            &\quad 
            \leqslant{}
            C\varepsilon^{\frac{2}{3r}}
            \langle \alpha^2 T \rangle^{2} 
            \left(
            \| (a_0,u_0) \|_{\dB_{q,1}^{\frac{1}{2}}}^{\ell;|\Omega|\varepsilon}
            +
            \| (f,g) \|_{L^1(0,t ; \dB_{2,1}^{\frac{1}{2}})}^{\ell;|\Omega|\varepsilon}
            \right).
            \end{split}
        \end{align}
        It follows from \eqref{lin-str-visc-2} that
        \begin{align}\label{pf:prop-str-5}
            \begin{split}
            &\| (a,u) \|_{\widetilde{L^r}( 0,t ; \dB_{q,1}^{\frac{3}{q}-1+\frac{2}{r}} )}^{m;|\Omega|\varepsilon{,}\alpha}\\
            &\quad
            \leqslant
            C|\Omega|^{-\frac{1}{r}}
            \left(
            \| (a_0,u_0) \|_{\dB_{q,1}^{\frac{1}{2}+\frac{2}{r}}}^{m;|\Omega|\varepsilon,\alpha}
            +
            \| (f,g) \|_{L^1(0,t ; \dB_{2,1}^{\frac{1}{2}+\frac{2}{r}})}^{m;|\Omega|\varepsilon,\alpha}
            \right)\\
            &\quad 
            \leqslant
            C|\Omega|^{-\frac{1}{r}}\alpha^{\frac{2}{r}}
            \left(
            \| (a_0,u_0) \|_{\dB_{q,1}^{\frac{1}{2}}}^{m;|\Omega|\varepsilon,\alpha}
            +
            \| (f,g) \|_{L^1(0,t ; \dB_{2,1}^{\frac{1}{2}})}^{m;|\Omega|\varepsilon,\alpha}
            \right).
            \end{split}
        \end{align}
        Hence, gathering \eqref{pf:prop-str-4} and \eqref{pf:prop-str-5}, we obtain \eqref{lin-str-visc-3} and complete the proof.
    \end{proof}

\section{Local well-posedness on the short time interval}\label{sec:LWP}
    In this section, we prove the local well-posedness of \eqref{eq:NSC-2} on some short time interval.
    \begin{prop}\label{prop:LWP}
        Let $\Omega \in \mathbb{R}$, $\varepsilon>0$, $1 \leq p <6$,
        and
        $(a_0,u_0) \in \dB_{p,1}^{\frac{3}{p}}(\mathbb{R}^3) \times \dB_{p,1}^{\frac{3}{p}-1}(\mathbb{R}^3)^3$ with $\rho_0=1+\varepsilon a_0 >0$.
        Then, 
        there exists a positive time $T_{\Omega,\varepsilon}=T(\mu,\Omega,\varepsilon,a_0,u_0,p)$ such that
        system \eqref{eq:NSC-2} possesses a unique solution $(a,u)$ in the class
        \begin{equation}
        \label{local-class}
            \begin{split}
                &
                a \in {C} ( [0,T_{\Omega,\varepsilon}] ; \dB_{p,1}^{\frac{3}{p}}(\mathbb{R}^3) ),\\
                &
                u \in {C} ( [0,T_{\Omega,\varepsilon}] ; \dB_{p,1}^{\frac{3}{p}-1}(\mathbb{R}^3) )^3 \cap L^1 (0,T_{\Omega,\varepsilon} ; \dB_{p,1}^{\frac{3}{p}+1}(\mathbb{R}^3))^3
            \end{split}
        \end{equation}
        with $\rho = 1 + \varepsilon a >0$ on $[0,T_{\Omega,\varepsilon}] \times \mathbb{R}^3$.
        Moreover, if {$a_0$ satisfies additionally} $a_0 \in \dB_{p,1}^{\frac{3}{p}-1}(\mathbb{R}^3)$, then it holds $ a \in {C}( [0,T_{\Omega,\varepsilon}] ; \dB_{p,1}^{\frac{3}{p}-1}(\mathbb{R}^3) )$.
    \end{prop}
    \begin{proof}
    The proof is based on the Lagrangian coordinate approach \cite{Da-14}.
    Since the proof is standard, we only give the outline of the proof briefly. \par
    Let $X = X (t, y)$ be the flow associated with the velocity field $u$ that is
    the solution to
    \begin{align}
    	X (t, y) = y + \int_0^t u (\tau, X (\tau, y)) \,\mathrm{d} \tau.
    \end{align}
    Let $\overline \rho (t, y) = \rho (t, X (t, y))$ and $\overline u (t, y) = u (t, X (t, y))$. Then $\overline \rho = J^{- 1} \rho_0$ with
    $J = \det D X$ and $\overline u$ solves
    \begin{align}
    	L_{\rho_0} (\overline u) = \rho_0^{- 1} \div (I_1 (\overline u, \overline u) + I_{2,\varepsilon} (\overline u))
    	- \Omega (e_3 \times \overline u)
    \end{align}
    with $L_{\rho_0} (u) = \partial_t u - \rho_0^{- 1} \mathcal L u $ and $I_1 (v, w)$ and $I_{2,\varepsilon} (v)$ are some functions satisfying the estimates:
    \begin{align}
    	\lVert I_1 (v, w) \rVert_{L^1 (0, T; \dot B^\frac3p_{p, 1})}
    	& \leqslant C \left(1 + \lVert a_0 \rVert_{\dot B^\frac3p_{p, 1}}\right)
    	\lVert v \rVert_{L^1 (0, T; \dot B^{\frac3p + 1}_{p, 1})} \lVert w \rVert_{L^1 (0, T; \dot B^{\frac3p + 1}_{p, 1})}, \\
    	\lVert I_{2,\varepsilon} (v) \rVert_{L^1 (0, T; \dot B^\frac3p_{p, 1})}
    	& \leqslant C_\varepsilon T \left(1 + \lVert a_0 \rVert_{\dot B^\frac3p_{p, 1}} \right)
    	\left( 1 + \lVert v \rVert_{L^1 (0, T; \dot B^{\frac3p + 1}_{p, 1})} \right), \\
    	\lVert I_1 (v^2, v^2) - I_1 (v^1, v^1) \rVert_{L^1 (0, T; \dot B^\frac3p_{p, 1})}
    	& \leqslant C_{\rho_0} \lVert (v^1, v^2) \rVert_{L^1 (0, T; \dot B^{\frac3p + 1}_{p, 1})}
    	\lVert \delta v \rVert_{L^1 (0, T; \dot B^{\frac3p + 1}_{p, 1})}, \\
    	\lVert I_{2,\varepsilon} (v^2) - I_{2,\varepsilon} (v^1) \rVert_{L^1 (0, T; \dot B^\frac3p_{p, 1})}
    	& \leqslant C_\varepsilon T \left(1 + \lVert a_0 \rVert_{\dot B^\frac3p_{p, 1}} \right)
    	\lVert \delta v \rVert_{L^1 (0, T; \dot B^{\frac3p + 1}_{p, 1})},
    \end{align}
    where we have set $\delta v = v^2 - v^1$, see \cite{Da-14} for the details. \par
    Set
    \begin{align}
    	E_p (T) = \left\{u \in C ([0, T]; \dot B^{\frac3p - 1}_{p, 1}{(\mathbb{R}^3)})\ ; \
    	\partial_t u, \nabla^2 u \in L^1 (0, T; \dot B^{\frac3p - 1}_{p, 1}{(\mathbb{R}^3)}) \right\}
    \end{align}
    and
    \begin{equation}
        \lVert u \rVert_{E_p (T)} := \lVert u \rVert_{L^\infty (0,T; \dot B^{\frac3p - 1}_{p,1} )} 
        + \lVert (\partial_t u, \nabla^2 u) \rVert_{L^1 (0,T; \dot B^{\frac3p - 1}_{p,1} )}.
    \end{equation}
    We first find a fixed point in $E_p (T)$ for the function $\Phi \colon v \mapsto u$ with $u$ the solution to
    \begin{align}
    	\left\{\begin{aligned}
    		L_{\rho_0} (u) & = \rho_0^{- 1} \div (I_1 (v, v) + I_{2,\varepsilon} (v)) - \Omega (e_3 \times v) , & \quad & t > 0, \, x \in \mathbb R^3, \\
    		u (0, x) & = u_0 (x), & \quad & x \in \mathbb R^3,
    	\end{aligned}\right.
    \end{align}
    provided that $T$ is small enough. Introduce $u_L \in E_p (T)$ that is the solution to
    \begin{align}
    	L_1 u_L = 0, \qquad u (0, x) = u_0 (x).
    \end{align}
    In the following, we show that $\Phi$ has a fixed point in some suitable closed ball $\overline B_{E_p (T)} (u_L, R)$.
    Let $v \in \overline B_{E_p (T)} (u_L, R)$ and define $u = \Phi (v)$. Then, $\widetilde u = u - u_L$ solves
    \begin{align}
    	\left\{\begin{aligned}
    		L_{\rho_0} (\widetilde u) & = \rho_0^{- 1} \div I_\varepsilon (v) + (L_1 - L_{\rho_0}) u_L {- \Omega (e_3 \times v)},
    		& \quad & t > 0, \, x \in \mathbb R^3, \\
    		\widetilde u (0, x) & = 0, & \quad & x \in \mathbb R^3
    	\end{aligned}\right.
    \end{align}
    with $I_\varepsilon (v) := I_1 (v, v) + I_{2,\varepsilon} (v)$. By the maximal $L^1$-regularity result for the Lam\'e system with nonconstant coefficients
    (cf. \cite{Da-14}*{Proposition~3.4}), we see that
    \begin{align}
    	\lVert \widetilde u \rVert_{E_p (T)}
    	& \leqslant C e^{C_{\rho_0} T}  \left( \lVert (L_1 - L_{\rho_0}) u_L \rVert_{L^1 (0, T; \dot B^{\frac3p - 1}_{p, 1})}
    	{+ \lvert\Omega\rvert T \lVert v \rVert_{L^\infty (0, T; \dot B^{\frac3p - 1}_{p, 1})} } \right. \\
    	& \quad \left. + \lVert \rho_0^{- 1} \rVert_{\mathcal M (\dot B^{\frac3p - 1}_{p, 1})} \lVert (I_1 (v, v), I_{2,\varepsilon} (v)) \rVert_{L^1 (0, T; \dot B^\frac3p_{p, 1})}
    	\right),
    \end{align}
    where the multiplier norm $\mathcal M (\dot B^s_{p,1})$ for $\dot B^s_{p,1} (\mathbb R^3)$ is defined by
    \begin{equation}
        \lVert f \rVert_{\mathcal M (\dot B^s_{p, 1})} := \sup_{\lVert \psi \rVert_{\dot B^s_{p,1}}\leqslant 1} \lVert \psi f \rVert_{\dot B^s_{p, 1}}.
    \end{equation}
    Notice that we may confirm that $\rho_0^{- 1}$ belongs to $\mathcal M (\dot B^{\frac3p - 1}_{p,1})$ by
    \begin{align}
	\lVert \psi \rho_0^{- 1} \rVert_{\dot B^{\frac3p - 1}_{p, 1}}
	= \left\lVert \psi \bigg(\frac{\varepsilon a_0}{1 + \varepsilon a_0} - 1 \bigg) \right\rVert_{\dot    B^{\frac3p - 1}_{p, 1}}
	\leqslant C_{a_0, \varepsilon} \lVert \psi \rVert_{\dot B^{\frac3p - 1}_{p, 1}} \left(1 + \lVert a_0 \rVert_{\dot B^\frac3p_{p, 1}} \right)
    \end{align}
    as follows from the product and composition estimates. Namely, we observe that
    \begin{equation}
        \lVert \rho_0^{- 1} \rVert_{\mathcal M (\dot B^{\frac3p - 1}_{p, 1})} \leqslant C_{a_0, \varepsilon} \left(1 + \lVert a_0 \rVert_{\dot B^\frac3p_{p, 1}} \right).
    \end{equation}
    We now assume that $T$ and $R$ are chosen such that for a small enough constant $c$ it holds
    \begin{align}
    	\int_0^T \lVert {\nabla} v(t) \rVert_{\dot B^\frac3p_{p, 1}} \,d t \leqslant c.
    \end{align}
    Decompose $v$ into $\widetilde v + u_L$. Since we have
    \begin{align}
    	\lVert (L_1 - L_{\rho_0}) u_L \rVert_{L^1 (0, T; \dot B^{\frac3p - 1}_{p, 1})}
    	\leqslant C \lVert a_0 \rVert_{\dot B^\frac3p_{p, 1}} \left(1 + \lVert a_0 \rVert_{\dot B^\frac3p_{p, 1}}\right)
    	\lVert u_L \rVert_{L^1 (0, T; \dot B^{\frac3p + 1}_{p, 1})},
    \end{align}
    see \cite{Da-14}*{(3.25)}, we arrive at
    \begin{multline}
            \lVert \widetilde u \rVert_{E_p (T)} 
    	\leqslant C_0 e^{C_{\rho_0} T} \left(1 + \lVert a_0 \rVert_{\dot B^\frac3p_{p, 1}}\right)^2
    	\left\{
            \left( T + \lVert a_0 \rVert_{\dot B^\frac3p_{p, 1}} \lVert u_L \rVert_{L^1 (0, T; \dot B^{\frac3p + 1}_{p, 1})} \right) \right. \\
    	\qquad + \left(R + \lVert u_L \rVert_{L^1 (0, T; \dot B^{\frac3p + 1}_{p, 1})} \right)  
            {\lVert u_L \rVert_{L^1 (0, T; \dot B^{\frac3p + 1}_{p, 1})} }
    	+ R^2 \\
    	\qquad \left. + \lvert\Omega\rvert T 
    	{\left(R + \lVert u_L \rVert_{L^\infty (0, T; \dot B^{\frac3p - 1}_{p, 1})} \right)} \right\}
    \end{multline}
    with some positive constant $C_0$ depending on $\lVert a_0 \rVert_{L^\infty}$ and $\varepsilon$ but independent of $T$ as well as $\Omega$.
    Here, we have used the fact that $v \in \overline B_{E_p (T)} (u_L, R)$ implies $\widetilde v \in \overline B_{E_p (T)} (0, R)$.
    We first choose $R$ so that for a small enough constant $\eta \in (0, {1 \slash 7})$,
    \begin{align}
    	{2 C_0} \left(1 + \lVert a_0 \rVert_{\dot B^\frac3p_{p, 1}}\right)^2 R \leqslant \eta,
    \end{align}
    and take $T$ so that
    \begin{gather}
    		C_{\rho_0} T \leqslant \log 2, \quad T \leqslant {R^2}, \quad
    		\lVert a_0 \rVert_{\dot B^\frac3p_{p, 1}} \lVert u_L \rVert_{L^1 (0, T; \dot B^{\frac3p + 1}_{p, 1})} \leqslant R^2, \\
    		\lVert u_L \rVert_{L^\infty (0, T; \dot B^{\frac3p - 1}_{p, 1})\cap L^1 (0, T; \dot B^{\frac3p + 1}_{p, 1})} \leqslant R, \quad
    		{2} \lvert\Omega\rvert T \leqslant {R}.
    \end{gather}
    We then see that $\Phi$ is a self-map on the closed ball $\overline B_{E_p (T)} (u_L, R)$. \par
    To verify that the map $\Phi$ is contractive, we consider two velocity fields $v^1, v^2 \in \overline B_{E_p (T)} (u_L, R)$
    and set $u^\ell = \Phi (v^\ell)$, $\ell = 1, 2$. Let $\delta u = u^2 - u^1$. Then, we have
    \begin{align}
    	L_{\rho_0} \delta u = \rho_0^{- 1} \div ((I_1 (v^2, v^2) - I_1 (v^1, v^1)) + (I_2 (v^2) - I_2 (v^1))) - \Omega e_3 \times ({\delta v}).
    \end{align}
    From the estimates for $I_1$ and $I_2$, we end up with
    \begin{align}
    	\lVert \delta u \rVert_{E_p (T)}
    	& \leqslant {2 C_0} \left\{\left(1 + \lVert a_0 \rVert_{\dot B^\frac3p_{p, 1}}\right)^2 
            \left(T + \lVert (v^1, v^2) \rVert_{L^1 (0, T; \dot B^\frac3p_{p, 1})} \right)
    	\lVert \delta v \rVert_{L^1 (0, T; \dot B^\frac3p_{p, 1} ) } \right. \\
    	& \quad \left. + \lvert\Omega\rvert T \lVert {\delta v} \rVert_{L^\infty (0, T; \dot B^{\frac3p - 1}_{p, 1})} \right\}
    \end{align}
    as follows from $C_{\rho_0} T \leqslant \log 2$. As $v^1, v^2 \in \overline B_{E_p (T)} (u_L, R)$, if we additionally suppose that {$R$} satisfies
    {$R \leqslant \min {\{}1, (14 C_0)^{- 1}{\}}$}
    we observe
    \begin{align}
    	\lVert \delta u \rVert_{E_p (T)} \leqslant \frac12 \lVert \delta v \rVert_{E_p (T)}.
    \end{align}
    This shows that $\Phi$ is contraction mapping, and hence it admits a unique fixed point in $\overline B_{E_p (T)} (u_L, R)$. \par
    The regularity of $a$ follows from
    \begin{align}
    	a = (J_u^{- 1} - 1) a_0 + a_0,
    \end{align}
    where $J_u = \det (D X_u)$. It is known that $J_u^{- 1} - 1$ belongs to $C ([0, T]; \dot B^\frac3p_{p, 1}(\mathbb R^3))$ so that $a$ admits the same regularity as well. In particular, the density is bounded away from 0 on $[0, T]$ due to $\dot B^\frac3p_{p, 1} (\mathbb R^3)\hookrightarrow L^\infty(\mathbb R^3)$ by taking $T$ smaller if necessary. 
    {Furthermore, if we additionally suppose that $a_0 \in \dot B^{\frac3p - 1}_{p,1} (\mathbb R^3)$, it follows that
    \begin{align}
        \lVert a (t) \rVert_{\dot B^{\frac3p - 1}_{p, 1}} & \leqslant
        \lVert (J^{- 1}_u - 1) a_0 \rVert_{\dot B^{\frac3p - 1}_{p, 1}}
        + \lVert a_0 \rVert_{\dot B^{\frac3p - 1}_{p, 1}} \\
        & \leqslant C \left(\lVert J^{- 1}_u - 1 \rVert_{\dot B^\frac3p_{p, 1}} + 1 \right) \lVert a_0 \rVert_{\dot B^{\frac3p - 1}_{p, 1}},
    \end{align}
    since the product maps {from} $\dot B^{\frac3p - 1}_{p,1} (\mathbb R^3) \times \dot B^\frac3p_{p, 1} (\mathbb R^3)$ {to} $\dot B^{\frac3p - 1}_{p,1} (\mathbb R^3)$ whenever $1 \leqslant p < 6$.
    Thus{,} we find that $a$ enjoys $a \in C ([0, T] ; \dot B^{\frac3p - 1}_{p, 1} (\mathbb R^3))$, where the continuity with respect to $t$ follows from the estimate \cite{Da-14}*{(A.20)}.}
    \par
    It remains to verify the uniqueness {of} the solution to \eqref{eq:NSC-2} and continuity of the flow map.
    We now consider two couples $(\rho_0^1, {u^1_0})$ and $(\rho_0^2, {u^2_0})$ of data fulfilling the assumptions of Theorem~\ref{prop:LWP}.
    We denote by $(\rho^1, u^1)$ and $(\rho^2, u^2)$ two solutions in~{\eqref{local-class}} corresponding to the data $(\rho^1_0, u^1_0)$ and $(\rho^2_0, u^2_0)$, respectively.
    Since it holds
    \begin{align}
    	L_{\rho_0^2} u^2 - L_{\rho_0^1} u^1 = L_{\rho^1_0} (\delta u) + (L_{\rho_0^2} - L_{\rho_0^1}) u^2,
    \end{align}
    we see that
    \begin{equation}
    	\label{eq:Lagrange-continuity}
    	\begin{split}
    		L_{\rho_0^1} (\delta u) & = (L_{\rho_0^1} - L_{\rho_0^2}) u^2 + \frac1\varepsilon e_3 \times (\delta u) + (\rho_0^1)^{- 1} \div (I^2 (u^2) - I^2 (u^1)) \\
    		& \quad + ((\rho^2_0)^{- 1} - (\rho_0^1)^{- 1}) \div I^2 (u^2) + (\rho_0^1)^{- 1} \div (I^2 (u^2) - I^1 (u^1)).
    	\end{split}
    \end{equation}
    Here, $I^1$ and $I^2$ correspond to the quantities defined by replacing $\rho_0$ by $\rho_0^1$ and $\rho^2_0$, respectively, in the definition of $I$.
    Then, employing the argument given in \cite{Da-14}*{pp.779{--}780}, we see that
    \begin{align}
    	\lVert \delta u \rVert_{E_p (t)} \leqslant C \left( \lVert \delta u_0 \rVert_{\dot B^{\frac3p - 1}_{p, 1}} + \lVert \delta \rho_0 \rVert_{\dot B^\frac3p_{p, 1}} \right)
    \end{align}
    provided that $\delta \rho_0 := \rho_0^2 - \rho^1_0$, $\delta u_0 :=  u_0^2 - u^1_0$, and $t$ are small enough.
    Concerning the density, we use $\delta a = J_{u^1}^{- 1} \delta a_0 + (J_{u^2}^{- 1} - J^{- 1}_{u^1}) a^2_0$, which yields
    \begin{align}
    	\lVert \delta a (t) \rVert_{\dot B^\frac3p_{p, 1}} \leqslant C \left(1 + \lVert u^1 \rVert_{L^1 (0, t; \dot B^{\frac3p + 1}_{p, 1})}\right)
    	\lVert \delta a_0 \rVert_{\dot B^\frac3p_{p, 1}} \lVert \delta u \rVert_{L^1 (0, t; \dot B^{\frac3p + 1}_{p, 1})}
    \end{align}
    for all $t \in [0, T]$. 
    Hence, the proof is complete.
    \end{proof}
\section{Proof of Theorem \ref{thm}}\label{sec:pf}
   {In this section, we }prove Theorem \ref{thm}.
    We begin with the definition of key norms in our calculus.
    \begin{df}
        Let $\alpha>0$, $0< \varepsilon < \beta_0/\alpha$, and $1 \leq q,r \leq \infty$.
        We first define the norm of initial data by 
        \begin{align}
            \| (a_0,u_0) \|_{D_{\varepsilon}}
            :=
            \| (a_0,u_0) \|_{\dB_{2,1}^{\frac{1}{2}}}^{\ell;\frac{\beta_0}{\varepsilon}}
            +
            \| (\varepsilon a_0,u_0) \|_{\dB_{2,1}^{\frac{3}{2}} \times \dB_{2,1}^{\frac{1}{2}} }^{h;\frac{\beta_0}{\varepsilon}}.
        \end{align}
        Next, we define the energy norm by
        \begin{align}
        \| (a,u) \|_{E_{\varepsilon}(t)}
        :={}&
        \| (a,u) \|_{\widetilde{L^{\infty}}(0,t;\dB_{2,1}^{\frac{1}{2}}) \cap L^1(0,t;\dB_{2,1}^{\frac{5}{2}})}^{\ell;\frac{\beta_0}{\varepsilon}}\\
        &
        +
        \varepsilon
        \| a \|_{{L^{\infty}}(0,t;\dB_{2,1}^{\frac{3}{2}})}^{h;\frac{\beta_0}{\varepsilon}}
        +
        \frac{1}{\varepsilon}
        \| a \|_{L^1(0,t;\dB_{2,1}^{\frac{3}{2}})}^{h;\frac{\beta_0}{\varepsilon}}
        +
        \| u \|_{{L^{\infty}}( 0,t ;\dB_{2,1}^{\frac{1}{2}}) \cap L^1( 0,t ;\dB_{2,1}^{\frac{5}{2}})}^{h;\frac{\beta_0}{\varepsilon}}.
        \end{align}
        We also define the auxiliary norm by
        \begin{align}
        \| (a,u) \|_{A^{q,r}_{\varepsilon,{\alpha}}(t)}
        :={}&
        \| (a,u) \|_{{L^r}(0,t;\dB_{q,1}^{\frac{3}{q}-1+\frac{2}{r}})}^{\ell;{\alpha}}
        +
        \| (a,u) \|_{{L^{\infty}}(0,t;\dB_{q,1}^{\frac{3}{q}-1}) \cap L^1(0,t;\dB_{q,1}^{\frac{3}{q}+1})}^{m;{\alpha}{,}\frac{\beta_0}{\varepsilon}}\\
        &
        +
        \varepsilon
        \| a \|_{{L^{\infty}}(0,t;\dB_{q,1}^{\frac{3}{q}})}^{h;\frac{\beta_0}{\varepsilon}}
        +
        \frac{1}{\varepsilon}
        \| a \|_{L^1(0,t;\dB_{q,1}^{\frac{3}{q}})}^{h;\frac{\beta_0}{\varepsilon}}
        +
        \| u \|_{{L^{\infty}}( 0,t ;\dB_{q,1}^{\frac{3}{q}-1}) \cap L^1( 0,t ;\dB_{q,1}^{\frac{3}{q}+1})}^{h;\frac{\beta_0}{\varepsilon}}
        \end{align}
        {
        and
        \begin{align}
        \mathcal{A}^{q,r}_{\varepsilon,{\alpha}}[a,u](t)
        :={}&
        \alpha \varepsilon \| (a,u) \|_{E_{\varepsilon}(t)} 
        + 
        \| (a,u) \|_{A^{q,r}_{\varepsilon,{\alpha}}(t)} 
        +
        \| (a,u) \|_{E_{\varepsilon}(t)}^{\frac{r-2}{r-1}}
        \| (a,u) \|_{A^{q,r}_{\varepsilon,{\alpha}}(t)}^{\frac{1}{r-1}}.
        \end{align}
        }
    \end{df}
    \begin{rem}
    Here, the roles of $E_{\varepsilon}(t)$-norm, $A_{\varepsilon,{\alpha}}^{q
    ,r}(t)$-norm{, and $\mathcal{A}^{q,r}_{\varepsilon,{\alpha}}[a,u](t)$} are the following:
    The $A_{\varepsilon,{\alpha}}^{q,r}(t)$-norm {and $\mathcal{A}^{q,r}_{\varepsilon,{\alpha}}[a,u](t)$} may be squeezed provided that $\varepsilon$ and $\delta$ are small and $\Omega$ is large {due to the Strichartz estimates established in the previous section}.
    Although a priori estimate for the energy $E_{\varepsilon}(t)$-norm is not small {for large initial data}, this a priori estimate may be closed by using the a priori estimate for {$\mathcal{A}^{q,r}_{\varepsilon,{\alpha}}[a,u](t)$}.
    {This idea is inspired by the first author's previous work \cite{Fu-22}.}
    \end{rem}

    In the next lemma, we prepare the nonlinear estimates in terms of the quantities defined above.
    \begin{lemm}\label{lemm:EA}
        {Let $\varepsilon,{\alpha},\mu >0$ and let $q$, $r$ satisfy 
        \begin{align}
            2 < q < 4,\qquad
            2 < r < \infty,\qquad
            \frac{2}{r} \leqslant \frac{3}{q} - \frac{1}{2}.
        \end{align}}
        Then, there exists a positive constant $C=C({\mu,}q,r)$ such that
        \begin{align}
            &
            \begin{aligned}\label{nonlin-1}
            &
            \| ( \div(au),N_{\varepsilon}[a,u] ) \|_{L^1( 0,t ; \dB_{2,1}^{\frac{1}{2}} )}^{\ell;\frac{\beta_0}{\varepsilon}}\\
            &\quad\leq{}
            {C 
            \mathcal{A}^{q,r}_{\varepsilon,{\alpha}}[a,u](t)
            \| (a,u) \|_{A^{q,r}_{\varepsilon,{\alpha}}(t)}},
            \end{aligned}\\
            &
            \begin{aligned}\label{nonlin-2}
            &
            \varepsilon \sum_{j \in \mathbb Z} 
            2^{\frac{3}{2}j}
            \| [u \cdot \nabla, \Delta_j]a \|_{L^1( 0,t ; L^2)}
            +
            \varepsilon
            \| (\div u)a \|_{L^1( 0,t ; \dB_{2,1}^{\frac{1}{2}} )}\\
            &\quad 
            +
            \varepsilon
            \left\| \| \div u \|_{L^{\infty}}\| a \|_{\dB_{2,1}^{\frac{3}{2}}} \right\|_{L^1(0,t)}
            +
            \| ( \div(au),N_{\varepsilon}[a,u] ) \|_{L^1(0,t;\dB_{2,1}^{\frac{1}{2}})}\\
            &\qquad 
            \leqslant
            {C
            \mathcal{A}^{q,r}_{\varepsilon,{\alpha}}[a,u](t)
            \| (a,u) \|_{E_{\varepsilon}(t)}},
            \end{aligned}\\
            &
            \begin{aligned}\label{nonlin-3}
            &
            \varepsilon \sum_{j \in \mathbb Z} 
            2^{\frac{3}{q}j}
            \| [u \cdot \nabla, \Delta_j]a \|_{L^1( 0,t ; L^q)}
            +
            \varepsilon
            \| (\div u)a \|_{L^1( 0,t ; \dB_{q,1}^{\frac{3}{q}-1} )}\\
            &\quad 
            +
            \varepsilon
            \left\| \| \div u \|_{L^{\infty}}\| a \|_{\dB_{q,1}^{\frac{3}{q}}} \right\|_{L^1(0,t)}
            +
            \| ( \div(au),N_{\varepsilon}[a,u] ) \|_{L^1(0,t;\dB_{q,1}^{\frac{3}{q}-1})}\\
            &\qquad 
            \leqslant
            {C
            \mathcal{A}^{q,r}_{\varepsilon,{\alpha}}[a,u](t)
            \| (a,u) \|_{A^{q,r}_{\varepsilon,{\alpha}}(t)}}
            \end{aligned}
        \end{align}
        for all {$t>0$ and} $(a,u) \in E_{\varepsilon}(t)$ {satisfying 
        \begin{align}
            \varepsilon \| a \|_{L^{\infty}(0,t;L^{\infty}\cap \dB_{q,1}^{\frac{3}{q}})}
            \leqslant 
            \frac{1}{2}.
        \end{align}}
    \end{lemm}
    \begin{proof}
        We first prove \eqref{nonlin-1}.
        It follows from \cite{Fu-22}*{Section 4} that 
        \begin{align}\label{pf:nonlin-1-1}
            \begin{split}
            &
            \| ( \div(au),N_{\varepsilon}[a,u] ) \|_{L^1( 0,t ; \dB_{2,1}^{\frac{1}{2}} )}^{\ell;\frac{\beta_0}{\varepsilon}}\\
            &\quad 
            \leqslant
            C 
            \varepsilon
            \| a \|_{L^{\infty}( 0,t ; \dB_{q,1}^{\frac{3}{q}})}
            \| u \|_{L^1( 0,t ; \dB_{q,1}^{\frac{3}{q}+1})}^{h;\frac{\beta_0}{\varepsilon}}
            +
            C\| (a,u) \|_{L^2(0,t ; \dB_{q,1}^\frac{3}{q})}^2\\
            &
            \qquad
            +
            C
            \| (a,u) \|_{L^r(0,t;\dB_{q,1}^{\frac{3}{q}-1+\frac{2}{r}})}
            \left(
            \| a \|_{L^{r'}(0,t;\dB_{q,1}^{\frac{3}{q}-1+\frac{2}{r'}})}^{\ell;\frac{4\beta_0}{\varepsilon}}
            +
            \| u \|_{L^{r'}(0,t;\dB_{q,1}^{\frac{3}{q}-1+\frac{2}{r'}})}
            \right).
            \end{split}
        \end{align}
        It follows from the Bernstein inequality and the interpolation inequality that 
        \begin{align}
            &
            \begin{aligned}\label{pf:nonlin-1-2}
            \varepsilon
            \| a \|_{L^{\infty}( 0,t ; \dB_{q,1}^{\frac{3}{q}})}
            \leqslant{}&
            C
            \alpha 
            \varepsilon
            \| a \|_{L^{\infty}( 0,t ; \dB_{2,1}^{\frac{1}{2}})}^{\ell;\alpha}
            +
            \beta_0
            \| a \|_{L^{\infty}( 0,t ; \dB_{q,1}^{\frac{3}{q}-1})}^{m;\alpha,\frac{\beta_0}{\varepsilon}}
            +
            \varepsilon
            \| a \|_{L^{\infty}( 0,t ; \dB_{q,1}^{\frac{3}{q}})}^{h;\frac{\beta_0}{\varepsilon}}\\
            \leqslant{}&
            C
            \left( \alpha \varepsilon \| (a,u) \|_{E_{\varepsilon}(t)} + \| (a,u) \|_{A^{q,r}_{\varepsilon,{\alpha}}(t)} \right),
            \end{aligned}\\
            &
            \| u \|_{L^1( 0,t ; \dB_{q,1}^{\frac{3}{q}+1})}^{h;\frac{\beta_0}{\varepsilon}}
            \leqslant
            \| (a,u) \|_{A^{q,r}_{\varepsilon,{\alpha}}(t)}, \label{pf:nonlin-1-3}
        \end{align}
        and
        \begin{align}\label{pf:nonlin-1-4}
            \begin{split}
            \| (a,u) \|_{L^r( 0,t ; \dB_{q,1}^{\frac{3}{q}-1+\frac{2}{r}} ) }
            \leqslant{}&
            \| (a,u) \|_{L^r( 0,t ; \dB_{q,1}^{\frac{3}{q}-1+\frac{2}{r}} ) }^{\ell;\alpha}
            +    
            \| (a,u) \|_{L^r( 0,t ; \dB_{q,1}^{\frac{3}{q}-1+\frac{2}{r}} ) }^{m;\alpha;\frac{\beta_0}{\varepsilon}}
            \\
            &
            +
            C
            \left( \frac{\beta_0}{\varepsilon} \right)^{-1+\frac{2}{r}}
            \| a \|_{L^r( 0,t ; \dB_{q,1}^{\frac{3}{q}})}^{h;\frac{\beta_0}{\varepsilon}}\\
            &
            {
            +
        \| u \|_{{L^{\infty}}( 0,t ;\dB_{q,1}^{\frac{3}{q}-1}) \cap L^1( 0,t ;\dB_{q,1}^{\frac{3}{q}+1})}^{h;\frac{\beta_0}{\varepsilon}}}\\
            \leqslant{}&
            \| (a,u) \|_{{L^r}( 0,t ; \dB_{q,1}^{\frac{3}{q}-1+\frac{2}{r}} ) }^{\ell;\alpha}
            +    
            \| (a,u) \|_{{L^{\infty}}( 0,t ; \dB_{q,1}^{\frac{3}{q}-1} ) \cap L^1( 0,t ; \dB_{q,1}^{\frac{3}{q}+1} ) }^{m;\alpha;\frac{\beta_0}{\varepsilon}}
            \\
            &
            +
            C
            \varepsilon
            \| a \|_{{L^{\infty}}( 0,t ; \dB_{q,1}^{\frac{3}{q}})}^{h;\frac{\beta_0}{\varepsilon}}
            +
            \frac{C}{\varepsilon}
            \| a \|_{L^1( 0,t ; \dB_{q,1}^{\frac{3}{q}})}^{h;\frac{\beta_0}{\varepsilon}}\\
            &
            {
            +
        \| u \|_{{L^{\infty}}( 0,t ;\dB_{q,1}^{\frac{3}{q}-1}) \cap L^1( 0,t ;\dB_{q,1}^{\frac{3}{q}+1})}^{h;\frac{\beta_0}{\varepsilon}}}\\
            \leqslant{}&
            C
            \| (a,u) \|_{A^{q,r}_{\varepsilon,{\alpha}}(t)}.
            \end{split}
        \end{align}
        Using \eqref{pf:nonlin-1-4}, we have
        \begin{align}\label{pf:nonlin-1-5}
            \begin{split}
            \| (a,u) \|_{L^2( 0,t ; \dB_{2,1}^{\frac{3}{q}} )}
            \leqslant{}&
            \| (a,u) \|_{L^r( 0,t ; \dB_{q,1}^{\frac{3}{q}-1+\frac{2}{r}} ) }^{\frac{r}{2(r-1)}}
            \| (a,u) \|_{L^1( 0,t ; \dB_{q,1}^{\frac{3}{q}+1} ) }^{\frac{r-2}{2(r-1)}}\\
            \leqslant{}&
            C
            \| (a,u) \|_{L^r( 0,t ; \dB_{q,1}^{\frac{3}{q}-1+\frac{2}{r}} ) }^{\frac{r}{2(r-1)}}
            \| (a,u) \|_{L^1( 0,t ; \dB_{2,1}^{\frac{5}{2}} ) }^{\frac{r-2}{2(r-1)}}\\
            \leqslant{}&
            C
            \| (a,u) \|_{A^{q,r}_{\varepsilon,{\alpha}}(t)}^{\frac{r}{2(r-1)}}
            \| (a,u) \|_{E_{\varepsilon}(t)}^{\frac{r-2}{2(r-1)}}
            \\
            \leqslant{}&
            {C
            \| (a,u) \|_{A^{q,r}_{\varepsilon,{\alpha}}(t)}
            +
            C
            \| (a,u) \|_{A^{q,r}_{\varepsilon,{\alpha}}(t)}^{\frac{1}{r-1}}
            \| (a,u) \|_{E_{\varepsilon}(t)}^{\frac{r-2}{r-1}}}
            \end{split}
        \end{align}
        and
        \begin{align}\label{pf:nonlin-1-6}
            \begin{split}
            &
            \| a \|_{L^{r'}(0,t;\dB_{q,1}^{\frac{3}{q}-1+\frac{2}{r'}})}^{\ell;\frac{4\beta_0}{\varepsilon}}
            +
            \| u \|_{L^{r'}(0,t;\dB_{q,1}^{\frac{3}{q}-1+\frac{2}{r'}})}\\
            &\quad 
            \leqslant{}
            \| a \|_{{L^{r'}}(0,t;\dB_{q,1}^{\frac{3}{q}-1+\frac{2}{r'}})}^{\ell;\frac{\beta_0}{\varepsilon}}
            +
            \| u \|_{{L^{r'}}(0,t;\dB_{q,1}^{\frac{3}{q}-1+\frac{2}{r'}})}
            +
            \| a \|_{{L^{r'}}(0,t;\dB_{q,1}^{\frac{3}{q}-1+\frac{2}{r'}})}^{m;\frac{\beta_0}{\varepsilon},\frac{4\beta_0}{\varepsilon}}\\
            &\quad 
            \leqslant{}
            \| a \|_{L^r( 0,t ; \dB_{q,1}^{\frac{3}{q}-1+\frac{2}{r}} ) }^{\frac{1}{r-1}}
            \left(
            \| a \|_{L^1( 0,t ; \dB_{q,1}^{\frac{3}{q}+1} ) }^{\ell;\frac{\beta_0}{\varepsilon}}
            \right)^{\frac{r-2}{r-1}}
            +
            \| u \|_{L^r( 0,t ; \dB_{q,1}^{\frac{3}{q}-1+\frac{2}{r}} ) }^{\frac{1}{r-1}}
            \| u \|_{L^1( 0,t ; \dB_{q,1}^{\frac{3}{q}+1} ) }^{\frac{r-2}{r-1}}\\
            &
            \qquad
            +
            C
            \left( \frac{\beta_0}{\varepsilon} \right)^{-1+\frac{2}{r'}}
            \| a \|_{{L^{r'}}(0,t;\dB_{q,1}^{\frac{3}{q}})}^{h;\frac{\beta_0}{\varepsilon}}\\
            &\quad 
            \leqslant{}
            \| (a,u) \|_{L^r( 0,t ; \dB_{q,1}^{\frac{3}{q}-1+\frac{2}{r}} ) }^{\frac{1}{r-1}}
            \left(
            \| a \|_{L^1( 0,t ; \dB_{2,1}^{\frac{5}{2}} ) }^{\ell;\frac{\beta_0}{\varepsilon}}
            +
            \| u \|_{L^1( 0,t ; \dB_{2,1}^{\frac{5}{2}} ) }
            \right)^{\frac{r-2}{r-1}}\\
            &
            \qquad
            +
            C
            \varepsilon
            \| a \|_{{L^{\infty}}( 0,t ; \dB_{q,1}^{\frac{3}{q}})}^{h;\frac{\beta_0}{\varepsilon}}
            +
            \frac{C}{\varepsilon}
            \| a \|_{L^1( 0,t ; \dB_{q,1}^{\frac{3}{q}})}^{h;\frac{\beta_0}{\varepsilon}}\\
            &\quad 
            \leqslant
            C
            \| (a,u) \|_{A^{q,r}_{\varepsilon,{\alpha}}(t)}^{\frac{1}{r-1}}
            \| (a,u) \|_{E_{\varepsilon}(t)}^{\frac{r-2}{r-1}}
            +
            C
            \| (a,u) \|_{A^{q,r}_{\varepsilon,{\alpha}}(t)}.
            \end{split}
        \end{align}
        Combining 
        \eqref{pf:nonlin-1-1}, 
        \eqref{pf:nonlin-1-2}, 
        \eqref{pf:nonlin-1-3}, 
        \eqref{pf:nonlin-1-4}, 
        \eqref{pf:nonlin-1-5}, and 
        \eqref{pf:nonlin-1-6},
        we obtain \eqref{nonlin-1}.
        
        Next, we show \eqref{nonlin-2}.
        It follows from \cite{Fu-22}*{Section 4} that
        \begin{align}
            &
            \varepsilon \sum_{j \in \mathbb Z} 
            2^{\frac{3}{2}j}
            \| [u \cdot \nabla, \Delta_j]a \|_{L^1( 0,t ; L^2)}
            +
            \varepsilon
            \| (\div u)a \|_{L^1( 0,t ; \dB_{2,1}^{\frac{1}{2}} )}\\
            &+
            \varepsilon
            \left\| \| \div u \|_{L^{\infty}}\| a \|_{\dB_{2,1}^{\frac{3}{2}}} \right\|_{L^1(0,t)}
            +
            \| ( \div(au),N_{\varepsilon}[a,u] ) \|_{L^1(0,t;\dB_{2,1}^{\frac{1}{2}})}\\
            &\quad 
            \leqslant
            C\varepsilon \| a \|_{L^{\infty}( 0,t ; \dB_{2,1}^{\frac{3}{2}})}\| u \|_{L^1( 0,t ; \dB_{q,1}^{\frac{3}{q}+1} )}^{h;\frac{\beta_0}{\varepsilon}}
            +
            C\varepsilon \| a \|_{L^{\infty}( 0,t ; \dB_{q,1}^{\frac{3}{q}})}\| u \|_{L^1( 0,t ; \dB_{2,1}^{\frac{5}{2}} )}^{h;\frac{\beta_0}{\varepsilon}}\\
            &\qquad
            +
            C \| (a,u) \|_{L^2(0,t; \dB_{2,1}^{\frac{3}{2}})} \| (a,u) \|_{L^2(0,t; \dB_{q,1}^{\frac{3}{q}})}.
        \end{align}
        Using 
        \eqref{pf:nonlin-1-2}, 
        \eqref{pf:nonlin-1-3}, 
        \eqref{pf:nonlin-1-5}, 
        and 
        \begin{align}
            &
            \| (a,u) \|_{L^2(0,t; \dB_{2,1}^{\frac{3}{2}})}
            +
            \varepsilon \| a \|_{L^{\infty}( 0,t ; \dB_{2,1}^{\frac{3}{2}})}
            +
            \| u \|_{L^1( 0,t ; \dB_{2,1}^{\frac{5}{2}} )}^{h;\frac{\beta_0}{\varepsilon}}\\
            &\quad 
            \leqslant
            \| (a,u) \|_{L^2(0,t; \dB_{2,1}^{\frac{3}{2}})}^{\ell;\frac{\beta_0}{\varepsilon}}
            +
            \| (a,u) \|_{L^2(0,t; \dB_{2,1}^{\frac{3}{2}})}^{h;\frac{\beta_0}{\varepsilon}}\\
            &\qquad
            +
            \varepsilon \| a \|_{L^{\infty}( 0,t ; \dB_{2,1}^{\frac{3}{2}})}
            +
            \| u \|_{L^1( 0,t ; \dB_{2,1}^{\frac{5}{2}} )}^{h;\frac{\beta_0}{\varepsilon}}\\
            &\quad 
            \leqslant
            C
            \| (a,u) \|_{{L^{\infty}}(0,t; \dB_{2,1}^{\frac{1}{2}}) \cap L^1(0,t; \dB_{2,1}^{\frac{5}{2}})}^{\ell;\frac{\beta_0}{\varepsilon}}\\
            &\qquad
            +
            C\varepsilon
            \| a \|_{{L^{\infty}}(0,t; \dB_{2,1}^{\frac{3}{2}})}^{h;\frac{\beta_0}{\varepsilon}}
            +
            \frac{C}{\varepsilon}
            \| a \|_{L^1(0,t; \dB_{2,1}^{\frac{3}{2}})}^{h;\frac{\beta_0}{\varepsilon}}
            +
            C
            \| u \|_{{L^{\infty}}(0,t; \dB_{2,1}^{\frac{1}{2}}) \cap L^1(0,t; \dB_{2,1}^{\frac{5}{2}})}^{h;\frac{\beta_0}{\varepsilon}}\\
            &\quad 
            =
            C
            \| (a,u) \|_{E_{\varepsilon}(t)},
        \end{align}
        we obtain \eqref{nonlin-2}.
        
        Finally, we prove \eqref{nonlin-3}.
        It follows from 
        \cite{Fu-22}*{Section 4},
        \eqref{pf:nonlin-1-2}, 
        \eqref{pf:nonlin-1-3}, 
        and
        \eqref{pf:nonlin-1-5}
        that
        \begin{align}
            &
            \varepsilon \sum_{j \in \mathbb Z} 
            2^{\frac{3}{q}j}
            \| [u \cdot \nabla, \Delta_j]a \|_{L^1( 0,t ; L^q)}
            +
            \varepsilon
            \| (\div u)a \|_{L^1( 0,t ; \dB_{q,1}^{\frac{3}{q}-1} )}\\
            &
            +
            \varepsilon
            \left\| \| \div u \|_{L^{\infty}}\| a \|_{\dB_{q,1}^{\frac{3}{q}}} \right\|_{L^1(0,t)}
            +
            \| ( \div(au),N_{\varepsilon}[a,u] ) \|_{L^1(0,t;\dB_{2,1}^{\frac{3}{q}-1})}\\
            &\quad
            \leqslant
            C 
            \varepsilon
            \| a \|_{L^{\infty}( 0,t ; \dB_{q,1}^{\frac{3}{q}})}
            \| u \|_{L^1( 0,t ; \dB_{q,1}^{\frac{3}{q}+1})}^{h;\frac{\beta_0}{\varepsilon}}
            +
            C\| (a,u) \|_{L^2(0,t ; \dB_{q,1}^\frac{3}{q})}^2\\
            &\quad 
            \leqslant
            {
            C
            \mathcal{A}^{q,r}_{\varepsilon,{\alpha}}[a,u](t)
            \| (a,u) \|_{A^{q,r}_{\varepsilon,{\alpha}}(t)}},
        \end{align}
        which completes the proof.
    \end{proof}
    
    Finally, we are now in a position to present the proof of Theorem \ref{thm}.
    \begin{proof}
        Let $\delta$ be a positive constant {to be} determined later.
        Since $(a_0,u_0) \in ( \dB_{2,1}^{\frac{1}{2}}(\mathbb R^3) \cap \dB_{2,1}^{\frac{3}{2}}(\mathbb R^3) ) \times \dB_{2,1}^{\frac{1}{2}}(\mathbb R^3)^3$, there exists a constant $\alpha_{\delta}=\alpha_{\delta}(a_0,u_0) \geqslant 1$ such that
        \begin{align}
            \| (a_0,u_0) \|_{\dB_{2,1}^{\frac{1}{2}}}^{h;\alpha_{\delta}}
            +
            \| a_0 \|_{\dB_{2,1}^{\frac{3}{2}}}^{h;\alpha_{\delta}}
            \leqslant
            \delta.
        \end{align}
        Let $\Omega \in \mathbb{R}$ and $0<\varepsilon<1$ satisfy 
        \begin{gather}
            |\Omega|\varepsilon \leqslant 1,\qquad
            0 < \varepsilon < \beta_0/\alpha_{\delta},
            \qquad
            \langle T \rangle\varepsilon \| (a_0,u_0) \|_{D_{\varepsilon}} \leqslant \delta,\\
            \langle \alpha_{\delta}^2T \rangle^{2} 
            \left( \varepsilon^{\frac{2}{3r}} + |\Omega|^{-\frac{1}{r}}\alpha^{\frac{2}{r}}_{\delta} \right)\| (a_0,u_0) \|_{\dB_{2,1}^{\frac{1}{2}}} \leqslant \delta,
            \quad
            \langle \alpha_{\delta}^2T \rangle^{2} 
            \left( \varepsilon^{\frac{2}{3r}} + |\Omega|^{-\frac{1}{r}}\alpha^{\frac{2}{r}}_{\delta} \right) \leqslant 1.
        \end{gather}
        Let $q$ and $r$ satisfy    
        \begin{align}
            2 < q < {4}, \qquad
            2 < r < \infty, \qquad
            \frac{2}{r} \leqslant \frac{3}{q} - \frac{1}{2},\qquad
            \frac{1}{q} + \frac{1}{r} \leqslant \frac{1}{2}.
        \end{align}
        {
        In Step.1 and Step.2 below, We fix $0 < t < T_{\Omega,\varepsilon}^{\rm max}$ and assume that the local solution $(a,u)$ satisfies
        \begin{align}\label{a-small}
            \varepsilon \| a \|_{L^{\infty}(0,t;L^{\infty}\cap \dB_{q,1}^{\frac{3}{q}})}
            \leqslant 
            \frac{1}{2}.
        \end{align}}
        
        \noindent 
        {\it Step.1 A priori estimates for the energy norm.}
        For the low frequency part, it follows from Corollary \ref{cor:low-ene} that
        \begin{align}\label{low-E-1}
            \begin{split}
            &
            \| (a,u) \|_{{L^{\infty}}(0,t;\dB_{2,1}^{\frac{1}{2}}) \cap L^1(0,t;\dB_{2,1}^{\frac{5}{2}})}^{\ell;\frac{\beta_0}{\varepsilon}}\\
            &\quad
            \leqslant{}
            C\langle T \rangle 
            \| (a_0,u_0) \|_{\dB_{2,1}^{\frac{1}{2}}}^{\ell;\frac{\beta_0}{\varepsilon}}
            +
            C\langle T \rangle 
            \|( \div(au),N_{\varepsilon}[a,u] )  \|_{L^1( 0,t ; \dB_{2,1}^{\frac{1}{2}} )}^{\ell;\frac{\beta_0}{\varepsilon}}\\
            &\quad 
            \leqslant{}
            C\langle T \rangle 
            \| (a_0,u_0) \|_{\dB_{2,1}^{\frac{1}{2}}}^{\ell;\frac{\beta_0}{\varepsilon}}
            +
            {
            C\langle T \rangle 
            \mathcal{A}^{q,r}_{\varepsilon,{\alpha_{\delta}}}[a,u](t)
            \| (a,u) \|_{E_{\varepsilon}(t)}}.
            \end{split}
        \end{align}
        For the high frequency analysis, it follows from Lemma \ref{lemm:high-ene} with $f=-\div(au)$ and $g=-N_{\varepsilon}[a,u]$ that
        \begin{align}\label{high-E-1}
            \begin{split}
            &
            \varepsilon \| a \|_{{L^{\infty}}( 0,t ; \dB_{2,1}^{\frac{3}{2}} )}^{h;\frac{\beta_0}{\varepsilon}}
            +
            \frac{1}{\varepsilon}\| a \|_{L^1( 0,t ; \dB_{2,1}^{\frac{3}{2}} )}^{h;\frac{\beta_0}{\varepsilon}}
            +
            \| u \|_{{L^{\infty}}( 0,t ; \dB_{2,1}^{\frac{1}{2}} ) \cap L^1( 0,t ; \dB_{2,1}^{\frac{5}{2}} )}^{h;\frac{\beta_0}{\varepsilon}}\\
            &
            \quad 
            \leqslant
            C
            \| (\varepsilon a_0,u_0) \|_{\dB_{2,1}^{\frac{3}{2}}}^{h;\frac{\beta_0}{\varepsilon}}\\
            &
            \qquad
            +
            C\varepsilon \sum_{j \in \mathbb Z} 
            {2^{\frac{3}{2}j}}
            \| [u \cdot \nabla, \Delta_j]a \|_{L^1( 0,t ; L^2)}
            +
            C\varepsilon
            \| (\div u)a \|_{L^1(0,t ; \dB_{2,1}^{\frac{3}{2}} )}
            \\
            &\qquad
            +
            C\varepsilon
            \left\| \| \div u \|_{L^{\infty}}\| a \|_{\dB_{2,1}^{\frac{3}{2}}} \right\|_{L^1(0,t)}
            +
            C\|( \div(au),N_{\varepsilon}[a,u] ) \|_{L^1(0,t;\dB_{2,1}^{\frac{1}{2}})}^{h;\frac{\beta_0}{\varepsilon}}\\
            &
            \quad 
            \leqslant
            C\| (\varepsilon a_0,u_0) \|_{\dB_{2,1}^{\frac{3}{2}} \times \dB_{2,1}^{\frac{1}{2}}}^{h;\frac{\beta_0}{\varepsilon}}
            +
            {
            C
            \mathcal{A}^{q,r}_{\varepsilon,{\alpha_{\delta}}}[a,u](t)
            \| (a,u) \|_{E_{\varepsilon}(t)}}.
            \end{split}
        \end{align}
        Hence, gathering \eqref{low-E-1} and \eqref{high-E-1}, we have
        \begin{align}\label{a-priori-E}
            \begin{split}
            \| (a,u) \|_{E_{\varepsilon}(t)}
            \leqslant{}&
            C_1\langle T \rangle
            \| (a_0,u_0) \|_{D_{\varepsilon}}
            +
            C_1\langle T \rangle
            {
            \mathcal{A}^{q,r}_{\varepsilon,{\alpha_{\delta}}}[a,u](t)
            \| (a,u) \|_{E_{\varepsilon}(t)}}
            \end{split}
        \end{align}
        for some positive constant $C_1=C_1(\mu,P,q,r)$.
        
        \noindent 
        {\it Step.2 A priori estimates for the auxiliary norm.}
        It follows from Proposition \ref{prop:lin-str-visc} that
        \begin{align}\label{low-A-1}
            \begin{split}
            \| (a,u) \|_{{L^r}( 0,t ; \dB_{q,1}^{\frac{3}{q}-1+\frac{2}{r}} )}^{\ell;\alpha_{\delta}}
            \leqslant{}&
            C
            \langle \alpha_{\delta}^2T \rangle^{2} 
            \left( \varepsilon^{\frac{2}{3r}} + |\Omega|^{-\frac{1}{r}} \right)
            \| (a_0,u_0) \|_{\dB_{q,1}^{\frac{1}{2}}}^{\ell;\alpha_{\delta}}\\
            &+
            C
            \langle \alpha_{\delta}^2T \rangle^{2} 
            \left( \varepsilon^{\frac{2}{3r}} + |\Omega|^{-\frac{1}{r}} \right)
            \|( \div(au),N_{\varepsilon}[a,u] ) \|_{L^1(0,t ; \dB_{2,1}^{\frac{1}{2}})}^{\ell;\alpha_{\delta}}\\
            \leqslant{}&
            C
            \delta+
            C
            \|( \div(au),N_{\varepsilon}[a,u] ) \|_{L^1(0,t ; \dB_{2,1}^{\frac{1}{2}})}^{\ell;\alpha_{\delta}}\\
            \leqslant{}&
            C\delta+
            {
            C
            \mathcal{A}^{q,r}_{\varepsilon,{\alpha_{\delta}}}[a,u](t)
            \| (a,u) \|_{A^{q,r}_{\varepsilon,{\alpha_{\delta}}}(t)}}.
            \end{split}
        \end{align}
        By the Bernstein inequality and Lemma \ref{lemm:middle-ene}, it holds 
        \begin{align}\label{low-A-2}
            \begin{split}
            \| (a,u) \|_{{L^{\infty}}(0,t;\dB_{q,1}^{\frac{3}{q}-1}) \cap L^1(0,t;\dB_{q,1}^{\frac{3}{q}+1})}^{m;\alpha_{\delta},\frac{\beta_0}{\varepsilon}}
            &
            \leqslant{}
            C
            \| (a,u) \|_{{L^{\infty}}(0,t;\dB_{2,1}^{\frac{1}{2}}) \cap L^1(0,t;\dB_{2,1}^{\frac{5}{2}})}^{m;\alpha_{\delta},\frac{\beta_0}{\varepsilon}}\\
            &
            {
            \leqslant{}
            C\| (a_0,u_0) \|_{\dB_{2,1}^{\frac{1}{2}}}^{m;\alpha_{\delta},\frac{\beta_0}{\varepsilon}}}\\
            &
            {
            \qquad
            +
            C
            \| (\div(au), N_{\varepsilon}[a,u] ) \|_{L^1(0,t;\dB_{2,1}^{\frac{1}{2}})}^{m;\alpha_{\delta},\frac{\beta_0}{\varepsilon}}}\\
            & 
            \leqslant{}
            C
            \delta
            +
            C
            \| ( \div(au),N_{\varepsilon}[a,u] ) \|_{L^1(0,t ; \dB_{2,1}^{\frac{1}{2}})}^{m;\alpha_{\delta},\frac{\beta_0}{\varepsilon}}\\
            &\leqslant{}
            C\delta+
            {
            C
            \mathcal{A}^{q,r}_{\varepsilon,{\alpha_{\delta}}}[a,u](t)
            \| (a,u) \|_{A^{q,r}_{\varepsilon,{\alpha_{\delta}}}(t)}}.
            \end{split}
        \end{align}
        From Lemma \ref{lemm:high-ene},
        there holds
        \begin{align}\label{high-A-1}
            \begin{split}
            &
            \varepsilon \| a \|_{{L^{\infty}}( 0,t ; \dB_{q,1}^{\frac{3}{q}} )}^{h;\frac{\beta_0}{\varepsilon}}
            +
            \frac{1}{\varepsilon}\| a \|_{L^1( 0,t ; \dB_{q,1}^{\frac{3}{q}} )}^{h;\frac{\beta_0}{\varepsilon}}
            +
            \| u \|_{{L^{\infty}}( 0,t ; \dB_{q,1}^{\frac{3}{q}-1} ) \cap L^1( 0,t ; \dB_{q,1}^{\frac{3}{q}+1} )}^{h;\frac{\beta_0}{\varepsilon}}\\
            &
            \quad 
            \leqslant
            C
            \| (\varepsilon a_0,u_0) \|_{\dB_{q,1}^{\frac{3}{q}} \times \dB_{q,1}^{\frac{3}{q}-1}}^{h;\frac{\beta_0}{\varepsilon}}
            \\
            &\qquad
            +
            C\varepsilon \sum_{j \in \mathbb Z} 
            {2^{\frac{3}{q}j}}
            \| [u \cdot \nabla, \Delta_j]a \|_{L^1( 0,t ; L^q)}
            +
            C\varepsilon
            \| (\div u)a \|_{L^1( 0,t ; \dB_{q,1}^{\frac{3}{q}} )}\\
            &\qquad
            +
            C\varepsilon
            \left\| \| \div u \|_{L^{\infty}}\| a \|_{\dB_{q,1}^{\frac{3}{q}}} \right\|_{L^1(0,t)}
            +
            C\| ( \div(au),N_{\varepsilon}[a,u] ) \|_{L^1(0,t;\dB_{q,1}^{\frac{3}{q}-1})}^{h;\frac{\beta_0}{\varepsilon}}\\
            &
            \quad 
            \leqslant
            C\delta 
            + 
            {
            C
            \mathcal{A}^{q,r}_{\varepsilon,{\alpha_{\delta}}}[a,u](t)
            \| (a,u) \|_{A^{q,r}_{\varepsilon,{\alpha_{\delta}}}(t)}}.
            \end{split}
        \end{align}
        Hence, combining {\eqref{low-A-1}, \eqref{low-A-2}, and} \eqref{high-A-1}, we obtain 
        \begin{align}\label{a-priori-A}
            \begin{split}
            \| (a,u) \|_{A_{\varepsilon,\alpha_{\delta}}^{q,r}(t)}
            \leqslant{}&
            C_2\delta 
            +
            C_2
            {
            \mathcal{A}^{q,r}_{\varepsilon,{\alpha_{\delta}}}[a,u](t)
            \| (a,u) \|_{A^{q,r}_{\varepsilon,{\alpha_{\delta}}}(t)}}
            \end{split}
        \end{align}
        for some positive constant $C_2=C_2(\mu,P,q,r)$.
        
        \noindent 
        {\it Step.3 The continuation argument.}
        Let $C_0:=\max\{C_1,C_2\}$. 
        Let $T_{\Omega,\varepsilon}^{\rm max}$ be the maximal existence time and let 
        \begin{align}
            T_{\Omega,\varepsilon}^*
            :=
            \sup
            \left\{
            t \in (0,T_{\Omega,\varepsilon}^{\rm max})
            \ ; \ 
            \begin{aligned}
                &
                \| (a,u) \|_{E_{\varepsilon}(t)}
                \leqslant
                3C_0
                \langle T \rangle
                \| (a_0,u_0) \|_{D_{\varepsilon}},\\
                &
                \| (a,u) \|_{A^{q,r}_{\varepsilon,\alpha_{\delta}}(t)}
                \leqslant
                3C_0
                \delta 
            \end{aligned}
            \right\}.
        \end{align}
        Note that Proposition \ref{prop:LWP} implies that $T_{\Omega,\varepsilon}^*>0$.
        Then, it suffices to show $T_{\Omega,\varepsilon}^* \geqslant T$.
        Supoose by contradiction that $T_{\Omega,\varepsilon}^*  < T$.
        {
        It follows from the embedding $\dB_{q,1}^{\frac{3}{q}}(\mathbb{R}^3) \hookrightarrow L^{\infty}(\mathbb{R}^3)$ and the Bernstein inequality that
        \begin{align}\label{a-small-2}
            \varepsilon
            \| a \|_{L^{\infty}(0,t; L^{\infty} \cap \dB_{q,1}^{\frac{3}{q}})}
            \leqslant
            C_4
            \mathcal{A}_{\varepsilon,\alpha_{\delta}}^{q,r}[a,u](t)
        \end{align}
        for all $0 < t < T_{\Omega,\varepsilon}^*$ with some positive constant $C_4=C_4(q,r)$.
        We now choose $\delta$ so small that 
        \begin{gather}
            \delta
            \leqslant
            \min \left\{
            \frac{1}{36C_0^3},\frac{1}{24C_0^2C_4}\right\},\label{small-1}\\
            \left( 
            \langle T \rangle
            \| (a_0,u_0) \|_{(\dB_{2,1}^{\frac{3}{2}} \cap \dB_{2,1}^{\frac{1}{2}}) \times \dB_{2,1}^{\frac{1}{2}}}
            \right)^{\frac{r-2}{r-1}}
            \delta^{\frac{1}{r-1}}
            \leqslant
            \min\left\{\frac{1}{18C_0^3},
            \frac{1}{12C_0^2C_4}
            \right\}.\label{small-2}
        \end{gather}
        Then, we see that for any $0<t<T_{\Omega,\varepsilon}^*$,
        \begin{align}\label{cal_A}
            \mathcal{A}_{\varepsilon,\alpha_{\delta}}^{q,r}[a,u](t)
            \leqslant{}&
            3C_0
            \left\{
            \varepsilon
            \alpha_{\delta}
            \langle T \rangle
            \| (a_0,u_0) \|_{D_{\varepsilon}}
            +\delta
            +
            (\langle T \rangle
            \| (a_0,u_0) \|_{D_{\varepsilon}})^{\frac{r-2}{r-1}}
            \delta^{\frac{1}{r-1}}
            \right\}\\
            \leqslant{}&
            6C_0^2
            \delta
            +
            3C_0^2
            \left( 
            \langle T \rangle
            \| (a_0,u_0) \|_{(\dB_{2,1}^{\frac{3}{2}} \cap \dB_{2,1}^{\frac{1}{2}}) \times \dB_{2,1}^{\frac{1}{2}}}
            \right)^{\frac{r-2}{r-1}}
            \delta^{\frac{1}{r-1}}\\
            \leqslant{}&
            \min\left\{
            \frac{1}{3C_0},
            \frac{1}{2C_4}
            \right\}.
        \end{align}
        Thus, we see by \eqref{a-priori-E}, \eqref{a-priori-A}, \eqref{cal_A}, and \eqref{a-small-2} that $(a,u)$  satisfies \eqref{a-small} and
        \begin{align}
            &
            \begin{aligned}
            \| (a,u) \|_{E_{\varepsilon}(t)}
            \leqslant{}&
            C_0
            \langle T \rangle
            \| (a_0,u_0) \|_{D_{\varepsilon}}
            + 
            C_0\langle T \rangle
            \mathcal{A}_{\varepsilon,\alpha_{\delta}}^{q,r}[a,u](t)
            \| (a,u) \|_{E_{\varepsilon}(t)}\\
            \leqslant{}&
            2C_0\langle T \rangle\| (a_0,u_0) \|_{D_{\varepsilon}},    
            \end{aligned}\\
            &
            \begin{aligned}
            \| (a,u) \|_{A_{\varepsilon,\alpha_{\delta}}^{q,r}(t)}
            \leqslant{}&
            C_0\delta 
            +
            C_0\mathcal{A}_{\varepsilon,\alpha_{\delta}}^{q,r}[a,u](t)\| (a,u) \|_{A_{\varepsilon,\alpha_{\delta}}^{q,r}(t)}\\
            \leqslant{}&
            2C_0\delta
            \end{aligned}
        \end{align}
        for all $0<t<T_{\Omega,\varepsilon}^*$,}
        which leads a contradiction to the definition of $T_{\Omega,\varepsilon}^*$.
        Hence, we obtain $T_{\Omega,\varepsilon}^* \geqslant T$ and complete the proof.
    \end{proof}

\noindent
{\bf Data availability}
Data sharing not applicable to this article as no datasets were generated or analysed during the current study.

\noindent
{\bf Conflict of interest statement.}\\
The author declares no conflicts of interest.

\noindent
{\bf Acknowledgements.} \\
The first author was partly supported by Grant-in-Aid for JSPS Research Fellow, Grant Number JP20J20941.
The second author was partly supported by JSPS KAKENHI, Grant Number {JP}21K13826.

\appendix
\def\thesection{\Alph{section}}
\section{Proof of Lemma \ref{lemm:stab}}\label{ses:app}
    \begin{proof}[Proof of Lemma \ref{lemm:stab}]
        Without loss of generality, we may assume that $x=0$.
        Indeed, once we obtain the estimate \eqref{osc:est} with $x=0$, we immediately obtain \eqref{osc:est} for all $x \in \mathbb{R}^d$ by {replacing} $p(\xi)$ {by} $p(\xi)+(x\cdot \xi)/\tau$.
        Here, we note that the constants $\varepsilon$ and $C$ in Lemma \ref{lemm:stab} are invariant under this replacement as it holds $\nabla^{\alpha}_{\xi} p(\xi)= \nabla^{\alpha}_{\xi}\{p(\xi)+(x\cdot \xi)/\tau\}$ for all $x$, $\xi$, $\tau$, and multi-index $\alpha$ with $|\alpha| \geqslant 2$.
        
        We first consider the non-degenerate case $k=d$.
        Using the partition of unity on the compact set $\supp \varphi$, we may assume that $\supp \varphi$ is a subset of some open ball with the radius $\delta$, which is a positive constant to be determined later.
        Then, we see that
        \begin{align}\label{J-0}
            \left|
            \int_{\mathbb{R}^d}
        		e^{i\tau q(\xi)}\varphi(\xi) d\xi
            \right|^2
            =
            \int_{|\eta|\leqslant 2\delta}
            J(\tau;\eta) d\eta
        \end{align}   
        {with}
        \begin{align}
            J(\tau;\eta)
            :=
            \int_{\mathbb{R}^d}
            e^{i\tau \{q(\xi+\eta)-q(\xi)\}}\widetilde{\varphi}(\xi,\eta)d\xi,
        \end{align}
        where we have set $\widetilde{\varphi}(\xi,\eta):=\varphi(\xi+\eta)\overline{\varphi(\xi)}$.
        Let $N \in \{0,1,2,...,d,d+1\}$.
        Then, from the $N$ times of the integration by parts
        and the direct computation via the chain rule, it follows that
        \begin{equation}\label{J-1}
        \begin{split}
            |J(\tau;\eta)|
            \leqslant{}
            &
            C_{N,d}|\tau|^{-N}\sum_{|\beta| \leqslant N}
            \int_{\mathbb{R}^d}
            |\nabla_{\xi}\{q(\xi+\eta)-q(\xi)\}|^{|\beta|-2N}\\
            &
            \times
            \sum_{1\leqslant|\gamma|\leqslant N+1}
            |\nabla_{\xi}^{\gamma}\{q(\xi+\eta)-q(\xi)\}|^{N-|\beta|}
            |\nabla_{\xi}\widetilde{\varphi}(\xi,\eta)|d\xi,
        \end{split}
        \end{equation}
        where the positive constant $C_{N,d}$ depends only on $N$ and $d$.
        By the mean value {theorem} and $\| q-p \|_{W^{d+3,\infty}(\supp\varphi)} \leqslant 1${, we have}
        \begin{align}\label{J-2}
        \begin{split}
            |\nabla_{\xi}^{\gamma}\{q(\xi+\eta)-q(\xi)\}|^{N-|\beta|}
            &\leqslant
            \| \nabla_{\xi}^{|\gamma|+1}q\|_{L^{\infty}(\supp\varphi)}|\eta|\\
            &\leqslant
            (1+\| \nabla_{\xi}^{|\gamma|+1}p\|_{L^{\infty}(\supp\varphi)})
            |\eta|.
        \end{split}
        \end{align}
        Since $k=d$, we see that $\det(\nabla^2p(\xi))\neq 0${, and hence}
        $\nabla^2 p(\xi)$ is invertible.
        Thus, we see that
        \begin{align}
            |\nabla^2 p(\xi)\eta|
            \geqslant
            |(\nabla^2 p(\xi))^{-1}|^{-1}|\eta|
            \geqslant
            \frac{\displaystyle \inf_{\xi \in \supp\varphi}|\det \nabla^2 p(\xi)|}{d^2(d-1)!\| \nabla^2 p\|_{L^{\infty}(\supp\varphi)}^{d-1}}|\eta|
            =:c|\eta|.
        \end{align}
        Then, it follows from the Taylor {theorem} that
        \begin{align*}
            &
            |\nabla_{\xi}\{q(\xi+\eta)-q(\xi)\}|\\
            &\quad=
            \left|
            \nabla_{\xi}^2q(\xi)\eta
            +
            \sum_{l,m=1}^d 
            \int_0^1(1-\theta)(\nabla\partial_{\xi_l}\partial_{\xi_m}q)(\xi+\theta\eta)\eta_l \eta_m d\theta
            \right|\\
            &\quad\geqslant
            |\nabla_{\xi}^2q(\xi)\eta|
            -\frac{d^2}{2}\| \nabla^3q\|_{L^{\infty}(\supp \varphi)}|\eta|^2\\
            &\quad\geqslant
            |\nabla_{\xi}^2p(\xi)\eta|
            -\|\nabla^2(q-p)\|_{L^{\infty}(\supp \varphi)}|\eta|
            -\frac{d^2}{2}(1+\| \nabla^3p\|_{L^{\infty}(\supp \varphi)})|\eta|^2\\
            &\quad\geqslant
            \left\{
            c
            -\varepsilon
            -d^2(1+\| \nabla^3p\|_{L^{\infty}(\supp \varphi)})\delta
            \right\}|\eta|.
        \end{align*}
        Therefore, taking $\varepsilon$ and $\delta$ so small that $\varepsilon \leqslant c/4$ and $d^2(1+\| \nabla^3p\|_{L^{\infty}(\supp \varphi)})\delta \leqslant c/4$,
        we have
        \begin{align}\label{J-3}
            |\nabla_{\xi}\{q(\xi+\eta)-q(\xi)\}|
            \geqslant \frac{c}{2}|\eta|.
        \end{align}
        Combining \eqref{J-0}, \eqref{J-1}, \eqref{J-2}{,} and \eqref{J-3},
        we obtain $|J(\tau,\eta)|\leqslant C_N|\tau \eta|^{-N}$.
        Choosing $N=0$ and $N=d+1$, we have
        \begin{align}
            \left|
            \int_{\mathbb{R}^d}
        		e^{i\tau q(\xi)}\varphi(\xi) d\xi
            \right|^2
            \leqslant
            C
            \int_{\mathbb{R}^d}
            \frac{1}{(1+|\tau \eta|)^{d+1}}d\eta
            \leqslant
            C|\tau|^{-d},
        \end{align}
        which completes {the proof for} the case $k=d$.
        
        Next, we consider the {degenarate} case $1\leqslant k\leqslant {d-1}$.
        By virture of the partition of unity on $\supp \varphi$, we may assume that $\supp \varphi$ is included in some open ball centered at $\xi_0 \in \supp \varphi$ {with} the radius $r$. Here, $r$ is a positive constant to be determined later.
        Let $k_0:=\rank \nabla^2p(\xi_0) \geqslant k$.
        It is easy to see that there {exist} $\lambda_1, ..., \lambda_{k_0} \in\mathbb{R} \setminus \{0\}$ and {an} orthogonal matrix $P$ such that $P^{\top}\nabla^2p(\xi_0)P=\operatorname{diag}\{\lambda_1,...,\lambda_{k_0},0,...,0\}$.
        Here, $\operatorname{diag}\{a_1,...,a_d\}$ denotes the diagonal matrix with its diagonal components $a_1,...,a_d$.
        Put $\eta_0:=P^{\top}\xi_0$ and $\widetilde{p}:=p\circ P$.
        Then, since
        $\nabla^2\widetilde{p}(\eta_0)=P^{\top}\nabla^2p(\xi_0)P$ and
        \begin{align}
            \det\{\partial_{\eta_m}\partial_{\eta_l}\widetilde{p}(\eta)\}_{1\leqslant m,l\leqslant k_0}|_{\eta=\eta_0}
            =
            \lambda_1\cdots\lambda_{k_0}
            \neq 0,
        \end{align}
        we may choose $r$ so that
        \begin{align}
            |\det\{\partial_{\eta_m}\partial_{\eta_l}\widetilde{p}(\eta)\}_{1\leqslant m,l\leqslant k_0}| \geqslant |\lambda_1\cdots\lambda_{k_0}|/2>0
        \end{align}
        for all $\eta$ with $|\eta-\eta_0|\leqslant r$ .
        Let $\varepsilon$ satisfy
        \begin{align}
            \varepsilon
            \leqslant
            \frac{|\lambda_1\cdots\lambda_{k_0}|}
            {8k_0^2(k_0-1)!\| \{\partial_{\eta_m}\partial_{\eta_l}\widetilde{p}(\eta)\}_{1\leqslant m,l\leqslant k_0}\|_{L^{\infty}(|\eta-\eta_0|\leqslant r)}^{k_0-1}}
        \end{align}
        and let $q$ satisfy $\| q-p\|_{W^{d+3,\infty}(\supp \varphi)}
            \leqslant C_*^{-1}\varepsilon$ where $C_*$ is a positive constant satisfying
        \begin{align}
            \sup_{|\eta''-\eta_0''|\leqslant r}
            \|\widetilde{q}(\eta',\eta'')-\widetilde{p}(\eta',\eta'')\|_{W^{k_0+3,\infty}_{\eta'}(|\eta'-\eta_0'|\leqslant r)}
            \leqslant C_*\| q-p\|_{W^{d+3,\infty}(\supp \varphi)}
        \end{align}
        {with $\widetilde{q}:=q \circ P$}.
        Here, we note that the constant $C_*$ is independent of $q$.
        Then, it follows from the result of {the non-degenerate case} {$(k = d)$} that
        \begin{align}
            \left|
            \int_{\mathbb{R}^d}
        	e^{i\tau q(\xi)}\varphi(\xi) d\xi
        	\right|
        	&=
        	\left|
        	\int_{\mathbb{R}^d}e^{i\tau \widetilde{q}(\eta)}\varphi(P\eta)d\eta
        	\right|\\
        	&\leqslant
        	\int_{\substack{\eta''\in \mathbb{R}^{d-k_0} \\|\eta''-\eta_0''|\leqslant r}}
        	\left|
        	\int_{\eta'\in\mathbb{R}^{k_0}}
        	e^{i\tau \widetilde{q}(\eta',\eta'')}\varphi(P(\eta',\eta''))d\eta'
        	\right|
        	d\eta''\\
        	&\leqslant
        	\int_{\substack{\eta''\in \mathbb{R}^{d-k_0} \\|\eta''-\eta_0''|\leqslant r}}
        	C(\eta'')(1+|\tau|)^{-\frac{k_0}{2}}
        	d\eta''\\
        	&\leqslant
        	C(1+|\tau|)^{-\frac{k_0}{2}}
        	\leqslant
        	C(1+|\tau|)^{-\frac{k}{2}},
        \end{align}
        where $C(\eta'')=C(\eta'',d,\varphi,p)$ is a constant continuously depending on $\eta''$.
        The proof is complete.
    \end{proof}

\end{document}